\RequirePackage{fix-cm}
\documentclass[smallextended]{svjour3}       % onecolumn (second format)
\smartqed  % flush right qed marks, e.g. at end of proof
\usepackage{graphicx}
\newcommand{\ud}{\,\mathrm{d}}
\newcommand{\udiv}{\, \mathrm{div}\,}
\newcommand{\eps}{\varepsilon}
\newcommand{\R}{{\mathbb{R}}}
\newcommand{\N}{{\mathbb{N}}}
\newcommand{\ind}{{\mathbb I}}

\usepackage{amsmath}
\usepackage{amssymb}
\begin{document}

\title{Quantitative estimates of propagation of chaos for stochastic systems with $W^{-1,\infty}$ kernels \thanks{P.E. Jabin is partially supported by NSF Grant 1312142 and by NSF Grant RNMS (Ki-Net) 1107444. Z. Wang is partially supported by NSF Grant 1312142 and Ann G. Wylie Dissertation Fellowship}
}
%\subtitle{Do you have a subtitle?\\ If so, write it here}

\titlerunning{Quantitative estimates of propagation of chaos}        % if too long for running head

\author{Pierre-Emmanuel Jabin  \and
        Zhenfu Wang %etc.
}

%\authorrunning{Short form of author list} % if too long for running head

\institute{P.E. Jabin \at
              CSCAMM and Dept. of Mathematics, University of Maryland,
College Park, MD 20742, USA. \\
              %Tel.: +1 301-405-1647 \\
              %Fax: +123-45-678910\\
              \email{pjabin@cscamm.umd.edu}           %  \\
%             \emph{Present address:} of F. Author  %  if needed 
           \and
           Z. Wang \at
              CSCAMM and Dept. of Mathematics, University of Maryland,
College Park, MD 20742, USA. \\
\email{zwang423@math.umd.edu} \\
\emph{Present address:} Department of Mathematics, University of Pennsylvania, Philadelphia, PA 19104, USA. 
}

\date{Received: date / Accepted: date}
% The correct dates will be entered by the editor

\maketitle 

\begin{abstract}
%Insert your abstract here. Include keywords, PACS and mathematical
%subject classification numbers as needed.
We derive quantitative estimates proving the propagation of chaos for large stochastic systems of interacting particles. We obtain explicit bounds on the relative entropy between the joint law of the particles and the tensorized law at the limit. We have to develop for this new laws of large numbers at the exponential scale. But our result only requires very weak regularity on the interaction kernel in the negative Sobolev space $\dot W^{-1,\infty}$, thus including the Biot-Savart law and the point vortices dynamics for the 2d incompressible Navier-Stokes.

% \PACS{PACS code1 \and PACS code2 \and more}
\end{abstract}

%%%%%%%%%%%%%%%%%%%%%%%%%%%%%%%%%%%%%%%%%%%%%%%%%%%%%%%%%%%%%%%%%%%%%%%%%%%%%%%
\section{Introduction} 
%%%%%%%%%%%%%%%%%%%%%%%%%%%%%%%%%%%%%%%%%%
\subsection{Motivation}
%%%%%%%%%%%%%%%%%%%%%%%%%%%%%%%%%%%%%%%%%%
We consider large systems   of  $N$ indistinguishable point-particles given by the coupled  stochastic differential equations (SDEs) 
\begin{equation}
\label{SDE1st}
\qquad \ud X_i = F(X_i) \ud t+\frac{1}{N} \sum_{j \ne i} K (X_i -X_j) \ud t + \sqrt{2 \sigma_N} \ud W_t^i,  \quad i=1, \cdots, N
\end{equation}
where for simplicity $X_i\in \Pi^d$, the $d$-dimensional torus, the $W^i$ are $N$ independent standard Wiener Processes (Brownian motions) in $\mathbb{R}^d$ and the stochastic term in \eqref{SDE1st} should be understood in the It\^{o} sense. 

The interaction term is normalized by the factor $1/N$, corresponding to the mean field scaling. For a fixed $N$ our goal is hence to derive explicit, quantitative estimates comparing System \eqref{SDE1st} to the {\em mean field limit} $\bar\rho$ solving
\begin{equation}
\partial_t\bar\rho+\udiv_x (\bar\rho\,[ F +K\star_x\bar\rho ] )=\sigma\,\Delta\bar\rho. \label{limitmeanfield}
\end{equation}
Such estimates in particular imply the propagation of chaos in the limit $N\rightarrow \infty$. But precisely because they are quantitative, they also characterize the reduction of complexity of System \eqref{SDE1st} for large and finite $N$.  

\medskip

A guiding motivation of interaction kernel $K$ in our work is given by the Biot-Savart law in dimension $2$, namely
\begin{equation}
K(x)=\alpha\,\frac{x^\perp}{|x|^2}+K_0(x),\label{biot-savart}
\end{equation}
where $x^\perp$ denotes the rotation of vector $x$ by $\pi/2$ and where $K_0$ is a smooth correction to periodize $K$ on the torus represented by $[-1/2,\;1/2]^d$. If $\omega(x)\in L^p(\Pi^d)$ with $p\geq 1$, then $u=K\star_x \omega$ solves 
\[
\mbox{curl}\,u=\mbox{curl}\, K\star_x \omega=\alpha\,\Big(\omega-\int_{\Pi^d} \omega\Big),\quad \udiv u=\udiv K\star \omega=0.
\]
If $F=0$, the limiting equation \eqref{limitmeanfield} becomes
\begin{equation}
\partial_t \omega+K\star_x \omega\cdot\nabla_x \omega=\sigma\,\Delta \omega,\label{vorticity}
\end{equation}
where we now write  on $\omega(t,x)$,  using the classical notation for the vorticity of a fluid.  Eq. \eqref{vorticity} is invariant by the addition of a constant $\omega\rightarrow \omega+C$. We may hence assume that $\int_{\Pi^d} \omega=0$ and Eq. \eqref{vorticity} is then equivalent to the 2d incompressible Navier-Stokes system on $u(t,x)$ s.t. $\omega=\mbox{curl}\,u$,
\begin{equation}
  \begin{split}
    &\partial_t u+u\cdot\nabla_x u=\nabla_x p+\sigma\,\Delta u,\\
    &\udiv u=0.
  \end{split}\label{2dNS}
\end{equation}
The system of particles \eqref{SDE1st} now corresponds to a system of interacting point vortices with additive noise. Because we present our method  in the simplest framework where particles are indistinguishable,  all point vortices necessarily have the same vorticity in this setting. 

\medskip

Our main results provide an {\em explicit estimate quantifying that the system \eqref{SDE1st} is within $O(N^{-1/2})$ from the limit \eqref{limitmeanfield}} in an appropriate statistical sense. This applies to 
\begin{itemize}
\item If the diffusion is non-vanishing, $\sigma_N\to\sigma>0$, to all kernels $K\in W^{-1,\infty}$ with $\udiv K\in W^{-1,\infty}$, see Theorem \ref{maintheorem}  in subsection \ref{non-vanish}. We devote subsection \ref{secapplications}
  to a long discussion of various examples of kernels $K$ that are covered by Theorem \ref{maintheorem} but emphasize here that {\em it applies to the Biot-Savart law} \eqref{biot-savart} and to any kernel $K$ s.t. $ |x|\,K\in L^\infty$ and  $\udiv K\in W^{-1,\infty}$.
  \item If the diffusion vanishes (or is degenerate in some directions), $\sigma_N\rightarrow \sigma$ with $\sigma\not>0$, we can handle any kernels $K\in L^\infty$ with $\udiv K\in L^{\infty}$. Moreover if the kernel is anti-symmetric, $K(-x)=-K(x)$ (which is the case for \eqref{biot-savart}), then we only need $ |x|\,K\in L^\infty$ with $\udiv K\in L^{\infty}$. The corresponding Theorem \ref{thmvanishingviscosity} is presented in subsection \ref{vanishingdiffusion}. 
\end{itemize}
We are therefore able to handle the Biot-Savart law independently of the viscosity. But we should note that Theorem \ref{maintheorem} applies to much more general kernels.

The key argument in our proof is given by Theorem \ref{MEII}, a new large deviation estimate which bounds an appropriate partition function
\[
\sup_N \int_{\Pi^{d\,N}} \exp\Big(\frac{1}{N}\sum_{i,j=1}^N \phi(x_i,x_j)\Big)\,\Pi_{i=1}^N \bar\rho( \ud x_i) <\infty,
\]
for a modified potential $\phi$ which is related to $K$ and $\bar\rho$ but is not the potential of the dynamics. The critical point is that such an estimate holds even if $\phi$ is not continuous, but only exponentially integrable with appropriate cancellations. 

The rest of the article is organized as follows: The last subsection in the introduction sketches the proof of our basic {\em a priori} estimates.
Section \ref{proofmain} presents the proof of our main results, assuming that one has two critical estimates, Theorems \ref{MEI} (law of large numbers at exponential scale) and \ref{MEII} (large deviation estimate mentioned above). We establish some preliminary combinatorics notations in section \ref{PreCombi}. This enables us to easily prove Theorem \ref{MEI} in section \ref{proofMEI}. The proof of Theorem \ref{MEII} is considerably more difficult; it is performed in section \ref{ProofMEII} which is the main technical contribution of this article.

%
%%%%%%%%%%%%%%%%%%%%%%%%%%%%%%%%%%%%%%%%%%%%%%
\subsection{Main results for non-vanishing diffusion\label{non-vanish}}
%%%%%%%%%%%%%%%%%%%%%%%%%%%%%%%%%%%%%%%%%%%%%%%

We start by recalling the precise definition of the space $\dot W^{-1,\infty}(\Pi^d)$ which is used both in Prop. \ref{entropybounds} and in Theorem \ref{maintheorem} and which is critical to our applications.
\begin{definition}
A function $f$ with $\int_{\Pi^d} f=0$ belongs to $\dot W^{-1,\infty}(\Pi^d)$ iff there exists a vector field $g$ in $L^\infty(\Pi^d)$ s.t. $f=\udiv g$. Similarly a vector field $K$ with $\int_{\Pi^d} K=0$ belongs to $\dot W^{-1,\infty}(\Pi^d)$ iff there exists a matrix field $V$ in $L^\infty(\Pi^d)$ s.t. $K=\udiv V$ or $K_\alpha=\sum_\beta \,\partial_\beta V_{\alpha\beta}$. We then denote  
\[
\|f\|_{\dot{W}^{-1, \infty}} = \inf_{g} \|g\|_{L^\infty}, \quad \text{with}\ \  f = \udiv g,
\]
and similarly 
\[
\|K\|_{\dot W^{-1, \infty}} = \inf_V \|V\|_{L^\infty}, \quad \text{with} \ \ K = \udiv V. 
\]
  \end{definition}

\medskip

Following the basic approach introduced in \cite{JW1}, our main idea  is to use relative entropy methods to compare the {\em coupled law} $\rho_N(t,x_1,\ldots,x_N)$ of the whole system \eqref{SDE1st} to the tensorized law
\[
\bar\rho_N(t,x_1,\ldots, x_N)=\bar\rho^{\otimes^N}=\Pi_{i=1}^N \bar\rho(t,x_i),
\]   
consisting of $N$ independent copies of a process following the law $\bar\rho$, solution to the limiting equation \eqref{limitmeanfield}. 

As our estimates carry over $\rho_N$, we do not consider directly the system of SDEs \eqref{SDE1st} but instead work at the level of the Liouville equation 
\begin{equation}
\partial_t\rho_N+\sum_{i=1}^N \udiv_{x_i}\bigg(\rho_N\,\Big(F(x_i)+ \frac{1}{N}\,\sum_{j=1}^N K(x_i-x_j)\Big) \bigg)=\sum_{i=1}^N \sigma_N\,\Delta_{x_i} \rho_N,\label{Liouville}
\end{equation}
where and hereafter we use the convention that $K(0)=0$. The law $\rho_N$ encompasses all the statistical information about the system. Given that it is set in $\Pi^{d \,N}$ with $N>>1$, the observable statistical information is typically contained in the marginals
\begin{equation}
\rho_{N,k}(t,x_1,\ldots,x_k)=\int_{\Pi^{d\,(N-k)}}\rho_N(t,x_1,\ldots, x_N)\,\ud x_{k+1}\ldots\,\ud x_{N}. \label{defmarginal}
\end{equation}
Our final goal is to obtain explicit bounds on $\rho_{N,k}-\bar\rho^{\otimes^k}$, where $\bar \rho^{\otimes^k} = \Pi_{i=1}^k \bar \rho(t, x_i).$  Those bounds will follow from a relative entropy estimate between $\bar\rho_N$ and a solution $\rho_N$ to \eqref{Liouville}.
But for this, we cannot use any weak solution to the Liouville \eqref{Liouville} and instead require
\begin{definition}[Entropy solution] A density $\rho_N\in L^1(\Pi^{d\,N})$, with $\rho_N\geq 0$ and $\int_{\Pi^{d\,N}} \rho_N\,\ud X^N=1$, is an entropy solution to Eq. \eqref{Liouville} on the time interval $[0,\ T]$, iff $\rho_N$ solves \eqref{Liouville} in the sense of distributions, and for $a.e.\ t\leq T$
\begin{equation}\begin{split}
&\int_{\Pi^{d\,N}}\rho_N(t,X^N)\,\log \rho_N(t,X^N)\,\ud X^N+\sigma_N\,\sum_{i=1}^N\int_0^t\int_{\Pi^{d\,N}}\frac{|\nabla_{x_i}\rho_N|^2}{\rho_N}\,\ud X^N\,\ud s\\
&\quad\leq \int_{\Pi^{d\,N}}\rho_N^0\,\log \rho_N^0\,\ud X^N\\
&\qquad-\frac{1}{N}\sum_{i,\;j=1}^N\int_0^t\int_{\Pi^{d\,N}}(\udiv F(x_i)+\udiv K(x_i-x_j))\,\rho_N\,\ud X^N\,\ud s,
\end{split}\label{entropydissip}
\end{equation}\label{entropysol}
where for convenience we use in the article the notation $X^N=(x_1, \cdots, x_N)$. 
\end{definition}
In general it can be difficult to obtain the well posedness of an advection-diffusion equation such as \eqref{Liouville} under very weak regularity of the advection field $K$, such as is our case here. We refer to \cite{Fig} for an example of such study.

In our case though, we do not need the well posedness and  
it is in fact straightforward to check that there exists at least one entropy solution to \eqref{Liouville}.
\begin{proposition}\label{entropybounds}
Assume that $\int_{\Pi^{d\,N}}\rho_N^0\,\log \rho_N^0<0$, $\sigma_N\geq \underline{\sigma}>0$, and that $F,\;\udiv F\in L^\infty$. Assume finally that $K\in \dot W^{-1,\infty}$ with as well $\udiv K\in \dot W^{-1,\infty}$.  Then there exists an entropy solution $\rho_N$ satisfying 
\begin{equation}
\begin{split}
&\int_{\Pi^{d\,N}}\rho_N(t,X^N)\,\log \rho_N(t,X^N)\,\ud X^N+\frac{\sigma_N}{2}\,\sum_{i=1}^N\int_0^t\int_{\Pi^{d\,N}}\frac{|\nabla_{x_i}\rho_N|^2}{\rho_N}\,\ud X^N\,\ud s\\
&\quad\leq \int_{\Pi^{d\,N}}\rho_N^0\,\log \rho_N^0\,\ud X^N+\frac{N\,t\,\|\udiv K\|_{\dot W^{-1,\infty}}^2}{2\,\underline{\sigma}}+N\,t\,\|\udiv F\|_{L^\infty}.
\end{split}\label{estentropybounds}
\end{equation}
Moreover for any $\phi\in L^2([0,\ T],\ W^{1,\infty}(\Pi^{2d}))$ with $\|\phi\|_{L^2_tW^{1,\infty}_x}\leq 1$
\begin{equation}\begin{split}
    &\frac{1}{\|K\|_{\dot W^{-1,\infty}}}\,\int_0^t \int_{\Pi^{2d}} \phi(s,x_1,x_2)\,K(x_1-x_2)\,\rho_{N,2}(s,x_1,x_2) \ud x_1 \ud x_2\, \ud s%+\frac{1}{C_\gamma}\,\int_0^t\int_{\Pi^{2\,d}} \frac{\rho_{N,2}}{|x_1-x_2|^\gamma}\,\ud x_1\, \ud x_1\,\ud s
    \\
& \leq 1+  t+ \frac{2}{N\,\underline{\sigma}}\int_{\Pi^{d\,N}}\rho_N^0\,\log \rho_N^0\,\ud X^N+\frac{t\,\|\udiv K\|_{\dot W^{-1,\infty}}^2}{\underline{\sigma}^2}+t\,\frac{ 2 \|\udiv F\|_{L^\infty}}{\underline\sigma},
\end{split}
\label{boundKrhoN}
\end{equation}
 so that the product $K\,\rho_N$ is well defined. 
\end{proposition}
Our method revolves around the control of the rescaled relative entropy
\begin{equation}
{\cal H}_N(\rho_N\;|\;\bar\rho_N)(t)=\frac{1}{N}\int_{\Pi^{d\,N}} \rho_N(t,X^N)\,\log \frac{\rho_N(t,X^N)}{\bar\rho_N(t,X^N)}\,\ud X^N, \label{scaled entropy}
\end{equation}
while our main result is the explicit estimate
\begin{theorem}
  Assume that $\udiv F\in L^\infty(\Pi^d)$, that $K\in \dot W^{-1,\infty}(\Pi^d)$ with  $\udiv K\in \dot W^{-1,\infty}$. Assume that $\sigma_N \geq \underline{\sigma} >0$.  Assume moreover that $\rho_N$ is an entropy solution to Eq. \eqref{Liouville} as per Def. \ref{entropysol}. Assume finally that $\bar\rho\in L^\infty([0,\ T],\ W^{2,p}(\Pi^d))$ for any $p<\infty$ solves Eq. \eqref{limitmeanfield} with $\inf \bar\rho>0$ and $\int_{\Pi^d} \bar\rho=1$.
Then
\[\begin{split}
 {\cal H}_N(\rho_N\,|\;\bar\rho_N)(t)\leq &e^{\bar M\,(\|K\|+\|K\|^2)\,t}\,
\bigg({\cal H}_N(\rho_N^0\,|\;\bar\rho_N^0)+\frac{1}{N}\\
&\qquad+\bar M(1+t\,(1+\|K\|^2))\,|\sigma-\sigma_N| \bigg),
\end{split}
\]
where we denote $\|K\|=\|K\|_{\dot W^{-1,\infty}}+\|\udiv K\|_{\dot W^{-1,\infty}}$ and $\bar M$ is a constant which only depends on
\[\begin{split}
\bar M\bigg(&d,\;\underline{\sigma},\;\inf \bar\rho,\;\|\bar\rho\|_{W^{1,\infty}},\;\sup_{p \geq 1} \frac{\|\nabla^2 \bar\rho\|_{L^p}}{p},\;\frac{1}{N}\int_{\Pi^{d\,N}} \rho_N^0\,\log \rho_N^0,\;\|\udiv F\|_{L^\infty}\Bigg).\\
%&\frac{1}{N}\int_{\Pi^{d\,N}}\rho_N^0\,\log\rho_N^0\bigg). 
\end{split}
\]
\label{maintheorem}
\end{theorem}

\begin{remark} The regularity assumptions for the limit $\bar \rho$ on the time interval $[0, T]$
can be established by propagating the regularities of the initial data. 
\end{remark}

\begin{remark}
There is no explicit regularity assumption on $F$ in the previous theorem, since $F$ does not appear explicitly in the evolution of $\mathcal{H}_N(\rho_N \vert \bar \rho_N)(t)$.  Nevertheless some regularity on $F$ is implicitly required, in particular to obtain $W^{2,p}$ solution $\bar\rho$ to \eqref{limitmeanfield}. The constant $\bar M$ depends on $\|\udiv F \|_{L^\infty}$ only in the case $\sigma_N \not \equiv\sigma$. See the proof of Lemma \ref{LemTimEvo} for details. 
\end{remark}
\begin{remark}
  While our results are presented for simplicity in the torus $\Pi^d$, they could be extended to any bounded domain $\Omega$ with appropriate boundary conditions. The possible extension to unbounded domains however appears highly non-trivial, in particular in view of the assumption $\inf \bar \rho>0$ which could not hold anymore.
  \end{remark}

\medskip

The proof of Theorem \ref{maintheorem} strongly relies on the properties of the relative entropy over tensorized spaces such as $\Pi^{d\,N}$. Those properties are also critical to derive appropriate control on the observables or marginals $\rho_{N,k}$. In particular the sub-additivity implies that the relative entropy of the marginals is bounded by the total relative entropy or
\begin{equation}
  {\cal H}_k(\rho_{N,k}\;|\;\bar\rho^{\otimes^k})=\frac{1}{k}\int_{\Pi^{d\,k}}\rho_{N,k}\, \log \frac{\rho_{N,k}}{\bar\rho^{\otimes^k}}\,\ud x_1\dots\ud x_k\leq {\cal H}_N(\rho_N\;|\;\bar\rho_N),\label{subadditive}
\end{equation}
for which we refer to \cite{HM,MM,MMW} where estimates quantifying the classical notion of propagation of chaos are thoroughly investigated.

It is then possible to derive from Theorem \ref{maintheorem} the strong propagation of chaos as per
\begin{corollary}\label{maincorollary}
Under the assumptions of Theorem \ref{maintheorem}, if ${\cal H}_N(\rho_N^0\,|\;\bar\rho_N^0)\rightarrow 0$ as $N\rightarrow \infty$, then over any fixed time interval $[0,\ T]$
\[
{\cal H}_N(\rho_N\,|\;\bar\rho_N)\longrightarrow 0,\quad as\ N\rightarrow\infty.
\]
As a consequence considering any finite marginal at order $k$, one has the strong propagation of chaos
\[
\|\rho_{N,k}-\bar\rho^{\otimes^k}\|_{L^\infty([0,\ T],\ L^1(\Pi^{d\, k}))}\longrightarrow 0.
\]
Finally in the particular case where $\sup_N N\,{\cal H}_N(\rho_N^0\,|\;\bar\rho_N^0)=H<\infty$, and where $\sup_N N\,|\sigma_N-\sigma|=S<\infty$, then one has that, for some constant $C$ depending only on $k,\;H,\;S,\;T$ and $\|K\|$ and $\bar M$ defined in Theorem \ref{maintheorem},
\begin{equation}
\|\rho_{N,k}-\bar\rho^{\otimes^k}\|_{L^\infty([0,\ T],\ L^1(\Pi^{d \,k}))}\leq \frac{C}{\sqrt{N}}.\label{optsqrtN}
\end{equation}
\end{corollary}
\begin{remark}
The rate of convergence in $1/\sqrt{N}$ in \eqref{optsqrtN} is widely considered to be optimal as it corresponds to the size of stochastic fluctuations. We refer for example to \cite{Mal} where entropy methods are used in this context for smooth interaction kernels; see also the prior \cite{BRTV} and \cite{BGM,CGM}.   
\end{remark}
\begin{proof}
  Corollary \ref{maincorollary} follows directly from Theorem \ref{maintheorem} by using inequality \eqref{subadditive} and the Csisz\'ar-Kullback-Pinsker inequality (see for instance \cite{V}) for any $f$ and $g$ functions on $\Pi^{d\,k}$
  \[
\|f-g\|_{L^1(\Pi^{d\,k})}\leq \sqrt{ 2 k \, {\cal H}_k(f\;|\;g)}. 
  \]   
\end{proof} 

\begin{remark} Theorem \ref{maintheorem} also provides the rate of convergence in the Wasserstein distance by a Talagrand-type inequality (See for instance \cite{BobkovG,BolleyVillani})
\[
W_p( \rho_{N, k}, \bar \rho^{\otimes^k} ) \leq C(\bar \rho, p) \left( k \mathcal{H}_k (\rho_{N, k} \vert \bar \rho^{\otimes^k})\right)^{\frac{1}{2p}}
\]
for any $p \geq 1$, since the underlying space $\Pi^d$ is compact. 
\end{remark}
The starting  steps in the proof of Theorem \ref{maintheorem}, such as the relative entropy and the reduction to a modified law of large numbers, had already been exposed in \cite{JW1}. However the present contribution expands much on the basic ideas and techniques introduced in \cite{JW1}: First we make better use of the diffusion, which was instead mostly considered as a perturbation in \cite{JW1}. This is the main reason why we are essentially able to {\em gain one full derivative} in our assumption on $K$ with respect to the $K\in L^\infty$ in \cite{JW1}.

The main technical contribution in the present article, namely the modified law of large numbers stated in Theorem \ref{MEII}, is considerably more difficult to prove than any equivalent in \cite{JW1}. This has lead to several new ideas in the combinatorics approach, detailed in the proof of Theorem \ref{MEII} in section \ref{ProofMEII}.
Theorem \ref{MEII} corresponds to classical large deviation estimates for instance in \cite{Arous} but for non-continuous potentials, which is new in the literature. We believe that it can be of further and wider use.

The importance of law of large numbers for the propagation of chaos or the mean field limit  has of course long been recognized, at least since Kac, see \cite{Kac} or \cite{Sznitman}. We also refer to \cite{GolMouRic} for an example where the classical law of large numbers is used but which is limited to Lipschitz kernels $K$.

The relative entropy at the level of the Liouville equation does not seem to have been widely used for  mean field limits yet. The relative entropy method, initiated in  \cite{Yau} in the context of hydrodynamics of Ginzburg-Landau and now has been extensively used for hydrodynamics limits (see chapter 6 in \cite{KL}), is maybe the closest to the approach developed here. A similar approach, namely a modulated energy argument, was introduced in \cite{Serfaty} to investigate  mean field limits for quantum vortices (see also \cite{DuerSer}), and has been used in \cite{Duer} for gradient flows with Riesz-like potentials and in \cite{SerfatyCou} for 1st order Coulomb flows.  We also refer to \cite{FoJo} for a different, trajectorial, view on the role of the entropy in SDEs.   

%
%%%%%%%%%%%%%%%%%%%%%%%%%%%%%%%%%%%%%%%%%%%%%%%%%%%
\subsection{Applications\label{secapplications}} \label{ApplicationSubsection}
%%%%%%%%%%%%%%%%%%%%%%%%%%%%%%%%%%%%%%%%%%%%%%%%%%%
We delve in this section into some examples of kernels $K$ that our method can handle and discuss at the same time where our result stands in comparison to the existing literature. In general quantitative estimates of propagation of chaos were previously only available for smooth, Lipschitz, kernels $K$ such as in the classical result \cite{McKean}; see also \cite{BRTV,BGM,CGM,Mal} for more on the classical Lipschitz case. Gronwall-like estimates with Lipschitz force fields, but a fixed number of SDEs,  were also at the basis of \cite{Ito}.  

System \eqref{SDE1st} retains simple additive interactions, contrary to the more complex structure found for example in \cite{MT1,MT2}; but it still includes a large range of first order models, such as swarming, opinion dynamics,  aggregation equations, neuroscience models,  see for instance \cite{BoCaCa,CaChHa,CS2,FS} or \cite{Krause} and the reference therein.  The propagation of chaos of stochastic system \eqref{SDE1st} is also  closely related to complex geometry, which has been investigated in \cite{BO,Berman}. The Dyson Brownian motions, i.e. \eqref{SDE1st} with $K(x) = 1/x$ in 1D, or more general mean filed models at low temperature, are also connected to  random matrix theory   \cite{Andersonetc,ErdosYau}...
The list of examples given below is hence by no means exhaustive and we refer to our recent survey \cite{JW2} for a more thorough discussion of current important questions.

\begin{itemize}
\item The 2d viscous vortex model where $K$ satisfies \eqref{biot-savart}. As mentioned in the introduction,
  the mean field limit \eqref{limitmeanfield} is then the 2d incompressible Navier-Stokes equation written in vorticity form, Eq. \eqref{vorticity}. We can write
  \[
  K=\udiv V,\quad V=\left[\begin{matrix}
  &-\phi\,\mbox{arctan} \frac{x_1}{x_2}+\psi_1&0\\
  &0&\phi\, \mbox{arctan} \frac{x_2}{x_1}+\psi_2
  \end{matrix}\right],
\]
where one can choose $\phi$ smooth with compact support in the representative
$(-1/2,\;1/2)^2$ of $\Pi^2$ and $(\psi_1,\;\psi_2)$ a corresponding smooth correction to periodize $V$. Therefore $K$ satisfies the assumptions of Theorem \ref{maintheorem}.
%$\psi_0$ is a smooth correction to periodize $\psi$. Under this form it is clear that $\psi\in BMO(\Pi^2)$, while $\udiv K=0$ which means that $K$ satisfies the assumptions of Theorem \ref{maintheorem}.

The convergence of the systems of point vortices \eqref{SDE1st} to the limit \eqref{vorticity} had first been established in \cite{Osada} for a large enough viscosity $\sigma$. The well posedness of the point vortices dynamics has been proved globally in \cite{Osada86}; see also  \cite{FGP}.  Finally the convergence to the mean field limit has been obtained with any positive viscosity $\sigma$ in the recent \cite{FHM}.

However those results rely on a compactness argument based on a control of the singular interaction provided by the dissipation of entropy in the system. %They hence cannot provide a convergence rate toward the limit.

As far as we know, this article is {\em the first to provide a quantitative rate of propagation of chaos  for the 2d viscous vortex model}. 

%As we mentioned before, we only treat present our result in a simplest setting (\ref{SDE1st}) to highlight the essential difficulty arsing from the singularity of the kernel \eqref{biot-savart}. 
%
\medskip
\item Hamiltonian structure. If the dimension $d$ is even then the previous example can be generalized to include any Hamiltonian structure. In that case one has $d=2n$, $x=(q,p)$ with $q,p\in \Pi^{n}$ and for some Hamiltonian $H\;:\Pi^{2n}\longrightarrow \R$,
  \[
K=(\nabla_p H,\;-\nabla_q H).
\]
Theorem \ref{maintheorem} now applies if $H\in L^\infty(\Pi^{2n})$, though this may not be the optimal condition (see the discussion below). The theorem provides propagation of chaos for such systems with diffusion with much weaker assumptions than any comparable result in the literature.

We are nevertheless somewhat limited by our framework here. One would for example typically want to apply this to the classical Newtonian dynamics where $H=\sum_i p_i^2/2+\frac{1}{N}\sum_{i,j} V(q_i-q_j)$. This is formally easy by choosing the appropriate function $F$ in the system of particles \eqref{SDE1st}.

The first issue is that the momentum should be unbounded instead of having $p\in \Pi^n$; as we mentioned in one of the remarks after Theorem \ref{maintheorem}, such an extension of our result to $p\in\R^n$ for example would be non-trivial...

The second issue concerns the diffusion which for such models usually applies only to the momentum. This leads to a degenerate diffusion whereas we absolutely require it in every variable.
\medskip
\item Collision-like interactions. We can even handle extremely singular interactions where some sort of collision event occurs at some fixed horizon. Consider for example any function $\phi\in L^1(\Pi^d)$, any smooth field $M(x)$
  of matrices and define
  \[
K=\udiv(M\,\ind_{\phi\leq 0}),\quad\mbox{or}\ K_\alpha(x)=\sum_\beta \partial_\beta (M_{\alpha\beta}(x)\,\ind_{\phi(x)\leq 0}).
  \]
It is straightforward to choose $M$ s.t. $\udiv K\in \dot W^{-1,\infty}$ or even $\udiv K=0$: A simple example is simply to take $M$ anti-symmetric.  As $M\,\ind_{\phi\leq 0}\in L^\infty$, Theorem \ref{maintheorem} applies. This particular choice of $K$ means that two particles $i$ and $j$ will interact exactly when $\phi(X_i-X_j)=0$. An obvious example is $\phi(x)=|x|^2-(2R)^2$ in which case the particles can be seen as balls of radius $R$ which interact when touching.

  But in the context of swarming, one could have birds, or other animals, which interact as soon as they can see each other; this is different from the cone of vision type of interaction found for example in \cite{CCHS} where the interaction is much less singular (bounded).  Micro-organisms such as bacteria may also have complicated, non-smooth shapes. In all those cases $\{\phi\leq 0\}$ is not a ball in general and may even be a singular set.

  Since $M(x)$ is smooth, one could interpret $K$ as being supported on the measure $\delta_{\phi=0}$. But in fact we do not need any regularity on $\phi$, not even $\phi\in BV$ and here $K$ may not even be a measure...
\medskip
\item Gradient flow structure. The dual to the Hamiltonian case is to take $K=\nabla \psi$ for some potential $\psi$. This lets us see the system of particles \eqref{SDE1st} as a gradient flow with diffusion and it endows the mean field limit \eqref{limitmeanfield} with the derived and nonlinear gradient flow structure.
  
  When $\psi$ is convex, but not necessarily smooth, it is possible to strongly use this gradient flow structure. This is in particular the key to obtain the well posedness of Eq. \eqref{limitmeanfield}, even without diffusion, as in \cite{CDFLS,CLM} and in \cite{BoCaCa} for the mean field limit.

  However it does not seem easy for our approach to fully make use of such gradient flows. This is seen on the assumptions of Theorem \ref{maintheorem} where having $K\in \dot W^{-1,\infty}$ is not very demanding, $\psi\in L^\infty$ would be enough, while the condition $\udiv K\in  \dot W^{-1,\infty}$ actually forces us to consider Lipschitz potentials $\psi$. Of course any $\psi$ convex is Lipschitz so that Theorem \ref{maintheorem} still extends the known theory for general $\psi$. But it is clearly not performing as well as in the Hamiltonian case. 

A very good example of this is the 2d Patlak-Keller-Segel model of chemotaxis where one would like to have $K=\alpha\,x/|x|^2+K_0(x)$. This choice of $K$ is just a rotation of $\pi/2$ from the 2d Navier-Stokes kernel given by \eqref{biot-savart}. Therefore we still have that $K\in \dot W^{-1,\infty}$ by using a rotation of the matrix $V$ that we wrote in the Navier-Stokes setting. But unfortunately $\udiv K$ is now one full derivative away from $ \dot W^{-1,\infty}$ and Theorem \ref{maintheorem} cannot be applied.  

By studying the specific properties of the system though, a convergence result to measure-valued solutions was obtained in  \cite{HaSc} while the convergence to weak solutions was achieved in \cite{FJ} (see also \cite{GQ} for the sub-critical case). We also refer to \cite{LY} for general Coulomb interactions. Those results are not quantitative though and a major open problem remains to find an equivalent of Theorem \ref{maintheorem} in this case.
\end{itemize}

\bigskip

We wish to conclude this subsection about kernels $K$ to which Theorem \ref{maintheorem} applies, by discussing more in details the assumption $K\in \dot W^{-1,\infty}$. 

We first come back to the vortex dynamics for 2d Navier-Stokes and the kernel $K$ given by the Biot-Savart law \eqref{biot-savart}. Since $\udiv K=0$, the classical way to represent $K$ is by $K=\mbox{curl}\,\psi$ with
\[
\psi(x)= \alpha\,\log |x|+\psi_0(x),
\]
with again $\psi_0$ a smooth correction to periodize $\psi$. Obviously $\psi$ is not bounded which at first glance suggests that $K$ does not belong to $ \dot W^{-1,\infty}$. This is incorrect as the ``right'' choice of $V$ above demonstrates but it means that {\em knowing whether $K\in \dot W^{-1,\infty}$ is not as simple as it may seem}.

The distinction is rather technical but it is {\em critical} for us as it allows us to handle the crucial example of the vortex model. It also turns out to be {\em connected with a fundamental difficulty} in our proof. Our estimates directly use a representation $K=\udiv V$ and the most difficult term would vanish if $V$ were anti-symmetric, which is the case if we take $K=\mbox{curl}\,\psi$. {\em The fact that we cannot take $K=\mbox{curl}\,\psi$ with $\psi\in L^\infty$ is responsible for the main technical difficulty in this article} and in particular this is what requires Theorem \ref{MEII} whose proof takes all of section \ref{ProofMEII}. We refer to the more specific comments that we make in subsection \ref{sketch}.

\medskip

In general the study of the $K$ for which there exists a matrix field $V\in L^\infty$ s.t. $\udiv V=K$ turns out to be a very complex mathematical question. This can be done coordinate by coordinate obviously so the question is equivalent to finding the scalar field $\phi$ for which there exists a vector field $u\in L^\infty$ s.t. $\udiv u=\phi$. 

The difficulty is that for a given $K$, there {\em does not exist a unique matrix field $V$ s.t. $\udiv V=K$}. Of course in dimension $d=2$ if $\udiv K=0$, then there exists a unique $\psi$ up to a constant,  s.t. $K=\mbox{curl}\,\psi$. In dimension $d>2$, if $\udiv K=0$, there exists an anti-symmetric matrix $V$ s.t. $K=\udiv V$. The anti-symmetric matrix $V$ is not unique in general though with the well known issue of the gauge choice for vector potential if $d=3$.

But even in dimension $2$, there is no reason why $\psi\in L^\infty$ if $K\in \dot W^{-1,\infty}$. This is indeed connected to the fact that the Riesz transforms are unbounded on $L^\infty$ and the kernel K of \eqref{biot-savart} is the classical example of this. Instead one only has in general that $\psi\in BMO$. 

However even in this simple case, {\em it is not known if $\psi\in BMO$ is equivalent to $K\in \dot W^{-1,\infty}$}. This question is connected to the classical representation of $BMO$ functions in \cite{FeSt}. For any $\psi\in BMO$, \cite{FeSt} showed that there exists $\psi_0,\;\psi_1,\;\psi_2\in L^\infty$ s.t. $\psi=\psi_0+R_1\,\psi_1+R_2\,\psi_2$ with $R_i$, $i=1,\;2$, the Riesz transforms. If it were always possible to take $\psi_0=0$ then we would have the equivalence but that seems (at best) highly non-trivial.

Instead the positive results that we have are much more recent and limited. This line of investigation was started in the seminal \cite{BoBr} which proved that if $K\in L^d(\Pi^d)$ then $K\in \dot W^{-1,\infty}(\Pi^d)$. If $K$ is known to be a signed measure then this was extended in \cite{PhTo} to find that $K=\udiv V$ with $V\in L^\infty$ iff there exists $C$ s.t. for any Borel set $U$
\begin{equation}
\left|\int_U K(dx)\right|\leq C\,|\partial U|.\label{phuctorres}
\end{equation}
This result in \cite{PhTo} hence has the direct consequence
\begin{proposition}
If $d>1$ and $K$ belongs to the Lorentz space $L^{d,\infty}(\Pi^d)$ then $K \in \dot W^{-1,\infty}$.
\label{lorentz}
\end{proposition}
\begin{proof}
Assuming $K\in L^{d,\infty}$ then for a constant $C$, we have that
\[
|\{x\in\Pi^d,\;|K(x)|\geq M\}|\leq \frac{C}{M^d}.
\]
Decompose now dyadically
\[
\int_U |K(x)|\,dx\leq |U|+\sum_{k\geq 0} 2^{k+1}\,|\{x\in U,\;|K(x)|\geq 2^k\}|.
\]
Define $k_0$ s.t. $2^{-d\,(k_0+1)}\leq |U|\leq 2^{-d\,k_0}$ and bound
\[\begin{split}
&|\{x\in U,\;|K(x)|\geq 2^k\}|\leq |U|\ \mbox{for}\ k\leq k_0,\\
&|\{x\in U,\;|K(x)|\geq 2^k\}|\leq |\{x\in \Pi^d,\;|K(x)|\geq 2^k\}|\leq \frac{C}{2^{d\,k}}\ \mbox{for}\ k> k_0.
\end{split}
\]
This leads to
\[\begin{split}
\int_U |K(x)|\,dx&\leq |U|+\sum_{k\leq k_0} 2^{k+1}\,|U|+C\,\sum_{k>k_0} 2^{(1-d)\,k+1} \\
&\leq |U|+2^{k_0+2}\,|U|+C\,2^{(1-d)\,k_0+1}\leq C'\,|U|^{\frac{d-1}{d}},
\end{split}
\]
by using the definition of $k_0$. By the isoperimetric inequality, there exists a constant $C_d$ s.t. $|U|^{\frac{d-1}{d}}\leq C_d\,|\partial U|$ so that we  verify the condition \eqref{phuctorres} which concludes the proof.
\end{proof} 
Prop. \ref{lorentz} not only applies to $K$ given by \eqref{biot-savart} but proves in general that any $K$ with  $|K(x)|\leq C/|x|$ belongs to $ \dot W^{-1,\infty}$. This in particular implies that  our result in the case with vanishing viscosity, Theorem  \ref{thmvanishingviscosity} in the next subsection, is indeed weaker that Theorem  \ref{maintheorem} when viscosity does not degenerate.

The original result in \cite{BoBr} is not constructive, and it is even proved that the $V\in L^\infty$ s.t. $K=\udiv V$ cannot be obtained linearly from $K$. The development of constructive algorithms to obtain $V$ is a current important  field of research, see \cite{Tad}.

%%%%%%%%%%%%%%%%%%%%%%%%%%%%%%%%%%%%%%%%%%%%%%%%%%%%%%%%%%%%%%%%%%
\subsection{The case with vanishing diffusion \label{vanishingdiffusion}}
%%%%%%%%%%%%%%%%%%%%%%%%%%%%%%%%%%%%%%%%%%%%%%%%%%%%%%%%%%%%%%%%%%%%%%%
While we are mostly interested in Eq. \eqref{Liouville} when the viscosity does not asymptotically vanishes, a nice (and essentially free) consequence of the method developed here is to also provide a result with vanishing viscosity.

The result is of course weaker and requires that $K\in L^\infty$ with $\udiv K\in L^\infty$ or that  $|K(x)| \leq C/|x|$ but $K$ is anti-symmetric ($K(-x) = -K(x)$) also with $\udiv K \in L^\infty$.  Obtaining an entropy solution to \eqref{Liouville} in the sense of Def. \ref{entropysol} is even more straightforward in these cases as there is no need for integration by parts. However, we emphasize that in the case that $K$ is anti-symmetric and $|K(x)| \leq C/|x|$, we should  understand the product $K \rho_N$ using the classical observation from Delort \cite{Delort} 
\[
\begin{split}
& \int_{\Pi^{2d}} \phi(t, x_1, x_2)  K(x_1-x_2) \rho_{N, 2}(t, x_1, x_2) \ud x_1 \ud x_2 \\
& = \frac{1}{2} \int_{\Pi^{2d}} \left( \phi(t, x_1, x_2) - \phi(t, x_2, x_1) K(x_1 -x_2) \rho_{N, 2}(t, x_1, x_2) \right) \ud x_1 \ud x_2  \\
& \leq C \|\nabla \phi \|_{L^\infty},  \\
\end{split}
\]
where the equality is ensured by the is anti-symmetry of $K$ and the symmetry of $\rho_N$ and therefore $\rho_{N, 2}$.

 Moreover we also directly obtain the following bound, which replaces in that case the one provided by Prop. \ref{entropybounds},
\begin{equation}\begin{split}
&\int_{\Pi^{d\,N}}\rho_N(t,X^N)\,\log \rho_N(t,X^N)\,\ud X^N+\sigma_N\,\sum_{i=1}^N\int_0^t\int_{\Pi^{d\,N}}\frac{|\nabla_{x_i}\rho_N|^2}{\rho_N}\,\ud X^N\,\ud s\\
    &\quad\leq \int_{\Pi^{d\,N}}\rho_N^0\,\log \rho_N^0\,\ud X^N+N\,t\,\left(\|\udiv K\|_{L^\infty}+\|\udiv F\|_{L^\infty}\right).\\
\label{entropybounds2}
  \end{split}
\end{equation}
Under those stronger assumptions on $K$, we have the following result
\begin{theorem}
  Assume that $\udiv F\in L^\infty(\Pi^d)$, $\udiv K\in L^\infty(\Pi^d)$ and that either $K\in L^\infty(\Pi^d)$ or for $d \geq 2$, $K(-x)=-K(x)$ with $|x|\, K(x) \in L^\infty(\Pi^d)$. Assume moreover that $\rho_N$ is an entropy solution to Eq. \eqref{Liouville} as per Def. \ref{entropysol}. Assume finally that $\bar\rho\in L^\infty([0,\ T],\ W^{1,\infty}(\Pi^d))$ solves Eq. \eqref{limitmeanfield} with $\int_{\Pi^d} \bar\rho=1$.
Then  
\begin{equation}\label{Theorem2EvoEntropy}\begin{split}
 {\cal H}_N(\rho_N\,|\;\bar\rho_N)(t)\leq &e^{\bar M_2\,\|K\|_\infty\,t}\,
\bigg({\cal H}_N(\rho_N^0\,|\;\bar\rho_N^0)+\frac{1}{N}\\
&\qquad+\bar M_2 \,(1 + \|K\|_{\infty} t) \,|\sigma-\sigma_N| \bigg),
\end{split}
\end{equation}
where we now denote $\|K\|_\infty=\|K\|_{L^{\infty}}+\|\udiv K\|_{L^{\infty}}$ in the general case and  $\|K\|_\infty=\||x|\,K\|_{L^{\infty}}+\|\udiv K\|_{L^{\infty}}$ for the anti-symmetric case while $\bar M_2$ is a constant which only depends on
\[\begin{split}
&\bar M_2\bigg( \sigma, \;  \|\log \bar \rho\|_{BMO}, \; \sup_{p \geq 1} \frac{\|\nabla \log \bar\rho\|_{L^p(\bar \rho \ud x )}}{p}, \; \frac{1}{N} \int_{\Pi^{d \, N}} \rho_N^0 \log \rho_N^0, \; \|\udiv F\|_{L^\infty}\\
&\qquad \| \bar \rho\|_{L^\infty},\; \sup_{p \geq 1}  \frac{\|\nabla^2 \log \bar\rho\|_{L^p(\bar \rho \ud x )}}{p}, \; \| \log \bar \rho\|_{W^{1, \infty}} \bigg).\\
%\frac{1}{N}\int_{\Pi^{d\,N}} \rho_N^0\,\log \rho_N^0,\;\|\udiv F\|_{L^\infty}
%&\frac{1}{N}\int_{\Pi^{d\,N}}\rho_N^0\, \log\rho_N^0, \| \log \bar \rho\|_{W^{1, \infty}}\Bigg). 
\end{split}
\]
%For the case when $K$ is anti-symmetric and $|K(x)| \leq C/|x|$, $\udiv K \in L^\infty$ but also $d\geq 2$, the estimate \eqref{Theorem2EvoEntropy} also holds but with 
%\[
%\|K \|_{\infty} = \|xK(x)\|_{L^\infty} + \| \udiv K  \|_{L^\infty} + \sup_{t\in[0, T]} \|K \star \rho_t\|_{L^\infty} ,
%\]
%and $\bar M_2$ now also  depends on  $\sup_{p \geq 1}  \frac{\|\nabla^2 \log \bar\rho\|_{L^p(\bar \rho \ud x )}}{p} $. 
  \label{thmvanishingviscosity}
\end{theorem}

\begin{remark} The constant $\bar M_2$ is in the above  complex form simply because we include all cases $\sigma_N \to \sigma\geq 0$. For instance if $\sigma_N \equiv \sigma$, then $\bar M_2$ does not depend on  $\frac{1}{N} \int_{\Pi^{d\, N}} \rho_N^0 \log \rho_N^0$, $\| \udiv F\|_{L^\infty}$ and $\|\log \bar \rho\|_{W^{1, \infty}}$. %$\|\log\bar\rho\|_{BMO}$.
%\[
%\bar M_2 \left( \sup_{p\geq 1} \frac{ \|\nabla \log \bar \rho\|_{L^p(\bar \rho \ud x )}}{p},  \left( \sup_{p \geq 1}\frac{\|\nabla^2 \log \bar\rho\|_{L^p(\bar\rho\,dx)}}{p} \right)\right).
%\]
%For the vanishing viscosity case $\sigma_N \to \sigma=0$, $\bar M_2$ only explicitly depends on 
%\[
%\bar M_2 \left(\sup_{p\geq 1} \frac{ \|\nabla \log \bar \rho\|_{L^p(\bar \rho \ud x )}}{p}, \; \left( \sup_{p \geq 1}\frac{\|\nabla^2 \log \bar\rho\|_{L^p(\bar\rho\,dx)}}{p}, \ \ \| \bar \rho \|_{L^\infty}\right), \; \   \|\log \bar \rho\|_{W^{1, \infty}} \right).
%\]
See the proof of Theorem 2 in subsection \ref{ProofTh2} for more details. 
 \end{remark}

\begin{remark} To control the error caused by the difference $|\sigma -\sigma_N|$, 
we need $\nabla \log \bar \rho \in L^\infty(\Pi^d)$. This can be replaced by appropriate  moment assumptions like $|\nabla \log \bar \rho (x) | \leq C |x|^k$ so that the result can easily be extended to the whole space $\mathbb{R}^d$. 
\end{remark}

Theorem \ref{thmvanishingviscosity} also applies to the Biot-Savart law  \eqref{biot-savart}, which for $\sigma_N=0$ corresponds to the inviscid point vortex model approximating 2D incompressible Euler equation.
This derivation was an early breakthrough from \cite{GHL} and \cite{Goodman91}, which obtained a very precise and quantitative comparison of the point vortex dynamics with its mean field limit. The results on those articles required however also a precise mesh-like distribution of the point vortices, that is in particular not compatible with random initial conditions. This was a strong motivation for the later works in \cite{Scho96} for example, which allowed for more general initial conditions but less optimal quantitative estimates.

As for the contributions just mentioned, our result strongly relies on the anti-symmetry of the kernel $K$. It provides the optimal rate of convergence while allowing random initial data (and in fact, doesn't work well if particles are initially strongly correlated). But more importantly, it does not require $\sigma_N=0$ so that it is compatible with all sort of vanishing viscosity approximations to the Euler system.

\medskip

As we remarked above, if $|x|\,K\in L^\infty$ then $K\in W^{-1,\infty}$ while on the other hand if $K\in W^{-1,\infty}$ then it can be singular on a more complex set, it can be measure-valued functions or even more general than measures as we discussed earlier in subsection \ref{ApplicationSubsection}.  For this reason Theorem  \ref{thmvanishingviscosity} is obviously mostly only useful in comparison to our main result if $\sigma_N\rightarrow \sigma=0$, including potentially the purely deterministic setting where $\sigma_N=0$ or cases where the viscosity is degenerate in some directions. But it may also require less regularity on the limit $\bar\rho$ and could also be of use in such a situation. In particular it does not require that $\inf \bar\rho>0$ and is hence easy to extend to unbounded domains contrary to Theorem  \ref{maintheorem}.

Because of its usefulness for degenerate viscosities, it is rather natural to compare Theorem  \ref{thmvanishingviscosity} to results for kinetic mean field limits based on the 2nd order dynamics
\begin{equation}
\qquad \ud Q_i=P_i\, \ud t,\quad \ud P_i=\frac{1}{N}\sum_{j=1}^N K(Q_i-Q_j) \ud t +\sqrt{ 2\sigma_N} \ud W_t^i.\label{2ndordersde} 
\end{equation}
We refer to \cite{Golse,Jabin1} for an introduction to the mean field question in this kinetic setting. The best results so far have been obtained in \cite{HJ} for a singular kernel $K$ with $|K(x)|\leq C\,|x|^{ -\alpha}$, $|\nabla K(x)|\leq C\,|x|^{-1-\alpha}$ with $\alpha<1$; in \cite{Ho} for H\"older continuous $K$. The most classical case is again the Poisson kernel $K(x)= \gamma_d \,  x/|x|^d$ which is unfortunately out of reach so far (except in dimension $1$ as in \cite{HS}). It is possible to treat truncated kernels such as $K(x)= \gamma_d \, x/(|x|+\eps_N)^d$ with the most realistic $\eps_N$ obtained in \cite{Laz,LazPic}. However none of the techniques in those articles seems, so far, to be able to handle any diffusion and especially vanishing or degenerate diffusion as in \eqref{2ndordersde}. In the case of \eqref{2ndordersde} where the limiting equation is often called Vlasov-Fokker-Planck, we refer for example to \cite{BGM} which requires more regularity on $K$.

We remark that in comparison, the theory of mean field limits for purely {\em 1st order systems without viscosity} is much more advanced.
In particular the limit of point vortices had already obtained in \cite{GHL}, with a very precise comparison at the level of characteristics but very specific initial conditions as well. The requirements on the initial data was later relaxed in \cite{Scho96} to allow for random initial distributions at the cost of a less accurate comparison. In \cite{LiuXin}, it was even possible to obtain 2D vortex sheet at the limit. 
Those results rely on the particular structure of the Biot-Savart law, and especially on the cancellation at the heart of Delort's argument.

Nevertheless, it was proved in \cite{Hauray} using appropriate Wasserstein distances, that the mean field limit holds for any interaction kernel with $|K(x)|\lesssim |x|^{-s}$ and $|\nabla K|\lesssim |x|^{-s-1}$ with $s<d-1$, without any other structure on $K$ and in any  dimension but not including the Coulomb case.

More recently, a relative entropy approach based on the natural energy of the system has been introduced in \cite{Duer}. This allows for a direct control on the difference between the empirical measure and the limit. The method performs especially well on gradient flows (where our present techniques are sub-optimal) and allows to obtain the mean field limit for general Riesz potentials (including Coulomb in 2D). This approach can also be used when the discrete dynamics is not immediately under the form of an aggregation equation, with Ginzburg-Landau vortices in \cite{DuerSer}. The technique also allowed to include Coulomb interaction in any dimension in \cite{SerfatyCou}.

However it remains quite challenging to employ the techniques developed in those deterministic settings  with any (possibly vanishing) diffusion.

Specifically for stochastic systems with diffusion, a proper use of the gradient flow structure (in comparison to the Hamiltonian structure of the Biot-Savart law) was instrumental in  \cite{CepaL} and more recently in \cite{BO16}.  This allowed to obtain propagation of chaos in \cite{CepaL}, in dimension $d=1$ and for a logarithmic interacting potential, or converting in our notation $K(x) = 1/x$ in 1D. This result could be extended in \cite{BO16} to $K(x) = C \frac{x}{|x|^{s}}$ with $s \in [1, 3)$ but still in 1D  by introducing the right notion of quasi-convexity.

Another obvious point of comparison for Theorem  \ref{thmvanishingviscosity} is our previous result in \cite{JW1}. This previous result covered the case of \eqref{2ndordersde} with the same assumption $K\in L^\infty$; it also introduced the basic ideas for the method used here, based on  the relative entropy  and combinatorics estimates.

However \cite{JW1} was relying strongly on the symplectic structure of the dynamics in \eqref{2ndordersde}. Extending the method to general kernels $K$ which may not even be Hamiltonian, as is done by Theorem  \ref{thmvanishingviscosity}, changes the scope of the result. It has also been proved to be quite complex: From a technical point of view, the whole combinatorics estimates of \cite{JW1} can be summarized in section \ref{PreCombi} of the present article while the new estimates are considerably longer, see section \ref{ProofMEII}.

%%%%%%%%%%%%%%%%%%%%%%%%%%%%%%%%%%%%%%%%%%%%%%%%%%%%%%%%%%%%%%%%%%%%%
\subsection{Sketch of the proof of Proposition \ref{entropybounds}}
%%%%%%%%%%%%%%%%%%%%%%%%%%%%%%%%%%%%%%%%%%%%%%%%%%%%%%%%%%%%%%%%%%%
The proof follows very classical ideas: Consider a regularized interaction kernel $K_\eps$. Eq. \eqref{Liouville} with $K_\eps$ now has a unique solution $\rho_{N,\eps}$ for any initial measure $\rho_N^0$. The goal is to take the limit $\eps\rightarrow 0$, by extracting weak-* converging sub-sequences of $\rho_{N,\eps}$, and to derive \eqref{Liouville} for the limiting kernel $K$ and the various estimates such as \eqref{entropydissip} and \eqref{estentropybounds}.

The only (small) difficulty in this procedure is to obtain adequate uniform bounds. For this reason we only explain here how to derive those bounds for any weak solution $\rho_N$ to \eqref{Liouville} which also satisfies \eqref{entropydissip}.

The first step is to prove from \eqref{entropydissip} that
\[
\sum_{i=1}^N \int_0^t\int_{\Pi^{d\,N}}\frac{|\nabla_{x_i}\rho_N|^2}{\rho_N}\,\ud X^N\,\ud s\leq C\,\int_{\Pi^{d\,N}} \rho_N^0\,\log \rho_N^0\,\ud X^N.
\]
Observe that if $\udiv K\in \dot W^{-1,\infty}$, that is $\udiv K=\udiv \psi$ with $\|\psi\|_{L^\infty}=\|\udiv K\|_{\dot W^{-1,\infty}}$,  then
\[\begin{split}
& - \frac{1}{N}\sum_{i,\;j=1}^N\int_0^t\int_{\Pi^{d\,N}}\udiv K(x_i-x_j)\,\rho_N\,\ud X^N\,\ud s\\
&\qquad\leq \|\udiv K\|_{\dot W^{-1,\infty}}\,\sum_{i=1}^N \int_0^t\int_{\Pi^{d\,N}} |\nabla_{x_i}\rho_N|\,\ud X^N\,\ud s
\end{split}
\]
On the other hand 
\[\begin{split}
&\sum_{i=1}^N \int_0^t\int_{\Pi^{d\,N}} |\nabla_{x_i} \rho_N|\,\ud X^N\,\ud s\\
&\quad\leq  \frac{\sigma_N}{2\,\|\udiv K \|_{\dot{W}^{-1, \infty}}}\sum_{i=1}^N \int_0^t\int_{\Pi^{d\,N}}\frac{|\nabla_{x_i}\rho_N|^2}{\rho_N}\,\ud X^N\,\ud s\\
&\qquad +\frac{\|\udiv\, K\|_{\dot W^{-1,\infty}}}{2\,\sigma_N} \sum_{i=1}^N \,\int_0^t\int_{\Pi^{d\,N}} \rho_N\,\ud X^N\,\ud s\\
&\quad\leq \frac{N\,t\,\|\udiv\, K\|_{\dot W^{-1,\infty}}}{2\,\sigma_N}+ \frac{\sigma_N}{2 \|\udiv K \|_{\dot{W}^{-1, \infty}}} \sum_{i=1}^N \int_0^t\int_{\Pi^{d\,N}} \frac{|\nabla_{x_i}\rho_N|^2}{\rho_N}\,\ud X^N\,\ud s.
\end{split}\]
This implies that
\[
\begin{split}
& -\frac{1}{N}\sum_{i,\;j=1}^N\int_0^t\int_{\Pi^{d\,N}}\udiv K(x_i-x_j)\,\rho_N\,\ud X^N\,\ud s\\
&\qquad\leq \frac{N\,t\,\|\udiv K\|_{\dot W^{-1,\infty}}^2}{2\,\sigma_N}+ \frac{\sigma_N}{2} \sum_{i=1}^N \int_0^t\int_{\Pi^{d\,N}}\frac{|\nabla_{x_i}\rho_N|^2}{\rho_N}\,\ud X^N\,\ud s.
\end{split}
\]
Introducing this bound in \eqref{entropydissip} shows that
\[\begin{split}
&\int_{\Pi^{d\,N}}\rho_N(t,X^N)\,\log \rho_N(t,X^N)\,\ud X^N+\frac{\sigma_N}{2}\,\sum_{i=1}^N\int_0^t\int_{\Pi^{d\,N}}\frac{|\nabla_{x_i}\rho_N|^2}{\rho_N}\,\ud X^N\,\ud s\\
&\quad\leq \int_{\Pi^{d\,N}}\rho_N^0\,\log \rho_N^0\,\ud X^N+\frac{N\,t\,\|\udiv K\|_{\dot W^{-1,\infty}}^2}{2\,\sigma_N}+N\,t\,\|\udiv F\|_{L^\infty},
\end{split}
\]
which since $\sigma_N\geq \underline{\sigma}$ exactly proves \eqref{estentropybounds}.

From Lemma 3.7 in \cite{HM}, i.e. the Fisher information of $2-$marginal $\rho_{N, 2}$ can be controlled by the  total Fisher information of $\rho_N$,  we know that
\[
\int_0^t\int_{\Pi^{2\,d}} \frac{|\nabla_{x_1}\rho_{N,2}|^2%+|\nabla_{x_1}\rho_{N,2}|^2
}{\rho_{N,2}}\,dx_1\,dx_2\leq  \frac{1}{N} \sum_{i=1}^N \int_0^t\int_{\Pi^{d\,N}}\frac{|\nabla_{x_i}\rho_N|^2}{\rho_N}\,\ud X^N\,\ud s,
\]
which can be proved  by applying Jensen's inequality to the convex function $(a, b) \mapsto |a|^2/b$. 

If $K \in \dot W^{-1,\infty}$, {\em i.e.} if $K(x)=\udiv V(x)$ meaning with the use of coordinates that $K_\alpha(x)=\sum_{\beta=1}^d \partial_{\beta} V_{\alpha \beta}(x)$ with $V$ a matrix-valued field, then for any $\phi\in W^{1,\infty}$
\[\begin{split}
&\int_{\Pi^{2\,d}} K (x_1-x_2)\,\phi(x_1,x_2)\,\rho_{N,2}\,dx_1\,dx_2\\
&\quad= -\int_{\Pi^{2\,d}} V(x_1-x_2)\,(\phi\,\nabla_{x_1}\rho_{N,2}+\nabla_{x_1}\phi\ \rho_{N,2})\,dx_1\,dx_2\\
&\quad\leq \|V\|_{L^\infty}\,\|\nabla\phi\|_{L^\infty}+\|V\|_{L^\infty}\,\|\phi\|_{L^\infty}\, \left(\int_{\Pi^{2\,d}} \frac{|\nabla_{x_1}\rho_{N,2}|^2}{\rho_{N,2}}\,dx_1\,dx_2\right)^{1/2},
\end{split}
\]
which leads to  \eqref{boundKrhoN} using that $\inf_V \|V\|_{L^\infty}=\|K\|_{\dot W^{-1,\infty}}$.

Finally, we note that
\[
\udiv \frac{x}{|x|^\gamma}=\frac{d}{|x|^\gamma}-\gamma\,\sum_\alpha \frac{x_\alpha\,x_\alpha}{|x|^{\gamma+2}}=\frac{d-\gamma}{|x|^{\gamma}},
\]
%Hence
%\[\begin{split}
%&\int_{\Pi^{2\,d}} \frac{\rho_{N,2}}{|x_1-x_2|^\gamma}\,dx_1\,dx_1\leq\frac{1}{d-\gamma}\,\int_{\Pi^{2\,d}} \frac{|\nabla_{x_1}\rho_{N,2}|}{|x_1-x_2|^{\gamma-1}}\,dx_1\,dx_1\\
%&\leq \frac{1}{d-\gamma}\,\left(\int_{\Pi^{2\,d}} \frac{\rho_{N,2}}{|x_1-x_2|^{2\gamma-2}}\right)^{1/2}\,\left(\int_{\Pi^{2\,d}} \frac{|\nabla_{x_1}\rho_{N,2}|^2}{\rho_{N,2}}\,dx_1\,dx_2\right)^{1/2}.
%\end{split}
%\]
%If $\gamma\leq 2$ then $2\gamma-2\leq \gamma$ and hence
so that with the same approach it would be possible to derive the bound
\[\begin{split}
&\int_{\Pi^{2\,d}} \frac{\rho_{N,2}}{|x_1-x_2|^\gamma}\,dx_1\,dx_2\leq \frac{1}{(d-\gamma)^2}\,\int_{\Pi^{2\,d}} \frac{|\nabla_{x_1}\rho_{N,2}|^2}{\rho_{N,2}}\,dx_1\,dx_2,
\end{split}
\]
for any $\gamma<2$ if $d=2$ and for $\gamma=2$ if $d>2$, 
which has proved critical in the previous derivation and studies of the 2d incompressible Navier-Stokes for instance see \cite{FGP,FHM} and \cite{Osada}.
%%%%%%%%%%%%%%%%%%%%%%%%%%%%%%%%%%%%%%%%%%%%%%%%%%%%%%%%%%%%%%%%%%%%%%%%%%%%
%%%%%%%%%%%%%%%%%%%%%%%%%%%%%%%%%%%%%%%%%%%%%%%%%%%%%%%%%%%%%%%%%%%%%%%%%%%%%
\section{Proofs of Theorems \ref{maintheorem} and \ref{thmvanishingviscosity} \label{proofmain}}
%%%%%%%%%%%%%%%%%%%%%%%%%%%%%%%%%%%%%%%%%%%%%%%%%%%%%%%%%%%%%%%%%%%%%%%%%%%%%
%%%%%%%%%%%%%%%%%%%%%%%%%%%%%%%%%%%%%%%%%%%%
\subsection{Sketch of the proof of Theorem \ref{maintheorem}\label{sketch}}
%%%%%%%%%%%%%%%%%%%%%%%%%%%%%%%%%%%%%%%%%%%%
Our goal in this subsection is to present the main  steps of the proof. For this reason, we make several simplifying assumptions that allow us to focus on the main ideas.
First of all, we assume that 
\[
F=0,\quad \udiv K=0,\quad K_\alpha=\sum_\beta \partial_\beta V_{\alpha\beta}\quad\mbox{with}\ \|V\|_{L^\infty(\Pi^d)}\leq \delta,
\]
for $\delta$ small in terms of some norms of $\bar\rho$.

We also assume that $\bar\rho\in C^\infty$ with $\inf \bar\rho>0$ and that $\rho_N$ is a classical solution to \eqref{Liouville} so that we may easily manipulate this equation.

Finally we assume that $\sigma_N=\sigma=1$.

Following our previous discussion about the criticality of the assumption $K=\udiv V$ with $V\in L^\infty$, we refer the readers in particular to the end of step $2$ after formula \eqref{defABsimplified} and to step $5$ in the following proof. That step requires the use of Theorem \ref{MEII} whose proof contains the main technical difficulties of the article.

If instead one would assume that $V$ is anti-symmetric then the term $\tilde B$ in step $5$ vanishes and as we mentioned above, we would have a much simpler proof. Unfortunately this would not let us handle our most important kernel $K=x^\perp/|x|^2$ corresponding to the 2d incompressible Navier-Stokes system.

\medskip

\noindent {\em Step 1: Time evolution of the relative entropy.} First of all it is straightforward to derive an equation on $\bar\rho_N$ from the limiting equation \eqref{limitmeanfield}
\[
\begin{split}
&\partial_t \bar\rho_N+\sum_{i=1}^N \frac{1}{N}\,\sum_{j=1}^N K(x_i-x_j)\cdot\nabla_{x_i} \bar\rho_N =\sum_{i=1}^N \sigma\,{\Delta_{x_i} \bar\rho_N}\\
&\quad +\sum_{i=1}^N  \bigg(\frac{1}{N}\,\sum_{j=1}^N K(x_i-x_j)-K\star_{x}\bar\rho(x_i)\bigg)\cdot\nabla_{x_i} \bar\rho_N.\\
\end{split}
\]
Combining this with the Liouville equation \eqref{Liouville}, one obtains that
\begin{equation}
\begin{split}
&\frac{\ud }{\ud t}{\cal H}_N(\rho_N\,|\;\bar\rho_N)(t)\\  
& \leq -\frac{1}{N^2}\,\sum_{i,\;j=1}^N \int_{\Pi^{d\,N}}\rho_N\, \left( K(x_i-x_j)-K\star_{x}\bar\rho(x_i)\right)\cdot\nabla_{x_i} \log\bar\rho_N\,\ud X^N\\
&\quad -\frac{1}{N}\,\sum_{i=1}^N\int_{\Pi^{d\,N}}\rho_N\,\left|\nabla_{x_i} \log \frac{ \rho_N}{ \bar \rho_N}\right|^2.
\end{split}\label{simplifiedentropy}
\end{equation}
A full justification of this calculation is given later in the main proof in Lemma \ref{entropytimeevolution}.

\medskip

\noindent {\em Step 2: Using $K=\udiv V$.} As the kernel $K$ is not bounded but we only have that $K=\udiv V$ with $V\in L^\infty$, the next step is to integrate by parts to make $V$ explicit in our estimates. Writing $K_\alpha=\sum_\beta \partial_\beta V_{\alpha\beta}$, we find
\[
\begin{split}
& -\frac{1}{N^2}\,\sum_{i,\;j=1}^N \int_{\Pi^{d\,N}}\rho_N\, \left( K(x_i-x_j)-K\star_{x}\bar\rho(x_i)\right)\cdot\nabla_{x_i} \log\bar\rho_N\,\ud X^N \\
  &= -\frac{1}{N^2}\,\sum_{\alpha,\beta}\sum_{i,\;j=1}^N \int_{\Pi^{d\,N}}\!\! \left( \partial_{x_i^\beta} V_{\alpha\beta}(x_i\!-\!x_j)-\partial_{x_i^\beta} V_{\alpha\beta}\star_{x}\bar\rho(x_i)\right)\, \frac{\rho_N}{\bar\rho_N} \,\partial_{x_i^\alpha}\bar\rho_N\,\ud X^N,
\end{split}
\]
so that integrating by part, this term is equal to 
\[
\begin{split}
  & \ \frac{1}{N^2}\,\sum_{\alpha,\beta}\sum_{i,\;j=1}^N \int_{\Pi^{d\,N}}\rho_N\, \left( V_{\alpha\beta}(x_i-x_j)-V_{\alpha\beta} \star_{x}\bar\rho(x_i)\right)\, \frac{\partial^2_{x_i^\alpha\,x_i^\beta} \bar\rho_N}{\bar\rho_N}\,\ud X^N\\
&+\frac{1}{N^2}\,\sum_{\alpha\beta}\sum_{i,\;j=1}^N \int_{\Pi^{d\,N}}  \left( V_{\alpha\beta}(x_i-x_j)-V_{\alpha\beta}\star_{x}\bar\rho(x_i)\right)\, \partial_{x_i^\alpha} \bar\rho_N  \partial_{x_i^\beta} \frac{\rho_N}{\bar\rho_N} \,\ud X^N.
\end{split}
\]
Writing in tensor form this is finally equal to
\[
\begin{split}
  &\  \frac{1}{N^2}\,\sum_{i,\;j=1}^N \int_{\Pi^{d\,N}}\rho_N\, \left( V(x_i-x_j)-V \star_{x}\bar\rho(x_i)\right)\,:\, \frac{\nabla^2_{x_i} \bar\rho_N}{\bar\rho_N}\,\ud X^N\\
&+\frac{1}{N^2}\,\sum_{i,\;j=1}^N \int_{\Pi^{d\,N}}  \left( V(x_i-x_j)-V\star_{x}\bar\rho(x_i)\right)\,:\, \nabla_{x_i} \bar\rho_N \otimes \nabla_{x_i} \frac{\rho_N}{\bar\rho_N} \,\ud X^N.
\end{split}
\]
The second term involves a derivative of $\rho_N/\bar\rho_N$ which can be controlled thanks to the dissipation term in 
\eqref{simplifiedentropy}. More precisely by Cauchy-Schwartz
\[
\begin{split}
&
\frac{1}{N^2}\,\sum_{i,\;j=1}^N \int_{\Pi^{d\,N}} \left( V(x_i-x_j)-V\star_{x}\bar\rho(x_i)\right) : \nabla_{x_i} \bar\rho_N \otimes \nabla_{x_i} \frac{\rho_N}{\bar\rho_N}  \ud X^N\\
&\leq \frac{1}{N}\sum_{i=1}^N \int_{\Pi^{d\,N}} \bigg|\nabla_{x_i} \frac{\rho_N}{\bar\rho_N}\bigg|^2\,\frac{\bar\rho_N^2}{\rho_N}\,\ud X^N\\
&\qquad+  \frac{1}{N}\sum_{i=1}^N \int_{\Pi^{d\,N}} \rho_N\,\frac{|\nabla_{x_i} \bar\rho_N|^2}{\bar\rho_N^2 }\,\bigg|\frac{1}{N}\sum_{j=1}^N \left( V(x_i-x_j)-V\star_{x}\bar\rho(x_i)\bigg)\right|^2\,\ud X^N.
\end{split}
\]
Of course
\[
\left|\nabla_{x_i} \frac{\rho_N}{\bar\rho_N}\right|^2\, \frac{\bar\rho_N^2}{\rho_N} = \left|\nabla_{x_i}\log \frac{\rho_N}{\bar\rho_N}\right|^2\,\rho_N
\]
so that the first term is actually bounded by the dissipation of entropy. On the other hand 
\[
\frac{|\nabla_{x_i} \bar\rho_N|^2}{\bar\rho_N^2}=\frac{|\nabla_{x_i} \bar\rho(x_i)|^2}{\bar\rho(x_i)^2}.
\]
Hence we obtain that
\begin{equation}
\frac{\ud}{\ud t}{\cal H}_N(\rho_N\,|\;\bar\rho_N)(t)\leq A+B,\label{simplifiedentropy2}
\end{equation}
where
\begin{equation}
\begin{split}  
&A=\frac{C_{\bar\rho}}{N}\sum_{i=1}^N \int_{\Pi^{d\,N}} \rho_N\, \bigg|\frac{1}{N}\sum_{j} \left( V(x_i-x_j)-V\star_{x}\bar\rho(x_i)\right)\bigg|^2\,\ud X^N,\\
&B= \frac{1}{N^2}\,\sum_{i,\;j=1}^N \int_{\Pi^{d\,N}}\rho_N\, \left( V(x_i-x_j)-V\star_{x}\bar\rho(x_i)\right)\,:\,\frac{\nabla^2_{x_i} \bar\rho(x_i)}{\bar\rho(x_i)}\,\ud X^N,\\
\end{split}\label{defABsimplified}
\end{equation}
and $C_{\bar\rho}$ is a constant depending only on the smoothness of $\bar\rho$. 

We point out here that $\nabla^2_{x_i} \bar\rho$ is a symmetric matrix. Hence, {\em if $V$ is anti-symmetric, then the term $B$ completely vanishes: $B=0$}.

\medskip

\noindent {\em Step 3: Change of law from $\rho_N$ to $\bar\rho_N$.} The two previous terms $A$ and $B$ can be seen as the expectations of the corresponding random variables with respect to the law $\rho_N$. Obviously we do not know the properties of $\rho_N$ and would much prefer having expectations with respect to the tensorized law $\bar\rho_N$. We hence use the following
\begin{lemma}
For any two probability densities $\rho_N$ and $\bar\rho_N$ on $\Pi^{d\,N}$, and any $\Phi\in L^\infty(\Pi^{d\,N})$, one has that $\forall \eta>0$
\[
\int_{\Pi^{d\,N}} \Phi\,\rho_N\,\ud X^N\leq \frac{1}{\eta} \Big( {\cal H}_N (\rho_N\,|\bar\rho_N)+ \frac{1}{N}\,\log \int_{\Pi^{d\,  N}} \bar\rho_N\,e^{  N \eta \Phi}\,\ud X^N \Big). 
\]\label{youngconvex}
\end{lemma}
\begin{proof}
We give the (short) proof for the sake of completeness. Without loss of generality, we assume that $\eta =1$.  Define
\[
f=\frac{1}{\lambda}\,e^{N\,\Phi}\,\bar\rho_N,\quad \lambda=\int_{\Pi^{d\, N}} \bar\rho_N\,e^{N\,\Phi}\,\ud X^N.
\]
Notice that $f$ is a probability density as $f\geq 0$ and $\int f=1$. Hence by the convexity of the entropy
\[
\frac{1}{N}\int_{\Pi^{d \,N}} \rho_N\,\log f\,\ud X^N\leq \frac{1}{N}\int_{\Pi^{d\,N}} \rho_N\,\log \rho_N\,\ud X^N.
\]
On the other hand, one can easily check that
\[\begin{split}
\frac{1}{N}\!\!\int_{\Pi^{d\,N}}\!\! \rho_N\,\log f\,\ud X^N&=\int_{\Pi^{d\,N}}\!\! \rho_N\,\Phi\,\ud X^N+\frac{1}{N}\int_{\Pi^{d\,N}}\!\! \rho_N\,\log \bar\rho_N\,\ud X^N-\frac{\log\lambda}{N},
\end{split}\]
which concludes the proof of the lemma.
\end{proof}

To apply Lemma \ref{youngconvex} to $A$, we first expand $A$ coordinate by coordinate as 
\[
A  \leq  \frac{C_{\bar\rho}}{N}\sum_{i=1}^N \sum_{\alpha, \beta=1}^d \int_{\Pi^{d\,N}} \rho_N\, \bigg(  \frac{1}{N}\sum_{j=1}^N \left( V_{\alpha, \beta}(x_i-x_j)-V_{\alpha, \beta}\star_{x}\bar\rho(x_i)\right)\bigg)^2\,\ud X^N.
\]

Now applying Lemma \ref{youngconvex} with  first to each 
\[
\Phi_{\alpha, \beta}=\bigg(\frac{1}{N}\sum_{j=1}^N \left( V_{\alpha, \beta}(x_i-x_j)-V_{\alpha, \beta}\star_{x}\bar\rho(x_i)\right)\bigg)^2,
\]
in $A$ and then to
\[
\Phi=\frac{1}{N^2}\sum_{i,j=1}^N \left( V(x_i-x_j)-V\star_{x}\bar\rho(x_i)\right)\,:\,\frac{\nabla^2_{x_i} \bar\rho(x_i)}{\bar\rho(x_i)},
\]
in $B$, we obtain that
\[
A+B\leq 2\,{\cal H}_N(\rho_N\,|\;\bar\rho_N)(t)+\tilde A+\tilde B,
\]
with
\begin{equation}
\begin{split}  
  &\tilde A=\frac{C_{\bar\rho}}{N^2}\sum_{i=1}^N \sum_{\alpha, \beta=1}^d \log \int_{\Pi^{d\,N}} \exp\bigg(\frac{1}{\sqrt{N}}\sum_{j=1}^N \left( V_{\alpha, \beta}(x_i-x_j)-V_{\alpha, \beta}\star_{x}\bar\rho(x_i)\right)\bigg)^2\\
  &\hskip10cm \bar\rho_N\,\ud X^N,\\
&\tilde B=\frac{1}{N}\log \int_{\Pi^{d\,N}}\bar \rho_N\,e^{ \frac{1}{N}\,\sum_{i,\;j=1}^N\left( V(x_i-x_j)-V\star_{x}\bar\rho(x_i)\right)\,:\,\frac{\nabla^2_{x_i} \bar\rho(x_i)}{\bar\rho(x_i)}}\,\ud X^N.\\
\end{split}\label{defABsimplified2}
\end{equation}
Observe that the cost to perform this change of law is, unfortunately, severe as we now have exponential factors in $\tilde A$ and $\tilde B$. That is the reason why we need $L^\infty$ (or almost $L^\infty$) bounds on $V$.

\medskip

\noindent{\em Step 4:  Bounding $\tilde A$ through a law of large number at the exponential scale.} By symmetry of permutation, we may take $i=1$ in $\tilde A$. Define 
\[
\psi_{\alpha,\beta}(z,x)=V_{\alpha,\beta}(z-x)-V_{\alpha,\beta}\star_{x}\bar\rho(z),
\] 
so that
\[\begin{split}
&\bigg(\frac{1}{\sqrt{N}}\sum_{j=1}^N \left( V_{\alpha, \beta}(x_1-x_j)-V_{\alpha, \beta}\star_{x}\bar\rho(x_1)\right)\bigg)^2\\
&\qquad=\frac{1}{N}  \sum_{j_1, j_2 =1}^N \psi_{\alpha\beta}(x_1,x_{j_1})\psi_{\alpha\beta}(x_1,x_{j_2}).
\end{split}
\]
We remark that each $\psi$ has vanishing expectation with respect to $\bar\rho$
\[
\int_{\Pi^d} \psi_{\alpha\beta}(z, x)\,\bar\rho(x)\,dx=0.
\]

\begin{theorem} \label{MEI} Consider any $\bar\rho \in L^1(\Pi^d)$ with $\bar\rho \geq 0$ and $\int_{\Pi^d} \bar\rho (x) \ud x=1 $. Assume that a scalar function $\psi\in L^\infty$ with  $ \|\psi\|_{L^\infty}  < \frac{1}{2 e}$, and that for any fixed $z$, $\int_{\Pi^d}\,\psi(z,x)\,\bar\rho(x)\,dx=0$ then 
 \begin{equation}
 \label{ME1}
\begin{split}
& \int_{\Pi^{d\,N}} \bar{\rho}_N \exp\bigg(\frac{1}{N}\sum_{j_1,j_2 =1}^N \psi(x_1,x_{j_1})\psi(x_1,x_{j_2})\bigg) \ud X^N \\
&\qquad \leq  C= 2 \left( 1 + \frac{ 10 \alpha}{(1 - \alpha)^3} + \frac{\beta}{1- \beta}  \right), 
\end{split}
 \end{equation}
 where $\bar{\rho}_N(t, X^N)= \Pi_{i=1}^N \bar\rho(t, x_i)$ 
 \begin{equation*}
 \alpha= \left(e\, \|\psi\|_{L^\infty} \right)^4 <1, \quad \beta= \left(\sqrt{2e}\, \|\psi\|_{L^\infty}  \right)^4 <1. 
 \end{equation*} 
\end{theorem}
We give a straightforward proof of Theorem \ref{MEI} in section \ref{proofMEI}, using the combinatorics techniques developed in the article. But note that this theorem is essentially a variant of the  well known law of large numbers at exponential scales; the main difference being that $\psi(x_1,x_{j_1})\psi(x_1,x_{j_2})$ does not have vanishing expectation if $j_1=j_2$, $j_1=1$ or $j_2=1$. {\em Technically Theorem \ref{MEI} is hence rather simple, contrary to Theorem \ref{MEII} below}. 

Using Theorem \ref{MEI} and by taking $\|V\|_{L^\infty}$ small enough, we deduce that
\begin{equation}
\tilde A\leq \frac{C_{\bar\rho }}{N}.\label{boundtildeA}
\end{equation}

\medskip

\noindent{\em Step 5: Bound on $\tilde B$ through a new modified law of large numbers.} We now define
\[
\phi(x,z)=(V(x-z)-V\star \bar\rho(x))\,:\,\frac{\nabla^2_{x} \bar\rho(x)}{\bar\rho(x)},
\]
and we apply to $\tilde B$ the following result
%
%%%%%%%%%%%%%%%%%%%%%%%%%%%%%%%%%%%%%%%%%%%%%%%
\begin{theorem} \label{MEII} Consider $\bar\rho \in L^1(\Pi^d) $ with $\bar\rho \geq 0$ and $\int_{\Pi^d} \bar\rho \ud x =1$. Consider further any $\phi(x,z)\in L^\infty$ with
\[
\gamma :=  C\,  \left( \sup_{p \geq 1} \frac{\|\sup_z |\phi(.,z)|\|_{L^p(\bar\rho \ud x)}}{p} \right)^2   <1,
\]  
where $C$ is a universal constant. Assume that $\phi$ satisfies the following cancellations
\begin{equation}\label{TwoCanLDP}
\int_{\Pi^d} \phi(x,z)\,\bar\rho(x)\,dx=0\quad\forall z, \qquad \int_{\Pi^d} \phi(x,z)\,\bar\rho(z)\,dz=0\quad\forall x.
\end{equation}
Then 
\begin{equation}
\label{ME2}
\int_{\Pi^{d\,N}} \bar{\rho}_N \exp\bigg(\frac{1}{N}\sum_{i,j=1}^N \phi(x_i,x_j)\bigg) \ud X^N \leq \frac{2}{1-\gamma} < \infty, 
\end{equation}
where we recall that $\bar{\rho}_N(t, X^N) = \Pi_{i=1}^N \bar\rho(t, x_i)$.
\end{theorem}
Theorem \ref{MEII} is by far the main technical difficulty in this article. Observe that contrary to classical laws of large numbers, it requires two precise cancellations on $\phi$, separately in $x$ where
\[
\int_{\Pi^d} \phi(x,z)\,\bar\rho(x)\,\ud x=\int_{\Pi^d} (\udiv K(x-z)-\udiv K\star_x \bar\rho(x))\,\bar\rho(x)\,dx=0,
\]
as $\udiv K=0$ and in $z$ where we use the classical cancellation
\[
\int_{\Pi^d} (V(x-z)-V\star_x\bar\rho(x))\,\bar\rho(z)\,dz=0.
\]
Choosing $\delta$ so that $\|V\|_{L^\infty}$ is small enough, Theorem \ref{MEII} again implies that 
\begin{equation}
\tilde B\leq \frac{C_{\bar\rho}}{N}.\label{boundtildeB}
\end{equation}

While Theorem \ref{MEII} looks similar to the modified law of large numbers that was at the heart of our previous result \cite{JW1}, it is considerably more difficult to prove. In \cite{JW1}, we relied a lot on the natural symplectic structure of the problem, which is completely absent here.
{\em The proof Theorem \ref{MEII} is therefore the main technical difficulty and contribution of the article, performed in Section \ref{ProofMEII}}. 

As we noticed earlier, if $V$ were anti-symmetric, then one would have $\phi=0$ and in turn $\tilde B=0$. The main technical difficulty here is  due to  the need for a $V$ without symmetries, which is required to handle 2d incompressible Navier-Stokes.  

\smallskip
Theorem \ref{MEI} is essentially a classical law of large numbers at the exponential scale. 
On the other hand, Theorem \ref{MEII} is actually a result of large deviation. If $\phi$ was continuous, it would follow from the classical \cite{Arous} for example. However with only $\phi$ bounded (which is critical if we want to apply this to the Biot-Savart law), we are not aware of any existing results in the literature. The connection to such large deviation estimates is briefly explained in subsection \ref{largedeviation} below.

\medskip

{\em Final step: Conclusion of the proof.} Inserting \eqref{boundtildeA} and \eqref{boundtildeB} in \eqref{simplifiedentropy2}, we deduce that
\[
\frac{ \ud}{ \ud t} {\cal H}_N(\rho_N\,|\;\bar\rho_N)\leq 2\, {\cal H}_N(\rho_N\,|\;\bar\rho_N)+\frac{C_{\bar\rho}}{N},
\]
allowing to conclude through Gronwall's lemma.

\medskip

There are several additional difficulties in the general proof. The fact that $\|V\|_{L^\infty}$ is not small forces us to carefully rescale all our estimates. Similarly since $\rho_N$ is only an entropy solution to the Liouville Eq. \eqref{Liouville}, we have to proceed more carefully in estimating the relative entropy.
%
%%%%%%%%%%%%%%%%%%%%%%%%%%%%%%%%%%%%%%%%%%%%%%%%%%%%%%%%%%%%%%%%%%%%%%%%%%%
\subsection{A comparison with classical large deviation results\label{largedeviation}}
%%%%%%%%%%%%%%%%%%%%%%%%%%%%%%%%%%%%%%%%%%%%%%%%%%%%%%%%%%%%%%%%%%%%%%%%%%%
We first recall the classical law of large numbers at the exponential scale which one can for instance formulate as
\begin{proposition} \label{LDP} 
  Assume that $\phi\in L^\infty(\Pi^d)$ with $\|\phi\|_{L^\infty}\leq 1$, denote $\mu_N=\frac{1}{N}\sum_i \delta(x-X_i)$ the empirical measure. Then there exists  universal constants $C_1, C_2 >0$, such that  for any $ \bar \rho\in {\cal P}(\Pi^d)$
  \[
\mathbb{E}_{\bar \rho^{\otimes N}} \bigg[ \exp \Big( {{N}\,\left|\int_{\Pi^d} \phi(x)\,( \ud \mu_N(x) - \ud  \bar \rho (x)) \right|^2/C_1}  \Big)  \bigg]  \leq C_2. 
  \]
  where the expectation is taken with respect to the joint distribution $\bar \rho^{\otimes N}$. 
  \end{proposition}
The proof of Proposition \ref{LDP}  can for example be found  in \cite{Ber,Pro,Yuri}.

We further remark that 
 \[
 \int_{\Pi^d} \phi(x) (\ud \mu_N(x) - \ud \bar \rho(x) ) = \frac{1}{N} \sum_{i=1}^N \tilde \phi (X_i)
 \]
  where  $\tilde \phi(x) = \phi(x) - \int_{\Pi^d} \phi(x) \bar \rho( \ud x)$ has mean zero on $\Pi^d$ and the previous expectation under $\bar \rho^{\otimes N}$ is simply 
\[
\int_{\Pi^{dN}} \exp\bigg( \frac{1}{ C_1 N} \sum_{i, j =1}^N \tilde \phi(x_i) \tilde \phi(x_j)\bigg)  \bar \rho^{\otimes N} (\ud x_1 \cdots \ud x_N). 
\]  
Hence Proposition \ref{LDP} implies our Theorem \ref{MEI}. %It has many direct consequences. One corollary  is that the empirical measure  $\mu_N$ converges in law to the limit $\bar \rho$, in the sense that for any $\phi \in C(\Pi^d)$ and $\delta>0$,
%\[
%\mathbb{P}_{\bar \rho^{\otimes N}} \left[ \left| \int_{\Pi^d} \phi(x) \ud \mu_N (x) - \int_{\Pi^d}  \phi(x) \ud \bar \rho(x) \right| \geq \delta  \right] \leq C_2 \exp\left( - \frac{N \delta^2 }{C_1}\right),
%\]
%where the probability is taken with respect to the joint distribution $\bar \rho^{\otimes N}$. 

The counterpart of our Theorem \ref{MEII} in the classical Large Deviation Principle can be found in \cite{Arous}, based on the classical results in \cite{Bol,Var}. See also some applications in  the context of Log and Riesz Gases in  \cite{LebSer}.
%We consider  a function $\phi \in L^\infty(\Pi^{2d})$ satisfying all assumptions in Theorem \ref{MEII},  in particular two cancellation rules \eqref{TwoCanLDP}, which leads to 
%\[
%\frac{1}{N} \sum_{i, j =1}^N \phi(X_i, X_j) = N \int_{\Pi^{2d}} \phi(x, y) (\ud \mu_N( x ) - \ud \bar \rho(x) )( \ud \mu_N (y)  - \ud \bar \rho (y) ),
%\]
%which can be viewed as an interaction type distance between two measures $\mu_N$ and $\bar \rho$. 
Let us reformulate as above  by using the empirical measure, so that estimate \eqref{ME2} in Theorem \ref{MEII} then becomes a bound on 
\begin{equation}\label{LDPInteraction}
Z_N=\mathbb{E}_{\bar \rho^{\otimes N}} \exp \left[  N  \int_{\Pi^{2d}} \phi(x, y) \ud \mu_N(x)\, \ud \mu_N(y)  \right], 
\end{equation}
which should of course be interpreted as a partition function but in our case for a potential that is not the original one.
If $\phi$ is continuous, the expression makes perfect sense (and is otherwise trickier to justify).

The results in \cite{Arous} show that $\lim_{N\to\infty} e^{N\,m_0} Z_N$ exists and is finite; and can even be fully characterized through the right quadratic form on $L^2_{\bar \rho}$. A fortiori $e^{N\,m_0} Z_N$ is bounded.

The key parameter $m_0$ is obtained through the study of the large deviation functional
\[
m_0=\inf_{\mu\in \mathcal{P}(\Pi^d)}\left(\int \log \frac{ \ud \mu(x)}{ \ud \bar\rho(x)}\, \ud \mu(x)-\int \phi(x,y)\, \ud \mu(dx)\, \ud \mu(dy)\right),
\]
where $\frac{ \ud \mu(x)}{ \ud \bar\rho(x)}$ is $+\infty$ unless $\mu$ is absolutely continuous w.r.t. $\bar\rho$ in which case $\frac{ \ud \mu(x)}{ \ud \bar\rho(x)}$ is just the Radon-Nikodym derivative.

The cancellation assumptions \eqref{TwoCanLDP} in Theorem \ref{MEII} which we recall are
\[
\int_{\Pi^d} \phi(x,z)\,\bar\rho(x)\, \ud x=0\quad\forall z, \qquad \int_{\Pi^d} \phi(x,z)\,\bar\rho(z)\, \ud z=0\quad\forall x,
\]
precisely allow to write
\[\begin{split}
m_0=\inf_{\mu\in \mathcal{P}(\Pi^d)}\bigg(&\int \log \frac{ \ud \mu(x)}{ \ud \bar\rho(x)}\,\mu( \ud x)\\
&-\int \phi(x,y)\,( \ud \mu(x)-\bar\rho(x)\, \ud x)\,( \ud \mu(y)-\bar\rho(y)\,\ud y)\bigg).
\end{split}
\]
But now the uniform convexity of $\int \log \frac{ \ud \mu(x)}{ \ud \bar\rho(x)}\,\mu( \ud x)$ dominates the second part provided for example that $\|\phi\|_{L^\infty}$ is small enough. In that case $m_0=0$ and the result in \cite{Arous} not only implies our Theorem \ref{MEII} but also provides a much more precise characterization of the limit.

Unfortunately \cite{Arous} imposes that $\phi$ be continuous and we do not know of another comparable result without that condition. In that sense Theorem \ref{MEII} appear to be new. It also seems to be an open question whether the assumptions on $\phi$ in this theorem are optimal or could be pushed further. And we finally note that even though we have a uniform bound in $N$, we cannot for the moment characterize the limit as in \cite{Arous} if we have so little regularity on $\phi$. 
%
%%%%%%%%%%%%%%%%%%%%%%%%%%%%%%%%%%%%%%%%%%%%%%%%%%%%%%%%%%%%%%%%%%%%%%%%%%%%%
\subsection{Time evolution of the relative entropy}
%%%%%%%%%%%%%%%%%%%%%%%%%%%%%%%%%%%%%%%%%%%%%%%%%
The first step in the proof is to estimate the time evolution of the relative entropy,
\begin{lemma} \label{LemTimEvo}
Assume that $\rho_N$ is an entropy solution to Eq. \eqref{Liouville} as per Def. \ref{entropysol}. Assume that $\bar\rho\in W^{1,\infty}([0,\ T]\times\Pi^{d})$ solves Eq. \eqref{limitmeanfield} with $\inf\bar\rho> 0$ and 
$\int_{\Pi^d}\bar\rho=1$. Then 
\[\begin{split}
&{\cal H}_N(\rho_N\,|\;\bar\rho_N)(t) = \frac{1}{N}\int_{\Pi^{d\,N}}\, \rho_N(t,X^N)\,\log  \frac{\rho_N(t,X^N)}{\bar\rho_N(t,X^N)}\,\ud X^N\leq {\cal H}_N(\rho_N^0\,|\;\bar\rho_N^0)\\
& -\frac{1}{N^2}\,\sum_{i,\;j=1}^N \int_0^t\int_{\Pi^{d\,N}}\rho_N\, \left( K(x_i-x_j)-K\star_{x}\bar\rho(x_i)\right)\cdot\nabla_{x_i} \log\bar\rho_N\,\ud X^N\,\ud s\\
&-\frac{1}{N^2}\,\sum_{i,\;j=1}^N \int_0^t\int_{\Pi^{d\,N}}\rho_N\,\left(\udiv\,K(x_i-x_j) - \udiv K\star_x \bar\rho(x_i)\right)\,\ud X^N\,\ud s\\
& -\frac{\underline{\sigma}}{N}\,\sum_{i=1}^N\int_0^t\int_{\Pi^{d\,N}}\rho_N\,\left|\nabla_{x_i} \log \frac{\rho_N}{ \bar \rho_N}\right|^2+C_1\,t\,|\sigma-\sigma_N|,
\end{split}
\]
where we recall that $\bar\rho_N(t,X^N)=\Pi_{i=1}^N \bar\rho(t,x_i)$ and with
\[
C_1=\frac{1}{N\,t}\, \frac{2}{\underline{\sigma}}\int_{\Pi^{d\,N}} \rho_N^0\,\log\rho_N^0
+2\,\|\log\bar\rho\|_{W^{1,\infty}}^2 +\frac{\|\udiv K\|_{\dot W^{-1,\infty}}^2}{\underline{\sigma}^2 }+ \frac{ 2\|\udiv F\|_{L^\infty}}{\underline{\sigma}}.
\]\label{entropytimeevolution}
\end{lemma}
\begin{proof}
From the limiting equation \eqref{limitmeanfield}, one can readily check that $\log\bar\rho_N$ solves
\begin{equation}\label{eqlogbarrhoN}
\begin{split}
&\partial_t \log\bar\rho_N+\sum_{i=1}^N \frac{1}{N}\,\sum_{j=1}^N (F(x_i)+K(x_i-x_j))\cdot\nabla_{x_i} \log\bar\rho_N =\sum_{i=1}^N \sigma\,\frac{\Delta_{x_i} \bar\rho_N}{\bar\rho_N}\\
&\quad +\sum_{i=1}^N  \bigg(\frac{1}{N}\,\sum_{j=1}^N K(x_i-x_j)-K\star_{x}\bar\rho(x_i)\bigg)\cdot\nabla_{x_i} \log\bar\rho_N\\
&\quad -\sum_{i=1}^N (\udiv F(x_i)+\udiv K\star_x \bar\rho(x_i)).
\end{split}
\end{equation}
Remark that $\log\bar\rho_N\in W^{1,\infty}([0,\ T]\times\Pi^{d\,N})$ since $\bar\rho\in W^{1,\infty}([0,\ T]\times\Pi^{d})$ and $\bar\rho$ is bounded from below. Therefore $\log\bar\rho_N$ can be used as a test function against $\rho_N$ in Eq. \eqref{Liouville}. This implies that
\[\begin{split}
&\int_{\Pi^{d\,N}} \rho_N\,\log\bar\rho_N\,\ud X^N=\int_{\Pi^{d\,N}} \rho_N^0\,\log\bar\rho_N^0\,\ud X^N\\
&+\int_0^t\!\!\int_{\Pi^{d\,N}}\!\!\rho_N\bigg(\partial_t  \log \bar\rho_N\!+\! \frac{1}{N}\sum_{i,j=1}^N (F(x_i)\!+\!K(x_i\!-\!x_j))\cdot\nabla_{x_i} \log\bar\rho_N \bigg) \ud X^N\ud s\\
&-\sigma_N\sum_{i=1}^N \int_0^t\int_{\Pi^{d\,N}} \nabla_{x_i} \log\bar\rho_N\,\nabla_{x_i}\rho_N\,\ud X^N\,\ud s.
\end{split}\]
Using the equation \eqref{eqlogbarrhoN} on $\log\bar\rho_N$, we obtain 
\[\begin{split}
&\int_{\Pi^{d\,N}} \rho_N\,\log\bar\rho_N\,\ud X^N=\int_{\Pi^{d\,N}} \rho_N^0\,\log\bar\rho_N^0\,\ud X^N\\
&+\sum_{i=1}^N \int_0^t\int_{\Pi^{d\,N}}\rho_N\, \bigg(\frac{1}{N}\,\sum_{j=1}^N K(x_i-x_j)-K\star_{x}\bar\rho(x_i)\bigg)\cdot\nabla_{x_i} \log\bar\rho_N\,\ud X^N\,\ud s\\
&-\int_0^t\int_{\Pi^{d\,N}}\rho_N\,\sum_{i=1}^N (\udiv F(x_i)+\udiv K\star_x \bar\rho(x_i))\,\ud X^N\,\ud s\\
&+\sum_{i=1}^N \int_0^t\int_{\Pi^{d\,N}}\left(\sigma\,\rho_N\,\frac{\Delta_{x_i} \bar\rho_N}{\bar \rho_N}-\sigma_N \nabla_{x_i}\rho_N\cdot\frac{\nabla_{x_i} \bar\rho_N}{\bar\rho_N}\right)\,\ud X^N\,\ud s.
\end{split}\]
Using the entropy dissipation for $\rho_N$ given by \eqref{entropydissip}, we have that
\begin{equation}
\begin{split}
& {\cal H}_N(\rho_N\,|\;\bar\rho_N)(t)\leq {\cal H}_N(\rho_N\,|\;\bar\rho_N)(0) +\frac{1}{N}\,D_N\\
&-\frac{1}{N^2}\sum_{i,\;j=1}^N \int_0^t\int_{\Pi^{d\,N}}\rho_N\, \left( K(x_i-x_j)-K\star_{x}\bar\rho(x_i)\right)\cdot\nabla_{x_i} \log\bar\rho_N\,\ud X^N\,\ud s\\
&-\frac{1}{N^2}\sum_{i,\;j=1}^N\int_0^t\int_{\Pi^{d\,N}}\rho_N\,\left( \udiv\,K(x_i-x_j) - \udiv K\star_x \bar\rho(x_i)\right)\,\ud X^N\,\ud s,
\end{split}\label{interrelatentropy}
\end{equation}
with 
\[
D_N=\sum_{i=1}^N \int_0^t\int_{\Pi^{d\,N}}\left(-\sigma\,\rho_N\,\frac{\Delta_{x_i} \bar\rho_N}{\bar \rho_N}+\sigma_N \nabla_{x_i}\rho_N\cdot\frac{\nabla_{x_i} \bar\rho_N}{\bar\rho_N}-\sigma_N\,\frac{|\nabla_{x_i} \rho_N|^2}{\rho_N}\right).
\]
By integration by parts
\begin{equation}\label{1strelatentropydissip}
\begin{split}
&\int_{\Pi^{d\,N}}\left( - \rho_N\,\frac{\Delta_{x_i} \bar\rho_N}{\bar \rho_N} + \nabla_{x_i}\rho_N\cdot\frac{\nabla_{x_i} \bar\rho_N}{\bar\rho_N}-\frac{|\nabla_{x_i} \rho_N|^2}{\rho_N}\right)\\
&\quad =-\int_{\Pi^{d\,N}}\left(\rho_N\frac{|\nabla_{x_i}\bar\rho_N|^2}{\bar\rho_N^2}-2\,\nabla_{x_i}\rho_N\cdot\frac{\nabla_{x_i} \bar\rho_N}{\bar\rho_N}+\frac{|\nabla_{x_i} \rho_N|^2}{\rho_N} \right)\\
&\quad=-\int_{\Pi^{d\,N}}\rho_N\,\left|\nabla_{x_i} \log \frac{ \rho_N}{ \bar \rho_N}\right|^2.
\end{split}
\end{equation}
On the other hand,
\[\begin{split}
&(\sigma-\sigma_N)\,\sum_{i=1}^N \int_{\Pi^{d\,N}}\rho_N\,\frac{\Delta_{x_i} \bar\rho_N}{\bar \rho_N}\\
&\quad=(\sigma-\sigma_N)\,\sum_{i=1}^N\int_{\Pi^{d\,N}} \left(-\nabla_{x_i}\rho_N\cdot\frac{\nabla_{x_i}\bar\rho_N}{\bar\rho_N}+\rho_N\,\frac{|\nabla_{x_i}\bar\rho_N|^2}{\bar\rho_N^2}\right).
\end{split}
\]
Of course 
\[
\sum_{i=1}^N \int_{\Pi^{d\,N}} \rho_N\,\frac{|\nabla_{x_i}\bar\rho_N|^2}{\bar\rho_N^2}=\sum_{i=1}^N \int_{\Pi^{d\,N}} \rho_N\,\frac{|\nabla_{x_i}\bar\rho(x_i)|^2}{\bar\rho(x_i)^2}\leq N\,\|\log\bar\rho\|_{W^{1,\infty}}^2,
\]
while by Cauchy-Schwartz
\[\begin{split}
&
\sum_{i=1}^N \int_0^t\int_{\Pi^{d\,N}} \nabla_{x_i}\rho_N\cdot\frac{\nabla_{x_i}\bar\rho_N}{\bar\rho_N}\leq N\,t\,\|\log\bar\rho\|_{W^{1,\infty}}^2+\sum_{i=1}^N \int_0^t\int_{\Pi^{d\,N}} \frac{|\nabla_{x_i}\rho_N|^2}{\rho_N}\\
&\leq N\,t\,\|\log\bar\rho\|_{W^{1,\infty}}^2 + \frac{2}{ \underline{\sigma}}\int_{\Pi^{d\,N}}\!\! \rho_N^0\,\log\rho_N^0+\frac{N\,t\,\|\udiv K\|_{\dot W^{-1,\infty}}^2}{\,\underline{\sigma}^2}\\
&\qquad+N\,t\, \frac{ 2 \|\udiv F\|_{L^\infty}}{\underline{\sigma}},
\end{split}
\]
by Prop. \ref{entropybounds} based on the entropy dissipation.
 
This leads to 
\begin{equation}
\begin{split}
&(\sigma-\sigma_N)\,\sum_{i=1}^N \int_0^t \int_{\Pi^{d\,N}}\rho_N\,\frac{\Delta_{x_i} \bar\rho_N}{\bar \rho_N}\\
  &\quad \leq |\sigma-\sigma_N|\,\bigg(N\,t\Big[ 2\,\|\log\bar\rho\|_{W^{1,\infty}}^2 +\frac{\|\udiv K\|_{\dot W^{-1,\infty}}^2}{\underline{\sigma}^2}+ \frac{2 \|\udiv F\|_{L^\infty} }{\underline{\sigma}}\Big]\\
  &\qquad+ \frac{2}{\underline{\sigma}}\int_{\Pi^{d\,N}} \rho_N^0\,\log\rho_N^0\bigg).
\end{split}\label{sigma-sigmaN}
\end{equation}
Finally combining \eqref{sigma-sigmaN} with \eqref{1strelatentropydissip}
\[\begin{split}
&D_N\leq -\underline{\sigma} \sum_{i=1}^N \int_0^t\int_{\Pi^{d\,N}}\rho_N\,\left|\nabla_{x_i} \log \frac{\rho_N}{ \bar \rho_N}\right|^2+|\sigma-\sigma_N|\,\bigg( \frac{2}{\underline{\sigma}}\int_{\Pi^{d\,N}} \rho_N^0\,\log\rho_N^0\\
&\quad+N\,t\Big[ 2\,\|\log\bar\rho\|_{W^{1,\infty}}^2 +\frac{\|\udiv K\|_{\dot W^{-1,\infty}}^2}{\underline{\sigma}^2 }+ \frac{2 \|\udiv F\|_{L^\infty}}{\underline{\sigma}}\Big] \bigg),
\end{split}\]
which inserted in \eqref{interrelatentropy} concludes the proof.
\end{proof}
%
%%%%%%%%%%%%%%%%%%%%%%%%%%%%%%%%%%%%%%%%%%%%%%%%%%%%%%%%%%%%%%%%%
\subsection{Bounding the interaction terms: The bounded divergence term}
%%%%%%%%%%%%%%%%%%%%%%%%%%%%%%%%%%%%%%%%%%%%%%%%%%%%%%%%%%%%%%%%%%%%%
We now have to obtain the main estimates, starting with the case where the kernel belongs to  $\dot W^{-1,\infty}(\Pi^d)$ and has bounded divergence.
\begin{lemma}
Assume that $\bar\rho \in W^{2,p}(\Pi^{d})$ for any $p<\infty$, then for any kernel $L\in \dot W^{-1,\infty}(\Pi^d)$ with $\udiv L\in L^\infty$, one has that
\[\begin{split}
&-\frac{1}{N^2}\,\sum_{i,\;j=1}^N \int_{\Pi^{d\,N}}\rho_N\, \left( L(x_i-x_j)-L\star_{x}\bar\rho(x_i)\right)\cdot\nabla_{x_i} \log\bar\rho_N\,\ud X^N\\
&-\frac{1}{N^2}\,\sum_{i,\;j=1}^N \int_{\Pi^{d\,N}}\rho_N\, \left( \udiv L(x_i-x_j)-\udiv L\star_{x}\bar\rho(x_i)\right)\,\ud X^N\\
&\qquad\leq \frac{\underline{\sigma}}{4\,N}\sum_{i=1}^N \int_{\Pi^{d\,N}} \rho_N\, |\nabla_{x_i} \log \frac{\rho_N}{\bar\rho_N}|^2\,\ud X^N
+C\,M_L^1 \left({\cal H}_N(\rho_N\,|\bar\rho_N)+\frac{1}{N}\right),
\end{split}
\]
where $C$ is a universal constant and
\[
M^1_L=d^3\,\frac{\|\bar\rho\|_{W^{1,\infty}}^2\,\|L\|_{\dot W^{-1,\infty}}^2}{\underline{\sigma}\,(\inf \bar\rho)^2}+\frac{\|L\|_{\dot W^{-1,\infty}}}{\inf \bar\rho}\, \sup_{p \geq 1}\frac{\|\nabla^2\bar\rho\|_{L^p}}{p}+\|\udiv L\|_{L^\infty}.
\]
\label{interactionboundI}
\end{lemma}
\begin{proof}
Remark that in this estimate, time is now only a fixed parameter and will hence not be specified in this proof. 

Denote $V\in L^\infty(\Pi^d)$ s.t. $L=\udiv V$ or using coordinates $L_\alpha=\sum_\beta \partial_\beta V_{\alpha\beta}$. By the definition of $\dot W^{-1,\infty}$ we assume that $\|V\|_{L^\infty}\leq 2\,\|L\|_{\dot W^{-1, \infty}}.$ Rewriting
\[\begin{split}
&-\frac{1}{N^2}\,\sum_{i,\;j=1}^N \int_{\Pi^{d\,N}}\rho_N\, \left( L(x_i-x_j)-L\star_{x}\bar\rho(x_i)\right)\cdot\nabla_{x_i} \log\bar\rho_N\,\ud X^N\\
&=-\frac{1}{N^2}\,\sum_{\alpha\beta}\sum_{i,\;j=1}^N \int_{\Pi^{d\,N}}\!\!\left( \partial_{x_i^\beta}V_{\alpha\beta}(x_i\!-\!x_j)-\partial_{x_i^\beta} V_{\alpha\beta}\star_{x}\bar\rho(x_i)\right)\,\frac{\rho_N}{\bar\rho_N}\, \partial_{x_i^\alpha}\bar\rho_N\,\ud X^N.\\
\end{split}
  \]
  By integration by parts, this is equal to
\[\begin{split}
&=-\frac{1}{N^2}\,\sum_{\alpha\beta}\sum_{i,\;j=1}^N \int_{\Pi^{d\,N}}\!\!\left( V_{\alpha\beta}(x_i-x_j)- V_{\alpha\beta}\star_{x}\bar\rho(x_i)\right)\,\frac{\rho_N}{\bar\rho_N}\, \partial^2_{x_i^\alpha\,x_i^\beta}\bar\rho_N\,\ud X^N\\
&\ -\frac{1}{N^2}\,\sum_{\alpha\beta}\sum_{i,\;j=1}^N \int_{\Pi^{d\,N}}\!\!\left( V_{\alpha\beta}(x_i-x_j)- V_{\alpha\beta}\star_{x}\bar\rho(x_i)\right)\,\partial_{x_i^\beta}\,\frac{\rho_N}{\bar\rho_N}\, \partial_{x_i^\alpha}\bar\rho_N\,\ud X^N.\\
\end{split}
  \]
When one adds the divergence term, one obtains in tensor form   
\[\begin{split}
&-\frac{1}{N^2}\,\sum_{i,\;j=1}^N \int_{\Pi^{d\,N}}\rho_N\, \left( L(x_i-x_j)-L\star_{x}\bar\rho(x_i)\right)\cdot\nabla_{x_i} \log\bar\rho_N\,\ud X^N\\
&-\frac{1}{N^2}\,\sum_{i,\;j=1}^N \int_{\Pi^{d\,N}}\rho_N\, \left( \udiv L(x_i-x_j)-\udiv L\star_{x}\bar\rho(x_i)\right)\,\ud X^N=A+B,
\end{split}
\]
with
\[\begin{split}
& A=\frac{1}{N^2}\,\sum_{i,\;j=1}^N \int_{\Pi^{d\,N}}\left( V(x_i-x_j)-V\star_{x}\bar\rho(x_i)\right)\,:\, \nabla_{x_i}\bar\rho_N  \otimes \nabla_{x_i}\frac{\rho_N}{\bar\rho_N}\,\ud X^N,\\
& B=\frac{1}{N^2}\,\sum_{i,\;j=1}^N \int_{\Pi^{d\,N}}\rho_N\, \Big[\left( V(x_i-x_j)-V\star_{x}\bar\rho(x_i)\right)\,:\,\frac{\nabla_{x_i}^2 \bar\rho_N}{\bar\rho_N}\\
&\qquad\qquad-\udiv L(x_i-x_j)+\udiv L\star_{x}\bar\rho(x_i)\Big]\,\ud X^N.
\end{split}
\]
We treat independently $A$ and $B$. 

\bigskip

\noindent{\em The bound on $A$.}
First by Cauchy-Schwartz and by using $a\,b\leq a^2/4+b^2$
\[
\begin{split}
& A\leq \frac{\underline{\sigma}}{4\,N}\sum_{i=1}^N \int_{\Pi^{d\,N}} \frac{\bar \rho_N^2}{\rho_N}\, |\nabla_{x_i} \frac{\rho_N}{\bar\rho_N}|^2\,\ud X^N\\
&\ +\frac{d}{N\,\underline{\sigma}}\sum_{i=1}^N \int_{\Pi^{d\,N}}\bigg(\frac{1}{N}\sum_{j=1}^N \left( V(x_i-x_j)-V\star_{x}\bar\rho(x_i)\right)\bigg)^2\,\left|\frac{\nabla_{x_i} \bar\rho_N}{\bar\rho_N}\right|^2\,\rho_N\,\ud X^N.
\end{split}
\]
Remark that 
\[
\left|\frac{\nabla_{x_i} \bar\rho_N}{\bar\rho_N}\right|^2=|\nabla_{x_i}\log \bar\rho(x_i)|^2\leq \frac{\|\bar\rho\|_{W^{1,\infty}}^2}{(\inf \bar\rho)^2}.
\]
Hence one has that
\begin{equation}
\begin{split}
& A\leq \frac{\underline{\sigma}}{4\,N}\sum_{i=1}^N \int_{\Pi^{d\,N}} \rho_N\, |\nabla_{x_i} \log \frac{\rho_N}{\bar\rho_N}|^2\,\ud X^N\\
  + & \frac{ d \|\bar\rho\|_{W^{1,\infty}}^2}{N\,\underline{\sigma}\,(\inf \bar\rho)^2}\, \sum_{i=1}^N \sum_{\alpha,\beta=1}^d \int_{\Pi^{d\,N}}\!\!\bigg(\frac{1}{N}\sum_{j=1}^N \left( V_{\alpha,\beta}(x_i\!-\!x_j)-V_{\alpha,\beta}\star_{x}\bar\rho(x_i)\right)\bigg)^2\\
  &\hskip9cm \,\rho_N\,\ud X^N,
\end{split}\label{firstboundA}
\end{equation}
where $V_{\alpha,\beta}$ is the corresponding coordinate of the matrix field $V$. 

For some $\eta>0$ to be chosen later, we apply Lemma \ref{youngconvex} with 
\[
\Phi=\bigg(\frac{1}{N}\sum_{j=1}^N \eta\,\left( V_{\alpha,\beta}(x_i-x_j)-V_{\alpha,\beta}\star_{x}\bar\rho(x_i)\right)\bigg)^2,
\]
to find
\begin{equation}
\begin{split}
& \frac{1}{N}\sum_{i=1}^N \sum_{\alpha,\beta=1}^d  \int_{\Pi^{d\,N}}\bigg(\frac{1}{N}\sum_{j=1}^N \left( V_{\alpha,\beta}(x_i-x_j)-V_{\alpha,\beta} \star_{x}\bar\rho(x_i)\right) \bigg)^2 \, \rho_N\,\ud X^N\\
&\quad\leq \frac{d^2}{\eta^2}\,{\cal H}_N(\rho_N\,|\;\bar\rho_N)\\
&\quad+\frac{1}{N^2\,\eta^2}\sum_{i=1}^N \sum_{\alpha,\beta=1}^N \log \int_{\Pi^{d\,N}} \bar\rho_N\,e^{N\,\left(\frac{1}{N}\sum_j \eta\,\left( V_{\alpha,\beta}(x_i-x_j)-V_{\alpha,\beta}\star_{x}\bar\rho(x_i)\right)\right)^2} \,\ud X^N.
\end{split}\label{youngforA}
\end{equation}
By symmetry
\[
\begin{split}
&\frac{1}{N}\sum_{i=1}^N \log \int_{\Pi^{d\,N}} \bar\rho_N\,e^{N\,\left(\frac{1}{N}\sum_j \eta\,\left( V_{\alpha,\beta}(x_i-x_j)-V_{\alpha,\beta}\star_{x}\bar\rho(x_i)\right)\right)^2} \,\ud X^N\\
&\quad=\log \int_{\Pi^{d\,N}} \bar\rho_N\,e^{N\,\left(\frac{1}{N}\sum_j  \eta \left( V_{\alpha,\beta}(x_1-x_j)-V_{\alpha,\beta}\star_{x}\bar\rho(x_1) \right)\right)^2} \,\ud X^N.
\end{split}
\]
Define $\psi(z,x)=\eta\,V_{\alpha, \beta}(z-x)-\eta\,V_{\alpha, \beta} \star\bar\rho(z)$. Choose $\eta=1/(4\,e\,\|V\|_{L^\infty})$ and note that $\|\psi\|_{L^\infty}\leq \frac{1}{4\,e}$ and that for a fixed $z$, $\int\bar\rho(x)\,\psi(z,x)\,dx=0$. Since
\[\begin{split}
&N\,\bigg(\frac{1}{N}\sum_{j=1}^N \eta\,\left( V_{\alpha, \beta}(x_1-x_j)-V_{\alpha, \beta}\star_{x}\bar\rho(x_1)\right)\bigg)^2\\
&=\frac{1}{N}\sum_{j_1,j_2 =1}^N \psi(x_1,x_{j_1})\,\psi(x_1,x_{j_2}),
\end{split}
\]
we may apply Theorem \ref{MEI} to obtain that
\[
\int_{\Pi^{d\,N}} \bar\rho_N\,e^{N\,\left(\frac{1}{N}\sum_j  \eta \left( V_{\alpha,\beta}(x_1-x_j)-V_{\alpha,\beta}\star_{x}\bar\rho(x_1)\right)\right)^2} \,\ud X^N\leq C,
\]
for some explicit universal constant $C$.

Combining \eqref{firstboundA}-\eqref{youngforA} with this bound yields the final estimate on $A$
\begin{equation}
\begin{split}
& A\leq \frac{\underline{\sigma}}{4\,N}\sum_{i=1}^N \int_{\Pi^{d\,N}} \rho_N\, |\nabla_{x_i} \log \frac{\rho_N}{\bar\rho_N}|^2\,\ud X^N\\
&+C\,d^3\,\frac{\|\bar\rho\|_{W^{1,\infty}}^2\,\|V\|_{L^\infty}^2}{\underline{\sigma}\,(\inf \bar\rho)^2}\, \left({\cal H}_N(\rho_N\,|\bar\rho_N)+\frac{1}{N}\right), \end{split}\label{finalboundA}
\end{equation}
again for some universal constant $C$. 

\bigskip

\noindent{\em The bound on $B$.} Define
\begin{equation}
\phi(x,z)=\left( V(x-z)-V\star_{x}\bar\rho(x)\right)\,:\,\frac{\nabla_{x}^2 \bar\rho(x)}{\bar\rho(x)}-\udiv L(x-z)+\udiv L\star_{x}\bar\rho(x),
\label{defphi1}
\end{equation}
so that
\[
\begin{split}
B&=\frac{1}{N^2}\sum_{i,j=1}^N \int_{\Pi^{d\,N}}\rho_N\,\bigg[\left( V(x_i-x_j)-V\star_{x}\bar\rho(x_i)\right)\,:\,\frac{\nabla_{x_i}^2 \bar\rho_N}{\bar\rho_N}\\
&\qquad\qquad-\udiv L(x_i-x_j)+\udiv L\star_{x}\bar\rho(x_i)\bigg]\,\ud X^N\\
&= \frac{1}{N^2}\sum_{i,j=1}^N \int_{\Pi^{d\,N}}\rho_N\, \phi(x_i,x_j)\,\ud X^N.
\end{split}
\]
Apply Lemma \ref{youngconvex} with
\[
\Phi=\frac{1}{N^2}\sum_{i,j=1}^N \eta\,\phi(x_i,x_j),
\]
so that
\begin{equation}
\begin{split}
B&\leq \frac{1}{\eta}\,{\cal H}_N(\rho_N\,|\;\bar\rho_N)
+\frac{1}{N\,\eta}\int_{\Pi^{d\,N}}\bar\rho_N\,e^{\frac{1}{N}\sum_{i,j} \eta\,\phi(x_i,x_j)}\,\ud X^N. 
\end{split} \label{firstboundB}
\end{equation}
Observe that
$\int_{\Pi^d} \phi(x,z)\,\bar\rho(z) \ud z=0$.
While by integration by parts
\[\begin{split}
&\int_{\Pi^d} \left( V(x-z)-V\star_{x}\bar\rho(x)\right)\,:\,\frac{\nabla_{x}^2 \bar\rho(x)}{\bar\rho(x)}\,\bar\rho(x)\ud x\\
&\qquad=\int_{\Pi^d} \left( \udiv L(x-z)-\udiv L\star_{x}\bar\rho(x)\right)\,\bar\rho(x) \ud x, 
\end{split}
\]
implying that $\int_{\Pi^d} \phi(x,z)\,\bar\rho(x)\, \ud x=0$. Note as well from \eqref{defphi1} that
\[
\|\sup_z |\phi(.,z)|\|_{L^p(\bar\rho\,dx)}\leq 2 \,\frac{\|V\|_{L^\infty}}{\inf \bar\rho}\, \|\nabla^2\bar\rho\|_{L^p}+2 \,\|\udiv L\|_{L^\infty}.
\]
Hence choosing
\[
\eta=\frac{1}{C\,\left(\frac{\|V\|_{L^\infty}}{\inf \bar\rho}\,\sup_p \frac{\|\nabla^2\bar\rho\|_{L^p}}{p}+\|\udiv L\|_{L^\infty}\right)},
\]
we may apply Theorem \ref{MEII} to bound
\[
\int_{\Pi^{d\,N}}\bar\rho_N\,e^{\frac{1}{N}\sum_{i,j} \eta\,\phi(x_i,x_j)}\,\ud X^N\leq C,
\]
for some universal constant $C$.
Hence from \eqref{firstboundB}, we conclude that
\begin{equation}
B\leq C\,\left(\frac{\|V\|_{L^\infty}}{\inf \bar\rho}\, \sup_p\frac{\|\nabla^2\bar\rho\|_{L^p}}{p}+\|\udiv L\|_{L^\infty}\right)\,\left({\cal H}_N(\rho_N\,|\;\bar\rho_N)+\frac{1}{N} \right).
\label{finalboundB}
\end{equation}

\bigskip

To finish the proof of the lemma, we simply have to add \eqref{finalboundA} and \eqref{finalboundB}, recalling that $\|V\|_{L^\infty} \leq 2 \|L\|_{\dot W^{-1,\infty}}$.
\end{proof}

%%%%%%%%%%%%%%%%%%%%%%%%%%%%%%%%%%%%%%%%%%%%%%%%%%%%%%%%%%%%%%%%%
\subsection{Bounding the interaction terms: The divergence term only in $\dot W^{-1,\infty}$  }
%%%%%%%%%%%%%%%%%%%%%%%%%%%%%%%%%%%%%%%%%%%%%%%%%%%%%%%%%%%%%%%%%%%%%
\begin{lemma}
Assume that $\bar\rho \in W^{1,p}(\Pi^{d})$ for any $p<\infty$, then for any kernel $L\in L^\infty(\Pi^d)$ with $\udiv L\in \dot W^{-1,\infty}$, one has that
\[\begin{split}
&-\frac{1}{N^2}\,\sum_{i,\;j=1}^N \int_{\Pi^{d\,N}}\rho_N\, \left( L(x_i-x_j)-L\star_{x}\bar\rho(x_i)\right)\cdot\nabla_{x_i} \log\bar\rho_N\,\ud X^N\\
&-\frac{1}{N^2}\,\sum_{i,\;j=1}^N \int_{\Pi^{d\,N}}\rho_N\, \left( \udiv L(x_i-x_j)-\udiv L\star_{x}\bar\rho(x_i)\right)\,\ud X^N\\
&\qquad\leq \frac{\underline{\sigma}}{4\,N}\sum_{i} \int_{\Pi^{d\,N}} \rho_N\, |\nabla_{x_i} \log \frac{\rho_N}{\bar\rho_N}|^2\,\ud X^N
+C\,M_L^2\, \left({\cal H}_N(\rho_N\,|\bar\rho_N)+\frac{1}{N}\right),
\end{split}
\]
where $C$ is a universal constant and
\[
M^2_L=  \left(\|L\|_{L^\infty}+ \|\udiv L\|_{\dot W^{-1,\infty}}\right)\,\frac{\|\nabla\bar\rho\|_{L^\infty}}{\inf \bar\rho} +  \frac{d }{\underline{\sigma}} \|\udiv L \|_{ \dot W^{-1, \infty}}^2.
\]
\label{interactionboundII}
\end{lemma}
\begin{proof}
The proof follows similar ideas to the proof of Lemma \ref{interactionboundI} but now we have to integrate by parts the term with $\udiv L$ instead of the term with $L$. Denote $\tilde L\in L^\infty$ s.t. $\udiv \tilde L=\udiv L$ and $\|\udiv L\|_{\dot W^{-1,\infty}}=\|\tilde L\|_{L^\infty}$. Write
\[\begin{split}
&-\frac{1}{N^2}\,\sum_{i,\;j=1}^N \int_{\Pi^{d\,N}}\rho_N\, \left( \udiv L(x_i-x_j)-\udiv L\star_{x}\bar\rho(x_i)\right)\,\ud X^N\\
&\qquad=\frac{1}{N^2}\,\sum_{i,\;j=1}^N \int_{\Pi^{d\,N}}\nabla_{x_i}\frac{\rho_N}{\bar\rho_N}\cdot \left( \tilde L(x_i-x_j)-\tilde L\star_{x}\bar\rho(x_i)\right)\,\bar\rho_N\ud X^N\\
&\qquad\quad+\frac{1}{N^2}\,\sum_{i,\;j=1}^N \int_{\Pi^{d\,N}}\rho_N\, \left( \tilde L(x_i-x_j)-\tilde L\star_{x}\bar\rho(x_i)\right)\cdot\nabla_{x_i}\log \bar\rho_N\ud X^N.
\end{split}
\]
Hence 
\begin{equation}\begin{split}
&-\frac{1}{N^2}\,\sum_{i,\;j=1}^N \int_{\Pi^{d\,N}}\rho_N\, \left( L(x_i-x_j)-L\star_{x}\bar\rho(x_i)\right)\cdot\nabla_{x_i} \log\bar\rho_N\,\ud X^N\\
&-\frac{1}{N^2}\,\sum_{i,\;j=1}^N \int_{\Pi^{d\,N}}\rho_N\, \left( \udiv L(x_i-x_j)-\udiv L\star_{x}\bar\rho(x_i)\right)\,\ud X^N=A+B,\\
\end{split}\label{introAB}
\end{equation}
with
\[
A=\frac{1}{N^2}\,\sum_{i,\;j=1}^N \int_{\Pi^{d\,N}}\nabla_{x_i}\frac{\rho_N}{\bar\rho_N}\cdot \left( \tilde L(x_i-x_j)-\tilde L\star_{x}\bar\rho(x_i)\right)\,\bar\rho_N\ud X^N,
\]
and
\[
B=\frac{1}{N^2}\,\sum_{i,\;j=1}^N \int_{\Pi^{d\,N}}\rho_N\, \left( \bar L(x_i-x_j)-\bar L\star_{x}\bar\rho(x_i)\right)\cdot\nabla_{x_i} \log\bar\rho_N\,\ud X^N,
\]
for $\bar L=\tilde L-L$.

\medskip

\noindent{\em Bound for $A$.} We start with Cauchy-Schwartz to bound
\[\begin{split}
A\leq &\frac{\underline{\sigma}}{4\,N}\,\sum_{i=1}^N \int_{\Pi^{d\,N}}|\nabla_{x_i}\frac{\rho_N}{\bar\rho_N}|^2\,\frac{\bar\rho_N^2}{\rho_N}\\
&+\frac{1}{N\,\underline{\sigma}}\,\sum_{i=1}^N \sum_{\alpha =1}^d\int_{\Pi^{d\,N}}\rho_N\,\bigg|\frac{1}{N}\sum_{j=1}^N (\tilde L_\alpha(x_i-x_j)-\tilde L_\alpha\star_x\bar\rho(x_i)\bigg|^2\,\ud X^N,
\end{split}
\]
where $\tilde L_\alpha$ is the $\alpha$ coordinate of $\tilde L$.

Denote $\psi(z,x)=\eta\,(\tilde L_\alpha(z-x)-\tilde L_\alpha\star \bar\rho(z))$, and use Lemma \ref{youngconvex} for $\Phi=\left|\frac{1}{N}\sum_{j=1}^N \psi(x_i,x_j)\right|^2$ to obtain
\[\begin{split}
&\frac{1}{N}\,\sum_{i=1}^N \int_{\Pi^{d\,N}}\rho_N\,\bigg|\frac{1}{N}\sum_{j=1}^N (\tilde L_\alpha(x_i-x_j)-\tilde L_\alpha\star_x\bar\rho(x_i)\bigg|^2\,\ud X^N\\
&\quad\leq \frac{1}{\eta^2}\,{\cal H}_N(\rho_N\,|\;\bar\rho_N)+\frac{1}{N^2\,\eta^2}\,\sum_{i=1}^N  \log \int_{\Pi^{d\,N}} \bar\rho_N\,e^{\left|\frac{1}{N}\,\sum_{j} \psi(x_i,x_j)\right|^2}\,\ud X^N.
\end{split}
\]
Of course $\int_{\Pi^d} \psi(z,x)\,\bar\rho(x)\,dx=0$ so that taking
\[
\eta=\frac{1}{4\,e\,\|\tilde L\|_{L^\infty}}=\frac{1}{4\,e\,\|\udiv L\|_{\dot W^{-1,\infty}}},
\]
and applying Theorem \ref{MEI}, we find
\begin{equation}\begin{split}
A\leq &\frac{\underline{\sigma}}{4\,N}\,\sum_{i=1}^N \int_{\Pi^{d\,N}} \rho_N |\nabla_{x_i}\log\frac{\rho_N}{\bar\rho_N}|^2 \ud X^N \\
&+C\,d\,\frac{\|\udiv L\|_{\dot W^{-1,\infty}}^2}{\underline\sigma}\,\left({\cal H}_N(\rho_N\,|\;\bar\rho_N)+\frac{1}{N}\right).\label{boundAII}
\end{split}\end{equation}

\medskip

\noindent{\em Bound for $B$.} We follow the same steps as before, define
\[
\phi(x,z)=\left( \bar L(x-z)-\bar L\star_{x}\bar\rho(x)\right)\cdot\nabla_{x} \log\bar\rho(x),
\]
and first apply Lemma \ref{youngconvex} with 
$\Phi=\frac{\eta}{N^2}\sum_{i,j=1}^N \phi(x_i,x_j)$
to find
\[
B\leq \frac{1}{\eta}{\cal H}_N(\rho_N\,|\;\bar\rho_N)+\frac{1}{N\,\eta}\  \log  \int_{\Pi^{d\,N}} \bar\rho_N\,e^{\frac{1}{N}\,\sum_{i,j} \phi(x_i,x_j)}\,\ud X^N.
\]
Since $\udiv \bar L=\udiv \tilde L-\udiv L=0$, we have that 
\[
\int_{\Pi^d} \phi(x,z)\,\bar\rho(z)\,dz=\int_{\Pi^d} \phi(x,z)\,\bar\rho(x)\,dx=0.
\]
Choose
\[
\eta=\frac{1}{C\,\|\bar L\|_{L^\infty}\,\sup_p \frac{\|\nabla \log \bar\rho\|_{L^p(\bar \rho \ud x)}}{p}}=\frac{\inf\bar\rho}{C\,(\|L\|_{L^\infty}+\|\udiv L\|_{\dot W^{-1,\infty}})\,\|\nabla\bar\rho\|_{L^\infty}},
\]
and apply now Theorem \ref{MEII} to conclude that
\begin{equation}
B\leq C\,(\|L\|_{L^\infty}+\|\udiv L\|_{\dot W^{-1,\infty}})\,\frac{\|\nabla\bar\rho\|_{L^\infty}}{\inf \bar\rho}\;\left({\cal H}_N(\rho_N\,|\;\bar\rho_N)+\frac{1}{N}\right).\label{boundBII}
\end{equation}

\medskip

Combining \eqref{boundAII} and \eqref{boundBII} concludes the proof.
\end{proof}
%%%%%%%%%%%%%%%%%%%%%%%%%%%%%%%%%%%%%%%%%
\subsection{Conclusion of the proof of Theorem \ref{maintheorem}}
%%%%%%%%%%%%%%%%%%%%%%%%%%%%%%%%%
The proof of Theorem \ref{maintheorem} follows from the previous estimates through a careful decomposition of the kernel $K$.

By the assumption of Theorem \ref{maintheorem}, we have that $K_\alpha=\partial_\beta  V_{\alpha\beta}$ where $V\in L^\infty(\Pi^d)$ is a matrix field, and  that  there exists $\tilde K\in L^\infty$ s.t. $\udiv K=\udiv \tilde K$ and $\|\tilde K\|_{L^\infty(\Pi^d)}\leq 2\,\|\udiv K\|_{\dot W^{-1,\infty}}$. For convenience, we use the notation
\[
\|K\|=\|\tilde K\|_{L^\infty(\Pi^d)}+\|V\|_{L^\infty(\Pi^d)}\leq 2\,\|\udiv K\|_{\dot W^{-1,\infty}}+2\,\|K\|_{\dot W^{-1,\infty}}.
\]
%First of all, solve $\Delta \phi=\udiv\,\udiv V=\sum_{\alpha,\beta} \partial_{\alpha,\beta} V_{\alpha,\beta}$. 

Define $\bar K=\udiv V-\tilde K$. Note that $\udiv \bar K=0$ and obviously since $\tilde K\in L^\infty$ and we can choose $\tilde K$ s.t. $\int \tilde K=0$, then $\bar K\in \dot W^{-1,\infty}$ with $\| \bar K\|_{\dot W^{-1,\infty}}\leq C_d\,\|K\|$. 

We combine Lemma \ref{entropytimeevolution} with Lemma \ref{interactionboundI} for $L=\bar K$, and finally with Lemma \ref{interactionboundII} for $L=\tilde K$. We obtain
\begin{equation}
\begin{split}
 & {\cal H}_N(\rho_N\,|\;\bar\rho_N)(t)\leq {\cal H}_N(\rho_N^0\,|\;\bar\rho_N^0)+C_1\,t\,|\sigma-\sigma_N|
\\
&\qquad+C \,\int_0^t\left(M_{\bar K}^1+M_{\tilde K}^2\right)\,\left( {\cal H}_N(\rho_N\,|\;\bar\rho_N)(s)+\frac{1}{N}\right)\,ds.
\end{split}\label{finalentropyboundI}
\end{equation}
With our specific bounds
\[\begin{split}
& M_{\bar K}^1\leq d^3\,\|\bar K\|^2_{\dot W^{-1,\infty}}\,\frac{\|\bar\rho\|_{W^{1,\infty}}^2}{\underline{\sigma}\,(\inf \bar\rho)^2}+\frac{\|\bar K\|_{\dot W^{-1,\infty}}}{\inf \bar\rho}\,\sup_p \frac{\|\nabla^2 \bar\rho\|_{L^p}}{p},\\
& M_{\tilde K}^2\leq \left( \|\tilde K\|_{L^\infty} + \|\udiv \tilde K \|_{\dot W^{-1, \infty}}  \right) \frac{\|\nabla\bar\rho\|_{L^\infty}}{\inf \bar\rho} + \frac{d}{\underline{\sigma}} \| \udiv \tilde{K}\|_{\dot W^{-1, \infty}}^2. 
\end{split}
\]
To keep calculations simple, we do not try here to obtain fully explicit bounds (which would still be possible) and simplify \eqref{finalentropyboundI} in
\begin{equation}
\begin{split}
 & {\cal H}_N(\rho_N\,|\;\bar\rho_N)(t)\leq {\cal H}_N(\rho_N^0\,|\;\bar\rho_N^0)+\bar M\,(1+t\,(1+\|K\|^2))\,|\sigma-\sigma_N|\\
&+
\bar M\,(\|K\|+\|K\|^2)\,\int_0^t \left( {\cal H}_N(\rho_N\,|\;\bar\rho_N)(s)+\frac{1}{N}\right),
\end{split}\label{finalentropyboundII}
\end{equation} 
where we only kept explicit a simplified dependence on $K$ and where the constant $\bar M$ depends only on
\[\begin{split}
\bar M\bigg(&d,\;\underline{\sigma},\;\inf \bar\rho,\;\|\bar\rho\|_{W^{1,\infty}},\;\sup_p \frac{\|\nabla^2 \bar\rho\|_{L^p}}{p},\;\frac{1}{N}\int_{\Pi^{d\,N}} \rho_N^0\,\log \rho_N^0,\;\|\udiv F\|_{L^\infty}\Bigg).\\
%&\frac{1}{N}\int_{\Pi^{d\,N}}\rho_N^0\,\log\rho_N^0\bigg). 
\end{split}
\]
By Gronwall lemma, \eqref{finalentropyboundII} implies that
\[\begin{split}
 {\cal H}_N(\rho_N\,|\;\bar\rho_N)(t)\leq &e^{\bar M\,(\|K\|+\|K\|^2)\,t}\,
\bigg({\cal H}_N(\rho_N^0\,|\;\bar\rho_N^0)+\frac{1}{N}\\
&\qquad+\bar M(1+t\,(1+\|K\|^2))\,|\sigma-\sigma_N| \bigg),
\end{split}
\]
which concludes the proof of Theorem \ref{maintheorem}.
%%%%%%%%%%%%%%%%%%%%%%%%%%%%%%%%%%%%%%%%%%%%%%%%%%%%%%%%%%
\subsection{Proof of Theorem \ref{thmvanishingviscosity}} \label{ProofTh2}
%%%%%%%%%%%%%%%%%%%%%%%%%%%%%%%%%%%%%%%%%%%%%%%%%%%%%%%%%
The proof of our result for vanishing viscosity is in fact now straightforward as it uses our previous analysis.

First of all, we have an direct equivalent of Lemma \ref{entropytimeevolution} 
%by taking $\underline{\sigma}=0$ as the diffusion is not bounded from below anymore
\begin{equation}\begin{split}
&{\cal H}_N(\rho_N\,|\;\bar\rho_N)(t) = \frac{1}{N}\int_{\Pi^{d\,N}}\, \rho_N(t,X^N)\,\log  \frac{\rho_N(t,X^N)}{\bar\rho_N(t,X^N)}\,\ud X^N\leq {\cal H}_N(\rho_N^0\,|\;\bar\rho_N^0)\\
& -\frac{1}{N^2}\,\sum_{i,\;j=1}^N \int_0^t\int_{\Pi^{d\,N}}\rho_N\, \left( K(x_i-x_j)-K\star_{x}\bar\rho(x_i)\right)\cdot\nabla_{x_i} \log\bar\rho_N\,\ud X^N\,\ud s\\
&-\frac{1}{N^2}\,\sum_{i,\;j=1}^N \int_0^t\int_{\Pi^{d\,N}}\rho_N\,\left(\udiv\,K(x_i-x_j) - \udiv K\star_x \bar\rho(x_i)\right)\,\ud X^N\,\ud s \, + \alpha_N \\
%& +C_2\,t\,\frac{|\sigma-\sigma_N|^2}{\sigma\,\sigma_N},
\end{split}\label{timeevolution2}
\end{equation}
where when  $\sigma_N \to \sigma=0$, 
\[
\alpha_N = \frac{\sigma_N}{4 N} \sum_{i=1}^N \int_0^t \int_{\Pi^{d \, N}}  \rho_N |\nabla_{x_i} \log \bar  \rho(x_i)|^2 \ud X^N \ud s \leq t |\sigma -\sigma_N| \, \| \log \bar \rho \|_{W^{1, \infty}}^2,  
\]
while when $\sigma_N \to \sigma >0$ as $N \to \infty$, we can take $\underline{\sigma} = \sigma/2$ as in Lemma \ref{entropytimeevolution} but use the entropy bound \eqref{entropybounds2} which gives 
\[
\alpha_N = C_2\,t\, |\sigma-\sigma_N|  
\]
with $C_2$ given by
\[
C_2=\frac{2}{ \sigma N\,t}\, \int_{\Pi^{d\,N}} \rho_N^0\,\log\rho_N^0
+ \frac{2}{\sigma} \|\udiv K\|_{L^\infty}+ \frac{2}{\sigma}\|\udiv F\|_{L^\infty} + 2 \| \log \bar \rho \|_{W^{1, \infty}}^2. 
\]

\smallskip
%\noindent{\em The case that $K \in L^\infty$ and $ \udiv K \in L^\infty$}. 
There is no need for any integration by part on the other terms in \eqref{timeevolution2}. When  $K \in L^\infty$ and $ \udiv K \in L^\infty$, one simply denotes for some $\eta>0$
\begin{equation}\label{DefPhiBoundedK}
\frac{1}{\eta}\,\phi(x,z)=- \left(K(x-z)-K\star_x \bar\rho(x)\right)\cdot \nabla_x \log \bar\rho (x) -(\udiv K(x-z) - \udiv K\star_x\bar\rho(x)).
\end{equation}

But for  more singular kernels $K$ with $K(-x) = -K(x)$, $|x|K(x) \in L^\infty$ and $\udiv K \in L^\infty$, we have  to do a symmetrization first as in  \cite{Delort}. Indeed, for any $\eta >0$, 
\[
\begin{split}
 \quad \frac{1}{\eta} \phi(x, z) =  - \frac{1}{2} \Big\{  & \left(K(x-z)-K\star_x \bar\rho(x)\right)\cdot \nabla_x \log \bar\rho (x) \\
& +\udiv K(x-z)-\udiv K\star_x\bar\rho(x) \\
& + \left(K(z-x)-K\star_x \bar\rho(z)\right)\cdot \nabla_z \log \bar\rho (z) \\
& +\udiv K(z-x)-\udiv K\star_x\bar\rho(z) \Big\}. 
\end{split}
\] 
Using $K(-x) = -K(x)$, we then obtain
\begin{equation}\label{DefPhiSingularK}
\begin{split}
 - \frac{2}{\eta} \phi(x, z)  & = K(x-z) \cdot \left( \nabla_x \log \bar \rho(x) - \nabla_z \log \bar \rho(z) \right) \\
& - K \star_x \bar \rho(x) \cdot \nabla_x \log \bar \rho(x) - K \star_x \bar \rho(z) \cdot \nabla_z \log \bar \rho(z) \\
& + \udiv K(x-z) - \udiv K \star_x \bar \rho(x)   \\
& + \udiv K(z- x) - \udiv K \star_x \bar \rho(z). 
\end{split}
\end{equation}

We then directly apply Lemma \ref{youngconvex} to
\[
\Phi=\frac{1}{N^2}\,\sum_{i,j =1}^N  \phi(x_i, x_j),
\]
and find
\[
\begin{split}
& -\frac{1}{N^2}\,\sum_{i,\;j=1}^N \int_0^t\int_{\Pi^{d\,N}}\rho_N\, \left( K(x_i-x_j)-K\star_{x}\bar\rho(x_i)\right)\cdot\nabla_{x_i} \log\bar\rho_N\,\ud X^N\,\ud s\\
&-\frac{1}{N^2}\,\sum_{i,\;j=1}^N \int_0^t\int_{\Pi^{d\,N}}\rho_N\,\left(\udiv\,K(x_i-x_j) - \udiv K\star_x \bar\rho(x_i)\right)\,\ud X^N\,\ud s\\
&\quad\leq \frac{1}{\eta}\,{\cal H}_N(\rho_N\,|\;\bar\rho_N)+\frac{1}{ \eta N}\,\log \int_{\Pi^{d\,N}} \bar\rho_N\,\exp\bigg(\frac{1}{N}\sum_{i,j=1}^N   \, \phi(x_i, x_j)\bigg)\, \ud X^N.
  \end{split}
\]
We use Theorem  \ref{MEII} and observe if $K\in L^\infty$ then one directly has that
\[\begin{split}
\gamma&=C\,\left(\sup_{p\geq 1}\frac{\|\sup_z |\phi(.,z)|\,\|_{L^p(\bar\rho\,dx)}}{p} \right)^2\\
&\leq C\,\eta^2\,\|K\|_{\infty}^2\,\left( 1+ \sup_{p \geq 1} \,\frac{\|\nabla\log \bar\rho\|_{L^p(\bar\rho\,dx)}}{p}   \right)^2<1,
\end{split}
\]
provided that one chooses
\[
\eta <  \frac{1}{C\,\|K\|_\infty\, \left( 1+ \sup_p\,\frac{\|\nabla\log \bar\rho\|_{L^p(\bar\rho\,dx)}}{p}  \right)},
\]
and where we recall that
\[
\|K\|_\infty=\|K\|_{L^\infty}+\|\udiv K\|_{L^\infty}.
\]

If $K(x)=-K(-x)$ with $|x|\,K\in L^\infty$, one now has to be careful in estimating
\[
\sup_{p \geq 1}  \frac{\|\sup_z |\phi(.,z)|\,\|_{L^p(\bar \rho\,dx)}}{p},
\]
as $\phi$ is now symmetric in $x$ and $z$.

First we recall the well known estimate, of which we give a short proof at the end of the subsection
\begin{lemma}
  For any function $f$ in $L^q$ with $q>d$, one has that for any $x,\;z$
  \[
|f(x)-f(z)|\leq C_d \,|x-z|\,(M\,|\nabla f|^q (x))^{1/q},
\]
where $M$ is the maximal operator.\label{rademacher}
\end{lemma}

By Lemma \ref{rademacher}, for some $q>d$,  two  terms in  \eqref{DefPhiSingularK} can be estimated as 
\[
\begin{split}
& |K(x-z) \cdot \left( \nabla_x \log \bar \rho(x) - \nabla_z \log \bar \rho(z) \right) | \\
& \leq  C_d \||x| K(x)\|_{L^\infty} (M\, |\nabla_x^2 \log \bar \rho|^q(x))^{1/q},
\end{split} 
\]
and  
\[
\begin{split}
& \left|  K \star_x \bar \rho(x) \cdot \nabla_x \log \bar \rho(x)  +  K \star_x \bar \rho(z) \cdot \nabla_z \log \bar \rho(z)  \right| \\
& \leq  \left|   K \star_x \bar \rho(x) + K \star_x \bar \rho(z)   \right| |\nabla_x \log \bar \rho(x) | \\ & + \left|K \star_x \bar \rho(z)  \right| \left| \nabla_x \log \bar \rho(x) - \nabla_z \log \bar \rho(z)  \right| \\
& \leq C_d\,  \| |x| K(x)  \|_{L^\infty}  \| \bar \rho  \|_{L^\infty} \left( |\nabla_x \log \bar \rho(x)|  +  (M\,|\nabla^2_x  \log \bar \rho|^q (x) )^{1/q} \right). 
\end{split}
\]
Combining with the trivial estimates for terms involving $\udiv K$, i.e. 
\[
|\udiv K(x-z) - \udiv K \star_x \bar \rho(x)|,  \, |\udiv K(z- x) - \udiv K \star_x \bar \rho(z)| \leq 2 \|\udiv K \|_{L^\infty}, 
\]
we finally obtain that  
\[
\begin{split}
\sup_z |\phi(x,z)| & \leq \eta \,  C_d \,\|K\|_\infty\, \bigg(1+\|\bar\rho\|_{L^\infty} |\nabla_x \log \bar \rho(x) | \\
& \quad + (1 + \|\bar \rho \|_{L^\infty}) (M\,|\nabla^2_x \log \bar \rho|^q (x) )^{1/q} \bigg),\\
\end{split}
\]
 where we recall that now 
\[
\| K\|_{\infty} =  \|\,|x| K(x) \|_{L^\infty} + \|\udiv K \|_{L^\infty}. 
\]
Assuming that $\bar\rho\in A_p$ is a Muckenhoupt weight  then for $p>q$
\[
\left\|(M\,|\nabla^2 \log \bar \rho|^q)^{1/q}\right\|_{L^p(\bar\rho\,dx)}\leq \|\bar\rho\|_{A_{p}}\,\left\|\nabla^2 \log \bar \rho\right\|_{L^p(\bar\rho\,dx)},
\]
where we write, through a slight abuse of notation
\[
\|\bar\rho\|_{A_p}=\sup_{B\ ball} \frac{1}{|B|}\int_B \bar\rho(x)\,dx\, \left(\frac{1}{|B|}\int_B \bar\rho(x)^{-p^*/p}\,dx\right)^{p/p^*}.
\]
But in addition $A_r\subset A_p$ if $r\geq p$ while we only need to work with large $p$. On the other hand (see Chapter 5 in \cite{Stein} for example), if $\log\bar\rho\in BMO$ then $\bar\rho\in A_p$ for some $p$.  

Thus we
finally find, similarly to the previous case, that
\[
\begin{split}
  \gamma \leq  & C\, \eta^2\,\|K \|_{\infty}^2  \cdot \left(1 + \| \bar \rho\|_{L^\infty} \right)^2  \bigg( 1  + \sup_{p \geq 1} \frac{\|\nabla \log \bar \rho \|_{L^p(\bar \rho \ud x )}}{p } \\ & +C(\|\log\bar\rho\|_{BMO})\,\sup_{p \geq 1}  \,\frac{\|\nabla^2 \log \bar\rho\|_{L^p(\bar\rho\,dx)}}{p}  \bigg)^2 <1, 
\end{split}
\]
provided again that $\eta$ is chosen small enough,
%\[
%\eta < \frac{1}{C\,\|K\|_\infty\, \left( 1+ \sup_p\,\frac{\|\nabla\log \bar\rho\|_{L^p(\bar\rho\,dx)}}{p}  \sup_{p \geq 1} \,\frac{\|\nabla^2 \log \bar\rho\|_{L^p(\bar\rho\,dx)}}{p} \right)}
%\]
and by  Theorem  \ref{MEII}, we hence have 
\[
\begin{split}
& -\frac{1}{N^2}\,\sum_{i,\;j=1}^N \int_0^t\int_{\Pi^{d\,N}}\rho_N\, \left( K(x_i-x_j)-K\star_{x}\bar\rho(x_i)\right)\cdot\nabla_{x_i} \log\bar\rho_N\,\ud X^N\,\ud s\\
&-\frac{1}{N^2}\,\sum_{i,\;j=1}^N \int_0^t\int_{\Pi^{d\,N}}\rho_N\,\left(\udiv\,K(x_i-x_j) - \udiv K\star_x \bar\rho(x_i)\right)\,\ud X^N\,\ud s\\
&\quad\leq \bar M_2\,\|K\|_\infty\,  \int_0^t \left( {\cal H}_N(\rho_N\,|\;\bar\rho_N)(s)  +\frac{1}{N} \right) \ud s .
  \end{split}
\]
Inserting this in \eqref{timeevolution2}, we find that
\[\begin{split}
    {\cal H}_N(\rho_N\,|\;\bar\rho_N)\leq &{\cal H}_N(\rho_N^0\,|\;\bar\rho_N^0)+\bar M_2\,\|K\|_\infty\, \int_0^t \left( {\cal H}_N(\rho_N\,|\;\bar\rho_N)(s) +\frac{1}{N} \right) \ud s  \\
    &+\bar M_2 \, \left( 1+  t \| K\|_{\infty} \right) \,|\sigma-\sigma_N| 
\end{split}
\]
for some constant $\bar M_2$ depending only on
\[\begin{split}
&\bar M_2\bigg(\sigma, \;   \|\bar\rho\|_{L^\infty},\;  \| \log \bar \rho\|_{BMO}, \;    \sup_{p \geq 1}\frac{\|\nabla \log \bar\rho\|_{L^p(\bar\rho\,dx)}}{p},\;   \sup_{p \geq 1}\frac{\|\nabla^2  \log \bar\rho\|_{L^p(\bar\rho\,dx)}}{p},  \\ & \quad \quad  \;  \| \log \bar \rho\|_{W^{1, \infty}},  \; 
\frac{1}{N} \int_{\Pi^{d \, N}} \rho_N^0 \log \rho_N^0,\;  \|\udiv F\|_{L^\infty}\bigg).
\end{split}
\]
This concludes the proof by Gronwall lemma. 

\bigskip

\begin{proof}[Proof of Lemma \ref{rademacher}]
Note that this estimate is also connected to the classical Rademacher theorem for a.e. differentiability of functions in $W^{1,q}$ for $q>d$.
  
  First we recall the very classical (see again \cite{Stein} for example)
  \[
|f(x)-f(z)|\leq C_d\,\int_{|y-x|\leq 2\,|x-z|} |\nabla f(y)|\,\left(\frac{dy}{|y-x|^{d-1}}+\frac{dy}{|z-y|^{d-1}}\right),
\]
which can be simply derived by integrating $|\nabla f|$ over arcs if circles between $x$ and $z$ and averaging over all such arcs that belong to the ball $|y-x|\leq 2\,|x-z|$.

Now we can just observe that for $q>d$
\[
\begin{split}
  &\int_{|y-x|\leq 2\,|x-z|} |\nabla f(y)|\,\frac{dy}{|z-y|^{d-1}}\leq \left(\int_{|y-x|\leq 2\,|x-z|} |\nabla f(y)|^q\,dy\right)^{1/q}\\
  &\qquad\qquad \left(\int_{|y-x|\leq 2\,|x-z|} \frac{dy}{|z-y|^{(d-1)\,q/(q-1)}}\right)^{(q-1)/q}\\
  &\qquad \leq C_d\,|x-z|^{1-d/q}\,\left(\int_{|y-x|\leq 2\,|x-z|} |\nabla f(y)|^q\,dy\right)^{1/q}\\
  &\qquad\leq C_d\,|x-z|\,(M|\nabla f|^q(x))^{1/q}.
  \end{split}
\]
The other term may be bounded in the same manner (and is in fact better as it could be controlled by $M|\nabla f|(x)$ directly), thus concluding the proof. 
  \end{proof}
%%%%%%%%%%%%%%%%%%%%%%%%%%%%%%%%%%%%%%%%%%%%%%%%%%%%%%%%%%%%%%%%%%%%%%%%%%
%%%%%%%%%%%%%%%%%%%%%%%%%%%%%%%%%%%%%%%%%%%%%%%%%%%%%%%%%%%%%%%%%%%%%%% 
\section{Preliminary of combinatorics} \label{PreCombi}
%%%%%%%%%%%%%%%%%%%%%%%%%%%%%%%%%%%%%%%%%%%%%%%%%%%%%%%%%%%%%%%%%%%%%%%%%
%%%%%%%%%%%%%%%%%%%%%%%%%%%%%%%%%%%%%%%%%%%%%%%%%%%%%%%%%%%%%%%%%%%%%%%%%%
Before the proof of the main estimates Theorem \ref{MEI} and  Theorem \ref{MEII}, we list some useful combinatorics results used throughout this article.
We first recall Stirling's formula
\begin{equation}
n! = \lambda_n \sqrt{2 \pi n} \left( \frac{n}{e}\right)^n,  \label{stirling}
\end{equation}
where $1 < \lambda_n  < \frac{11}{10}$ and $\lambda_n \to 1$ as $n \to \infty$.

We have the elementary bound following from \eqref{stirling}
\begin{lemma}\label{BioCoeBouLem} For any $1 \leq p \leq q$, one has 
\[
\binom{q}{p} \leq e^p q^p p^{-p}. 
\]
\end{lemma}
One also has the basic combinatorics on $p$-tuples
\begin{lemma}\label{Diophantine} For any $1 \leq p \leq q$,  one has 
\[
|\{(b_1,\ldots,b_p)\in \N^p\, \vert\; \forall l\ 1\leq b_l\leq q\  \mbox{and}\ b_1 + b_2 + \cdots + b_p =q \} | = \binom{q-1}{p-1}. 
\]
\end{lemma}
\begin{proof}[Proof of Lemma \ref{Diophantine}] 
When $p=1$, the lemma trivially  holds true with the convention $\binom{0}{0}=1$  if $p=q=1$. We thus assume  $p \geq 2 $ in the following.   Since each $p-$tuple $(b_1, b_2, \cdots, b_p)$ uniquely determines a $(p-1)-$tuple $(c_1, c_2, \cdots, c_{p-1})$ and reciprocally via 
 \[
 c_1=b_1, c_2= b_1 +b_2, \cdots, c_{p-1}= b_1 + b_2 + \cdots + b_{p-1},
 \]
it suffices to verify that 
\[
|\{ (c_1, c_2, \cdots, c_{p-1})\, \vert\; 1 \leq c_1 < c_2 < \cdots < c_{p-1} \leq q-1\}| = \binom{q-1}{p-1}.
\]
 This is simply obtained by choosing $p-1$ distinct integers from the set $\{1, 2,\cdots, q-1\}$ and assigning the smallest one to $c_1$, the second smallest to $c_2$, and so on. 
\end{proof}

Much of the combinatorics that we handle is based only on the multiplicity in the multi-indices. It is therefore convenient to know how many multi-indices can have the same multiplicity signature
\begin{lemma} For any $a_1,\ldots,a_q\in \N$ s.t. $a_1+\cdots+a_q=p$, then the set of multi-indices $I_p=(i_1,\ldots,i_p)$ with $1\leq i_k\leq q$ and corresponding multiplicities has cardinal 
\[
\bigg|\Big\{(i_1,\ldots,i_p)\in \{1,\ldots,q\}^p\,|\; \forall l\ a_l=|\{k\,,\;i_k=l\}|\Big\}\bigg|=\frac{p!}{a_1!\cdots a_q!}.
\] 
%where $Z=|\{l\,|a_l\geq 1\}|$ is the total number of indices that appear in $I_p$. 
\label{multiplicity}
\end{lemma} 
\begin{proof}
This is the basic multinomial relation: We have to choose $1$ $a_1$ times among $p$ positions, $2$ $a_2$ times among the remaining positions and so on...
\end{proof}
Similarly as for the binomial coefficients, $\frac{p!}{a_1!\cdots a_q!}$ is the coefficient of $x_1^{a_1}\,\dots\,x_q^{a_q}$ in the expansion of $(x_1+\cdots +x_q)^p$ leading to the obvious estimate
\begin{equation}
\sum_{a_1,\ldots, a_q\geq 0,\ a_1+\cdots+a_q=p}\frac{p!}{a_1!\cdots a_q!}=q^p.
\label{multinomialsum}
\end{equation}

\smallskip

Let us fix some notations here. We write the \textbf{integer valued} $p-$tuple as $I_{p} =(i_1, \cdots, i_{p})$ . The \textbf{overall set} $\mathcal{T}_{q, p} $ of those indices is defined as 
\begin{equation}\label{TotSet}
\mathcal{T}_{q, p} = \{ I_p=(i_1, \cdots, i_p) \vert 1 \leq i_\nu \leq q, \text{\ for all\ }  1 \leq  \nu \leq p\}. 
\end{equation}
We thus define the \textbf{multiplicity function} $\Phi_{q, p}: \mathcal{T}_{q, p} \to \{0, 1, \cdots, p \}^q$ with $\Phi_{q, p}(I_p)=A_q= (a_1, a_2, \cdots, a_q)$,  where 
\[
a_l= |\{ 1 \leq \nu \leq p\,\vert\; i_\nu =l \}|. 
\]
In many of our proofs, we use cancellations so that any $I_p$ which has an index of multiplicity exactly $1$ leads to a vanishing term.

This leads to the definition of  the  ``\textbf{effective set}" $\mathcal{E}_{q, p}$  by 
\[\begin{split}
\mathcal{E}_{q, p} = \{I_p \in \mathcal{T}_{q,p}\;|\ &\Phi_{q, p}(I_p)= A_q=(a_1, \cdots, a_q)\\
& \text{\ with \ } a_\nu \ne 1 \text{\  for any \ } 1 \leq \nu \leq q.  
\}
\end{split}
\]

One has the following combinatorics result 
\begin{lemma} \label{CombiI}
Assume that $1 \leq p \leq q$. Then 
\begin{equation}
\label{boundQp}
|\mathcal{E}_{q, p}|\leq \sum_{l=1}^{\lfloor \frac{p}{2}\rfloor} \binom{q}{l}l^{p} \leq \lfloor \frac{p}{2}\rfloor \binom{q}{\lfloor \frac{p}{2}\rfloor} \left( \lfloor \frac{p}{2}\rfloor\right)^{p} \leq \frac{p}{2} e^{\frac{p}{2}} q^{\frac{p}{2}} \left( \frac{p}{2}\right)^{\frac{p}{2}}.
\end{equation}
\end{lemma}

\begin{proof}[Proof of Lemma \ref{CombiI}]
Pick any multi-index $I_p=(i_1, \cdots, i_{p}) \in \mathcal{E}_{q,p}$ and write that  $|I_p|=|\{i_1, \cdots, i_p\}|$. Each element in $I_p$ appears at least twice and hence $|I_p| \leq  \lfloor \frac{p}{2} \rfloor.$

%The fact that $I_p \in \mathcal{E}_{q,p}$ implies that the multiplicity of each integer cannot be one. Hence $1 \leq \vert I_p\vert \leq \lfloor \frac{p}{2} \rfloor$.  Indeed, if $|I_p| \geq \lfloor \frac{p}{2} \rfloor +1$, then 
%\[ 
%p \geq 2 \left(\lfloor \frac{p}{2} \rfloor +1  \right) >2 \left( \frac{p}{2}-1 +1\right)=p,
%\]
%which is impossible.  

If $p=1$, then $\mathcal{E}_{q,p}=\emptyset$. The estimate (\ref{boundQp}) holds trivially. In the following we assume that $p \geq 2$. 

Denote $l = |I_p|$ which can be $1, 2, \cdots, \lfloor \frac{p}{2} \rfloor$.  Consequently, one has  by summing all possible choices for $|I_p|$
\[
|\mathcal{E}_{q, p}| = \sum_{l=1}^{\lfloor \frac{p}{2} \rfloor} \vert \{I_p \in \mathcal{E}_{q, p} \vert \ |I_p|=l \}\vert. 
\]
 For a fixed $|I_p| =l$, there are $\binom{q}{l}$ many choices of numbers $l$ from $S= \{1, 2, \cdots, q\}$ to compose $I_p$. 

Having already chosen those $l$ numbers from $S$, without loss of generality we may assume that $I_p$ as a set coincides with $\{1, 2, \cdots, l \}$.  The total choices  of $p-$tuple $I_p$ can be bounded by $l^p$ trivially since each $i_\nu$ has at most $l$ choices.  

Remark that  $1 \leq l \leq \lfloor\frac{p}{2}\rfloor \leq \lfloor\frac{q}{2}\rfloor $, so that
\[
\binom{q}{l} \leq \binom{q}{\lfloor\frac{p}{2}\rfloor}.
\] 
Hence one has 
\[
|\mathcal{E}_{q, p}|  \leq \sum_{l=1}^{\lfloor \frac{p}{2} \rfloor} \binom{q}{l} l^p \leq \lfloor\frac{p}{2}\rfloor \binom{q}{\lfloor\frac{p}{2}\rfloor} \left(\lfloor\frac{p}{2}\rfloor \right)^{p}.
\]
The last inequality in (\ref{boundQp}) is now ensured by Lemma \ref{BioCoeBouLem}, in particular the following inequality 
\[
\binom{q}{\lfloor  \frac{p}{2}\rfloor} \leq e^{\lfloor  \frac{p}{2}\rfloor} q^{\lfloor  \frac{p}{2}\rfloor} (\lfloor  \frac{p}{2}\rfloor)^{ - \lfloor  \frac{p}{2}\rfloor}. 
\]
This finishes the proof of Lemma \ref{CombiI}. 
\end{proof}

%%%%%%%%%%%%%%%%%%%%%%%%%%%%%%%%%%%%%%%%%%%%%%%%%%%
%%%%%%%%%%%%%%%%%%%%%%%%%%%%%%%%%%%%%%%%%%%%%%%%%%%%%%%
\section{Proof of Theorem \ref{MEI}\label{proofMEI}}
%%%%%%%%%%%%%%%%%%%%%%%%%%%%%%%%%%%%%%%%%%%%%%%%%%%%%%
%%%%%%%%%%%%%%%%%%%%%%%%%%%%%%%%%%%%%%%%%%%%%%%%%%%%%%
 
The goal here is to bound
\[
\int_{\Pi^{d\,N}} \bar{\rho}_N \exp\bigg(\frac{1}{N}\sum_{j_1,j_2=1}^N \psi(x_1,x_{j_1})\,\psi(x_1,x_{j_2})\bigg) \ud X^N, 
\]
for any bounded $\psi$ with vanishing average against $\bar\rho$.  

Since 
\[
\exp(A) \leq \exp(A) + \exp(-A) = 2   \sum_{k=0}^\infty \frac{1}{(2k)!} A^{2k},
\]
it suffices only to bound the series with even terms 
\begin{equation}
\label{eventerms}\begin{split}
&\int_{\Pi^{d\,N}} \bar{\rho}_N \exp\bigg(\frac{1}{N}\sum_{j_1,j_2=1}^N \psi(x_1,x_{j_1})\,\psi(x_1,x_{j_2})\bigg) \ud X^N\\
&\qquad\leq 2 \sum_{k=0}^\infty \frac{1}{(2k)!} \int_{\Pi^{d \,N}} \bar{\rho}_N\, \bigg(\frac{1}{N}\sum_{j_1,j_2=1}^N \psi(x_1,x_{j_1})\,\psi(x_1,x_{j_2})\bigg)^{2k} \ud X^N, 
\end{split}
\end{equation}
where in general the $k-$th even term can be expanded as 
\begin{equation}
\label{even_k}
\begin{split}
& \frac{1}{(2k)!} \int_{\Pi^{d\,N}} \bar{\rho}_N \bigg(\frac{1}{N}\sum_{j_1,j_2=1}^N \psi(x_1,x_{j_1})\,\psi(x_1,x_{j_2})\bigg)^{2k} \ud X^N \\ 
&\quad=  \frac{1}{(2k)!} \frac{1}{N^{2k}} \sum_{j_1, \cdots, j_{4k}=1}^N \int_{\Pi^{d\,N}} \bar{\rho}_N \, \psi(x_1,x_{j_1}) \cdots \psi(x_1,x_{j_{4k}})  \ud X^N. 
\end{split}
\end{equation}

We divide the proof in two different cases: Where $k$ is small compared to $N$ and in the simpler case where $k$ is comparable to or larger than $N$. 

\bigskip

\paragraph{Case: $4 \leq 4k \leq N$} First observe that for any particular choice of indices $j_1,\ldots,j_{4k}$, one has 
\begin{equation} \label{OneTerm}
\begin{split}
  \int_{\Pi^{d\,N}} \bar{\rho}_N \, \psi(x_1,x_{j_1}) \cdots \psi(x_1,x_{j_{4k}})\ud X^N  \leq \|\psi\|_{L^\infty}^{4k}.
 \end{split}
\end{equation}
The whole estimate hence relies on counting how many choices of multi-indices $(j_1,\ldots,j_{4k})$ lead to a non-vanishing term. Denote hence ${\cal N}_{N,4k}$ the set of multi-indices $(j_1,\cdots,j_{4k})$ s.t.
\[
\int_{\Pi^{d\,N}} \bar{\rho}_N \, \psi(x_1,x_{j_1}) \cdots \psi(x_1,x_{j_{4k}})\ud X^N\neq 0.
\]
Denote by $(a_1,\ldots,a_N)$ the multiplicity for $(j_1,\ldots,j_{4k})$, 
\[
a_l=|\{ \nu \in \{1,\ldots,4k\},\ j_\nu=l\}|.
\]
If there exists $l\neq 1$ s.t. $a_l=1$, then the variable $x_{l}$ enters exactly once in the integration. Assume for simplicity that $j_1=l$ then
\[
\begin{split}
  &\int_{\Pi^{d\,N}} \bar{\rho}_N \, \psi(x_1,x_{j_1}) \cdots \psi(x_1,x_{j_{4k}})\ud X^N\\
&\quad=\int_{\Pi^{d\,(N-1)}}  \, \psi(x_1,x_{j_2}) \cdots \psi(x_1,x_{j_{4k}})\, \Pi_{i\neq l}\bar{\rho}(x_i)\,\ud x_i\\
&\qquad\qquad\qquad\int_{\Pi^d} \bar\rho(x_{j_1})\,\psi(x_1,x_{j_1})\,\ud x_{j_1}=0,
\end{split}
\] 
by the assumption of vanishing mean average for $\psi$, {\em provided $l=j_1\neq 1$}.

Recall the definitions of the overall set (see \eqref{TotSet}) and the effective set
\[\begin{split}
\mathcal{E}_{q, p} = \{I_p \in \mathcal{T}_{q,p}\;|\ & (a_1, \cdots, a_q) = \Phi_{q, p}(I_p)\\
& \text{\ with \ } a_\nu \ne 1 \text{\  for any \ } 1 \leq \nu \leq q  
\},
\end{split}
\]
where $(a_1,\cdots,a_q)$ denotes the multiplicity of the multi-index $I_p$. 

Therefore the integral
\[
\int_{\Pi^{d\,N}} \bar{\rho}_N \, \psi(x_1,x_{j_1}) \cdots \psi(x_1,x_{j_{4k}})\ud X^N
\]
vanishes unless $j_1,\cdots,j_{4k}$ belongs to $\mathcal{E}_{N,4k}$ (all multiplicities are different from $1$) or satisfies $a_1=1$ and every $a_l\neq 1$ for $l>1$.

In that last case, we have to choose one index $n$ s.t. $j_n=1$, with $4k$ possibilities. The rest of the multi-index $(j_1,\cdots,j_{n-1},j_{n+1},\cdots, j_{4k})$ must have all multiplicities different from $1$. This multi-index hence belongs to $\mathcal{E}_{N-1,4k-1}$. 

Consequently,
\[
|{\cal N}_{N,4k}|\leq |\mathcal{E}_{N,4k}|+4k\,|\mathcal{E}_{N-1,4k-1}|.
\]
We now apply Lemma \ref{CombiI}
\[
{\cal N}_{N,4k} \leq (1+ 4k )\,|\mathcal{E}_{N, 4k}| \leq 10\, k^2\, e^{2k}\, N^{2k}\, (2k)^{2k}.
\]
Using \eqref{OneTerm} , for $1 \leq k \leq \lfloor \frac{N}{4} \rfloor$, we obtain
\begin{equation}\label{ksmall}
\begin{split}
&\frac{1}{(2k)!} \frac{1}{N^{2k}} \sum_{j_1, \cdots, j_{4k}=1}^N \int_{\Pi^{d\,N}} \bar{\rho}_N \, \psi(x_1,x_{j_1}) \cdots \psi(x_1,x_{j_{4k}})  \ud X^N\\
 & \quad\leq \frac{1}{(2k)!} \frac{10}{N^{2k}}\, k^2\, e^{2k }\, N^{2k}\, (2k)^{2k}\,\|\psi\|_{L^\infty}^{4k}  \\
&\quad \leq 5\,e^{4k}\, k^\frac{3}{2}\,\|\psi\|_{L^\infty}^{4k} ,
\end{split}
\end{equation}
by   Stirling's formula for $n = 2k$. 

\bigskip

\paragraph{Case: $4k >N$.} In this case, we do not need to use any combinatorics. We simply  remark that there can be at most $N^{4k}$ multi-indices. From \eqref{OneTerm}, we have for $k > \lfloor \frac{N}{4} \rfloor$ 
\begin{equation}
\label{klarge}
\begin{split}
&\frac{1}{(2k)!} \frac{1}{N^{2k}} \sum_{j_1, \cdots, j_{4k}=1}^N \int_{\Pi^{d\,N}} \bar{\rho}_N \, \psi(x_1,x_{j_1}) \cdots \psi(x_1,x_{j_{4k}})  \ud X^N\\
&\qquad\leq  \frac{1}{(2k)!}\, \frac{1}{N^{2k}}\, N^{4k}\, \|\psi\|_{L^\infty}^{4k} \leq k^{-\frac{1}{2}}\,2^{2k}\,e^{2k}\, \|\psi\|_{L^\infty}^{4k},\\
\end{split}
\end{equation}
still by  Stirling's formula. 

\bigskip

\paragraph{Conclusion of the proof.}
Combining \eqref{ksmall},  \eqref{klarge} and \eqref{eventerms},  we have that 
\[
\begin{split}
&\int_{\Pi^{d\,N}} \bar{\rho}_N \exp\bigg(\frac{1}{N}\sum_{j_1,j_2=1}^N \psi(x_1,x_{j_1})\,\psi(x_1,x_{j_2})\bigg) \ud X^N\\
&\qquad \leq 2 \bigg( 1 + \sum_{k=1}^{\lfloor \frac{N}{4} \rfloor} 5\,e^{4k}\, k^\frac{3}{2}\,\|\psi\|_{L^\infty}^{4k}   + \sum_{k=\lfloor \frac{N}{4} \rfloor +1 }^\infty k^{-\frac{1}{2}}\,2^{2k}\,e^{2k}\, \|\psi\|_{L^\infty}^{4k}\bigg) . 
\end{split}
\]
The proof of Theorem \ref{MEI} is completed by 
\[
\begin{split}
 &\sum_{k=1}^{\lfloor \frac{N}{4} \rfloor}   5\,e^{4k}\, k^\frac{3}{2}\,\|\psi\|_{L^\infty}^{4k}  \leq 5\, \alpha \sum_{k=1}^\infty k(k+1) \alpha^{k-1}  \\& =5\, \alpha\, \frac{\ud^2}{\ud \alpha^2} \bigg( \sum_{k=0}^\infty \alpha^k  \bigg)= 5\, \alpha\, \left( \frac{1}{1-\alpha} \right)^{''} = \frac{10\, \alpha}{(1- \alpha)^3} < \infty, 
\end{split}
\]
while
\[\begin{split}
\sum_{k=\lfloor \frac{N}{4} \rfloor +1 }^\infty  k^{-\frac{1}{2}}\,2^{2k}\,e^{2k}\, \|\psi\|_{L^\infty}^{4k}  &\leq \sum_{k=1}^\infty \beta^k = \frac{1}{1-\beta} -1 \\
&= \frac{\beta}{1-\beta} < \infty,
\end{split}\]
where we recall 
\begin{equation*}
 \alpha= \left(e\, \|\psi\|_{L^\infty} \right)^4 <1, \quad \beta= \left(\sqrt{2e}\, \|\psi\|_{L^\infty}  \right)^4 <1. 
 \end{equation*}

%%%%%%%%%%%%%%%%%%%%%%%%%%%%%%%%%%%%%%%%%%%%%%%%%%%%%%%%%%%%%%%%%%%%%%%%
%%%%%%%%%%%%%%%%%%%%%%%%%%%%%%%%%%%%%%%%%%%%%%%%%%%%%%%%%%%%%%%%%%%%%%%%%%%
\section{Proof of Theorem \ref{MEII} \label{ProofMEII}}
%%%%%%%%%%%%%%%%%%%%%%%%%%%%%%%%%%%%%%%%%%%%%%%%%%%%%%%%%%%%%%%%%%%%%%%%%%
%%%%%%%%%%%%%%%%%%%%%%%%%%%%%%%%%%%%%%%%%%%%%%%%%%%%%%%%%%%%%%%%%%%%%%%%%%%
We recall that our goal is to bound
\[
\int_{\Pi^{d\,N}} \bar{\rho}_N \exp\bigg(\frac{1}{N}\sum_{i,j=1}^N \phi(x_i,x_j)\bigg) \ud X^N
\]
with the assumptions
\begin{equation}
\int_{\Pi^d} \phi(x,z)\,\bar\rho(x)\,dx=0\quad\forall z, \qquad \int_{\Pi^d} \phi(x,z)\,\bar\rho(z)\,dz=0\quad\forall x.\label{phicancellation}
\end{equation}
As in the proof of Theorem \ref{MEI},  one expands the exponential in series and only needs to bound the even terms 
\begin{equation}\label{exp2k}\begin{split}
&\int_{\Pi^{d\,N}} \bar{\rho}_N\, \exp\bigg(\frac{1}{N}\sum_{i,j=1}^N \phi(x_i,x_j)\bigg)\, \ud X^N \\
&\qquad\leq 2 \sum_{k=0}^\infty \frac{1}{(2k)!} \int_{\Pi^{d\,N}} \bar{\rho}_N\, \bigg|\frac{1}{N}\sum_{i,j=1}^N \phi(x_i,x_j) \bigg|^{2k}\, \ud X^N. 
\end{split}
\end{equation}
As in the proof of Theorem \ref{MEI}, we separate the proof into two cases: the case where $k$ is relatively small  compared to $N$ which requires a careful combinatorial analysis to take vanishing terms into account and the more straightforward case when $k$ is comparable to or larger than $N$. 
 
 Accordingly  Theorem \ref{MEII} is a consequence of the  following two propositions 
\begin{proposition}\label{PropkLarge}If $4 k >N$, one has 
\[\begin{split}
& \frac{1}{(2k)!} \int_{\Pi^{d\,N}} \bar{\rho}_N\, \bigg|\frac{1}{N}\sum_{i,j=1}^N \phi(x_i,x_j) \bigg|^{2k}\, \ud X^N\\
&\qquad\qquad \leq   \left(6e^2\,  \sup_{p \geq 1} \frac{\| \sup_z |\phi(.,z)|\|_{L^p(\bar\rho\,dx)}}{p}\right)^{2k}. 
\end{split}
\]

\end{proposition}

\begin{proposition}\label{PropkSmall}
For $4 \leq 4k \leq N$, one has 
\[\begin{split}
& \frac{1}{(2k)!} \int_{\Pi^{d\,N}} \bar{\rho}_N\, \bigg|\frac{1}{N}\sum_{i,j=1}^N \phi(x_i,x_j) \bigg|^{2k}\, \ud X^N  \\
&\qquad\qquad\leq   \left(1600\, \sup_{p \geq 1} \frac{\| \sup_z |\phi(.,z)|\|_{L^p(\bar\rho\,dx)}}{p}\right)^{2k}  . 
\end{split}
\]
\end{proposition}
Let us give a quick proof of Theorem \ref{MEII} assuming Proposition \ref{PropkLarge} and Proposition \ref{PropkSmall}. 
\begin{proof}[Proof of Theorem \ref{MEII}] By \eqref{exp2k} and Proposition \ref{PropkLarge} and Proposition \ref{PropkSmall}, one has 
\[
\begin{split}
& \int_{\Pi^{d\,N}} \bar{\rho}_N\, \exp\bigg(\frac{1}{N}\sum_{i,j=1}^N \phi(x_i,x_j)\bigg)\, \ud X^N\\  
&\qquad\leq 2 \Bigg( 1   + \sum_{k=1}^{\lfloor \frac{N}{4} \rfloor} \left( 1600\,\sup_{p \geq 1} \frac{\| \sup_z |\phi(.,z)|\|_{L^p(\bar\rho\,dx)}}{p}\right)^{2k}\\  
&\qquad\qquad + \sum_{k = \lfloor \frac{N}{4} \rfloor + 1 }^{\infty}   \left(6e^2\, \sup_{p \geq 1} \frac{\| \sup_z |\phi(.,z)|\|_{L^p(\bar\rho\,dx)}}{p}\right)^{2k}  \Bigg).   \\ 
\end{split}
\]
We defined $\gamma=C\,\left( \sup_{p \geq 1} \frac{\| \sup_z |\phi(.,z)|\|_{L^p(\bar\rho\,dx)}}{p}\right)^2 <1$. One obtains, taking $C = 1600^2 +  36\,e^4$, 
\[
\int_{\Pi^{d\,N}} \bar{\rho}_N\, \exp\bigg(\frac{1}{N}\sum_{i,j=1}^N \phi(x_i,x_j)\bigg)\, \ud X^N \leq 2 \sum_{k =0}^\infty \gamma^k = \frac{2}{1 - \gamma} < \infty, 
\]
completing the proof of Theorem \ref{MEII}. 
\end{proof} 
% 
%%%%%%%%%%%%%%%%%%%%%%%%%%%%%%%%%%%%%%%%%%%%%%%%%%%%%%%%%%%%%%%%%%%%%%%%%%% 
\subsection{The case $4k > N$: Proof of Proposition  \ref{PropkLarge} } 
%%%%%%%%%%%%%%%%%%%%%%%%%%%%%%%%%%%%%%%%%%%%%%%%%%%%%%%%%%%%%%%%%%%%%%%%%%%
For $k > \frac{N}{4} $ the $k-$th even term can be estimated by  
\[
%\label{ThetaKLar}
\begin{split}
&\frac{1}{(2k)!} \int_{\Pi^{d\,N}} \bar{\rho}_N\, \bigg|\frac{1}{N}\sum_{i,j=1}^N \phi(x_i,x_j) \bigg|^{2k}\, \ud X^N \\
&\ \leq \frac{1}{(2k)!} \frac{1}{N^{2k}} \sum_{i_1,j_1, \cdots, i_{2k},j_{2k}=1}^N \int_{\Pi^{d\,N}} \bar{\rho}_N\, \sup_z |\phi(x_{i_1},z)|\cdots  \sup_z |\phi(x_{i_{2k}},z)|\, \ud X^N \\
&\quad=  \frac{1}{(2k)!} \int_{\Pi^{d\,N}} \bar{\rho}_N \left( \sum_{i=1}^N \sup_z |\phi(x_{i},z)| \right)^{2k} \ud X^N.  \\
\end{split}
\]
Hence
\begin{equation}\label{ThetaKLar}
\begin{split} 
&\frac{1}{(2k)!} \int_{\Pi^{d\,N}} \bar{\rho}_N\, \bigg|\frac{1}{N}\sum_{i,j=1}^N \phi(x_i,x_j) \bigg|^{2k}\, \ud X^N \\
&\qquad
\leq  \frac{1}{(2k)!} \sum_{\substack{a_1+\cdots + a_N=2k,\\a_1 \geq 0,\cdots a_N \geq 0}} \frac{(2k)!}{(a_1)! \cdots (a_N)!} M_{a_1}^{a_1} \cdots M_{a_N}^{a_N},
\end{split}
\end{equation}
where we denote 
\[
M_{a_i}^{a_i} = \int_{\Pi^d} \sup_z |\phi(x, z)|^{a_i} \bar\rho(x) \ud x
\]
with the convention that $M_0^0=1$. Remark that
\[\begin{split}
M_{a_i}^{a_i} &\leq  a_i^{a_i} \left( \sup_{p \geq 1} \frac{ \|\sup_z |\phi(x, z)|\|_{L^p(\bar\rho\,dx)}}{p}\right)^{a_i} \\
&\leq e^{a_i} (a_i)! \left( \sup_{p \geq 1} \frac{ \|\sup_z |\phi(x, z)|\|_{L^p(\bar\rho\,dx)}}{p}\right)^{a_i} , 
\end{split}
\]
where the last inequality $n^n \leq e^n n!$ can be easily verified  by Stirling's formula. Inserting it into \eqref{ThetaKLar}, one obtains
\begin{equation}\label{LargeKFin}\begin{split}
&\frac{1}{(2k)!} \int_{\Pi^{d\,N}} \bar{\rho}_N\, \bigg|\frac{1}{N}\sum_{i,j=1}^N \phi(x_i,x_j) \bigg|^{2k}\, \ud X^N\\
 &\qquad\leq  \,e^{2k}\,\left( \sup_{p \geq 1} \frac{ \|\sup_z |\phi(x_{i_1},z)|\|_{L^p(\bar\rho\,dx)}}{p}\right)^{2k}\, \sum_{\substack{a_1+\cdots + a_N=2k,\\a_1 \geq 0,\cdots a_N \geq 0}} 1. 
\end{split}
\end{equation}
The quantity $\sum_{\substack{a_1+\cdots + a_N=2k\\a_1 \geq 0,\cdots a_N \geq 0}} 1$ is equal to the cardinality of the set 
\[
\{ (a_1, a_2, \cdots, a_N) \vert a_1 + \cdots + a_N =2k, a_i \geq 0 \text{\  for \  } 1 \leq i \leq N \}
\]
or the cardinality of the following equinumerous set 
\[
\{ (b_1, b_2, \cdots, b_N) \vert b_1 + \cdots + b_N =2k +N , b_i \geq 1 \text{\  for \  } 1 \leq i \leq N \}.
\]
Applying Lemma \ref{Diophantine} in section  \ref{PreCombi} by taking $p = N$ and $q = 2k+N$, this cardinal is exactly $\binom{2k+N -1}{N-1}$. 
 
This expression can be simplified.  Note that if $a\geq b/2$ by Stirling's formula
\[
\binom{a+b}{b}\leq \frac{\sqrt{a+b}}{\sqrt{\,\pi\,a\,b}}\,(a+b)^{a+b} a^{-a} b^{-b}\leq (1+b/a)^{a}\,(1+a/b)^{b}\leq 3^a\,(1+a/b)^{b}.
\]
Since $(1+ \frac{1}{s})^s < e$ for any $s >0$, this gives
\[
\binom{a+b}{b}\leq (3\,e)^{a}.
\]
Since $4k >N$, $\binom{2k+N-1}{N-1}\leq 3^{2k}\,e^{2k}$ and therefore from \eqref{LargeKFin}, one obtains that  
 \begin{equation}\label{klargefinal}\begin{split}
&  \frac{1}{(2k)!} \int_{\Pi^{d\,N}} \bar{\rho}_N\, \bigg|\frac{1}{N}\sum_{i,j=1}^N \phi(x_i,x_j) \bigg|^{2k}\, \ud X^N\\
& \qquad\leq  3^{2k}\,e^{4k}\,\left( \sup_{p \geq 1} \frac{ \|\sup_z |\phi(x, z)|\|_{L^p(\bar\rho\,dx)}}{p}\right)^{2k}.
\end{split} \end{equation}
This proves Proposition \ref{PropkLarge}. 
%
%%%%%%%%%%%%%%%%%%%%%%%%%%%%%%%%%%%%%%%%%%%%%%%%%%%%%%%%%%%%%%%%%%%%%%%%%%
\subsection{The case  $4 \leq 4k \leq N$: Proof of Proposition \ref{PropkSmall}}
%%%%%%%%%%%%%%%%%%%%%%%%%%%%%%%%%%%%%%%%%%%%%%%%%%%%%%%%%%%%%%%%%%%%%%%%%%
 In this case, the previous straightforward approach fails, even assuming that $\phi\in L^\infty$  as we would only get
\[
\begin{split}
&  \frac{1}{(2k)!} \int_{\Pi^{d\,N}} \bar{\rho}_N\, \bigg|\frac{1}{N}\sum_{i,j=1}^N \phi(x_i,x_j) \bigg|^{2k}\, \ud X^N\leq \frac{N^{2k}}{(2k)!}\,\|\phi\|_{L^\infty}^{2k},
\end{split}
\]
which blows up when $N$ goes to infinity. The key here, as is in the proof of Theorem \ref{MEI}, is to identify the right cancellations in the expansion 
\begin{equation}\label{EveTerSer} \begin{split}
   &  \frac{1}{(2k)!} \int_{\Pi^{d\,N}} \bar{\rho}_N\, \bigg|\frac{1}{N}\sum_{i,j=1}^N \phi(x_i,x_j)\bigg|^{2k}\, \ud X^N  \\  
&\quad \leq  
 \frac{1}{(2k)!} \frac{1}{N^{2k}} \sum_{i_1,j_1 \cdots, i_{2k},j_{2k}=1}^N \int_{\Pi^{d\,N}} \bar{\rho}_N\, \phi(x_{i_1},x_{j_1})\cdots \phi(x_{i_{2k}}, x_{j_{2k}})\, \ud X^N. 
\end{split}
 \end{equation}
%%%%%%%%%%%%%%%%%
\subsubsection{Notations and preliminary considerations}
%%%%%%%%%%%%%%%%%
We denote by $I_{2k} =(i_1, \cdots, i_{2k})$ the $i-$indices and by $J_{2k}= (j_1, \cdots, j_{2k})$ similarly the $j-$indices, where all $i_\nu, j_\nu$ are in  $ \{1, 2, \cdots, N \}$ for $1 \leq \nu \leq 2k$.

We denote by $ (a_1, a_2, \cdots, a_{N})$ the multiplicities of $I_{2k}$, 
\[
a_{l} = |\{1 \leq \nu \leq 2k \vert i_\nu = l \}|, \quad  l = 1, 2, \cdots, N, 
\]
and by $(b_1,\cdots,b_N)$ the multiplicities of $J_{2k}$. 

For the study of cancellations, the critical parameter will be the number of multiplicities which are exactly $1$ in $I_{2k}$, so that we denote
\begin{equation}
m_I=|\{l\,|\;a_l=1\}|,\quad n_I=|\{l\,|\;a_l>1\}|.\label{defmInI}
\end{equation}
Note that $m_I+n_I$ is exactly the number of integers  present in $I_{2k}$: $m_I+n_I=|\{l\,|\;a_l\geq 1\}|$.

We start by the following lemma which, for every $I_{2k}$, identifies the only possible $J_{2k}$ s.t. the integral does not vanish. 

First we simplify the possible expression of $I_{2k}$ which makes the counting easier by using the natural symmetry by permutation of the problem. For any $\tau \in {\cal S}_N$, we simply define $\tau (I_{2k})=(\tau(i_1),\ldots,\tau (i_{2k}))$. Thus $\tau $ is a one-to-one application on the $I_{2k}$ and moreover 
\[\begin{split}
&\int_{\Pi^{d\,N}} \bar{\rho}_N\, \phi(x_{i_1},x_{j_1})\cdots \phi(x_{i_{2k}},x_{j_{2k}}) \ud X^N\\
&\quad=\int_{\Pi^{d\,N}} \bar{\rho}_N  \phi(x_{\tau(i_1)},x_{\tau(j_1)})\cdots \phi(x_{\tau(i_{2k})},x_{\tau(j_{2k}})) \ud X^N.
\end{split}\]
Therefore to identify cancellations, we only need to consider one $I_{2k}$ in each of the equivalence classes $\{\tau (I_{2k}),\;\forall\tau \in {\cal S}_N\}$, leading to  
\begin{definition}
A multi-index $I_{2k}$ belongs to the reduced form set ${\cal R}_{N,2k}$ iff  $0<a_1\leq a_2\ldots\leq a_n$ and $a_{n+1}=\cdots=a_N=0$, with $n=m_I + n_I$ in \eqref{defmInI}. 
\end{definition}
Note that for any $I_{2k}$ there exists only one $\tilde I_{2k}\in {\cal R}_{N,2k}$ that belongs to the same class, even though there can be several $\tau$ s.t. $\tau(I_{2k})=\tilde I_{2k}$ (as any repeated index leaves $\tilde I_{2k}$ invariant under the corresponding transposition).
%
%%%%%%%%%%%%%%%%%%%%%%%%%%%%%%%%
\subsubsection{Identifying the ``right'' indices $J_{2k}$}
%%%%%%%%%%%%%%%%%%%%%%%%%%%%%%%%
Remark that by the definition of $m_I$ and $n_I$ in \eqref{defmInI}, if $I_{2k}\in {\cal R}_{N,2k}$ is under its reduced form, one has  
\[\begin{split}
&a_l=1\quad \mbox{for}\ l=1, \cdots,  m_I,\\
&a_l>1\quad \mbox{for}\ l=m_I+1, \cdots,  m_I+n_I,\\ 
&a_l=0\quad \mbox{for}\ l>m_I+n_I.
\end{split}
\] 
Based on this simple structure, we can prove that 
\begin{lemma}
For any $m,\;n$, define as ${\cal J}_{m,n}$ the set of indices $J_{2k}$ with multiplicities $(b_1,\ldots,b_N)$ satisfying
\begin{itemize}
\item $b_l\geq 1$ for any $l=1\dots m$;
\item $b_l\neq 1$ for any $l>m+n$.
\end{itemize}
Then for any $I_{2k}\in {\cal R}_{N,2k}$ and any $J_{2k}\not\in {\cal J}_{m_I,n_I}$, one has that
\[
\int_{\Pi^{d\,N}} \bar{\rho}_N\, \phi(x_{i_1}, x_{j_1})\cdots \phi(x_{i_{2k}}, x_{j_{2k}})\, \ud X^N=0.
\]
\label{identifycancel}
\end{lemma}
This lemma identifies, for each $I_{2k}\in {\cal R}_{N,2k}$, a relevant subset of indices $ {\cal J}_{m_I,n_I}$; in the sense that any multi-index $J_{2k}$ out of this set leads to a vanishing integral and hence can be removed from our summation.  Lemma \ref{identifycancel} is not an equivalence though: There can still be indices  $J_{2k}\in {\cal J}_{m_I,n_I}$ giving a vanishing integral. But the formulation above allows for simpler combinatorics and in particular ${\cal J}_{m_I,n_I}$ only depends in a basic manner on $I_{2k}$ through the two integers $m_I$ and $n_I$.

\begin{proof}[Proof of Lemma \ref{identifycancel}]
Choose any $I_{2k}\in {\cal R}_{N,2k}$, up to a permutation, we may freely  assume that $I_{2k}$ has the following form
\[
I_{2k} = \Big( 1, 2, \cdots, m_I, \underbrace{m_I+1, \cdots, m_I+1}_{a_{m_I+1}}, \cdots, \underbrace{m_I+n_I, \cdots, m_I+ n_I}_{a_{m_I+n_I}}\Big). 
\]
 Choose any $J_{2k}\not\in {\cal J}_{m_I,n_I}$. That means  that there exists $l\leq m_I$ s.t. $b_l=0$ or that there exists $l>m_I+n_I$ s.t. $b_l=1$. Each case corresponds to a different cancellation in the integral.

\medskip

\noindent {\em The case $b_l=0$ for some $l\leq m_I$}. By the definition of the reduced form, $a_l=1$ and therefore the index $l$ appears only once in $I_{2k}$ and never in $J_{2k}$ thus being present exactly once in the product inside the integral. Assume that $i_\nu=l$ for some $\nu$ so 
\[\begin{split}
&\int_{\Pi^{d\,N}} \bar{\rho}_N\, \phi(x_{i_1}, x_{j_1})\cdots \phi(x_{i_{2k}}, x_{j_{2k}})\, \ud X^N\\
&=\int_{\Pi^{d\,(N-1)}} \frac{\bar\rho_N}{\bar\rho(x_{i_\nu})}\, \Big(\int_{\Pi^d} \bar\rho(x_{i_\nu})\,\phi(x_{i_\nu},x_{j_\nu})\,dx_{i_\nu}\Big)\,\Pi_{\nu'\neq \nu} \phi(x_{i_{\nu'}}, x_{j_{\nu'}})\, \Pi_{l' \ne l}\ud x_{l'}.
\end{split}
\] 
Now it is enough to remark that for any $i$ and $j\neq i$, as is the case here since all $j_{\nu'}\neq l$,
\[
\int_{\Pi^d} \bar\rho(x_i)\, \psi(x_{i},x_j)\, \ud x_i =0,
\]
which is exactly the first assumption in \eqref{phicancellation}.
\medskip

\noindent {\em The case $b_l=1$ for some $l> m_I+n_I$}. By definition, this means that $a_l=0$. The index $l$ appears only once in $J_{2k}$ and never in $I_{2k}$. Again it is present exactly once in the product inside the integral. Assume that $j_\nu=l$ for some $\nu$ so
\[\begin{split}
&\int_{\Pi^{d\,N}} \bar{\rho}_N\, \phi(x_{i_1}, x_{j_1})\cdots \phi(x_{i_{2k}}, x_{j_{2k}})\, \ud X^N\\
&=\int_{\Pi^{d\,(N-1)}} \frac{\bar\rho_N}{\bar\rho(x_{j_\nu})}\, \Big(\int_{\Pi^d} \bar\rho(x_{j_\nu})\,\phi(x_{i_\nu},x_{j_\nu})\,dx_{j_\nu}\Big)\,\Pi_{\nu'\neq \nu} \phi(x_{i_{\nu'}}, x_{j_{\nu'}})\,  \Pi_{l' \ne l} \ud x_{l'}.
\end{split}
\]  
The results then follows from the fact that for $i\neq j$
\[
\int_{\Pi^d} \bar\rho(x_j)\, \phi(x_i,x_j) \, \ud x_j =0, 
\]
which is the second equality in \eqref{phicancellation}.
\end{proof}

%%%%%%%%%%%%%%%%
\subsubsection{The cardinality of ${\cal J}_{m,n}$}
%%%%%%%%%%%%%%%%
Our next step is to show that $|{\cal J}_{m,n}|$ is much less than the total number of multi-indices $J_{2k}$, namely $N^{2k}$,
\begin{lemma}
One has that for some universal constant $C$
\[
|{\cal J}_{m,n}|\leq C^{k}\,N^{k-m/2}\,k^{k+m/2}, 
\]
\label{cardJmn}
where $C$ can be chosen as $512\  e $ or roughly $1400$. 
\end{lemma}
\begin{proof}[Proof of Lemma \ref{cardJmn}]
A multi-index $J_{2k}$ belongs to ${\cal J}_{m,n}$ iff $b_l\geq 1$ for $l\leq m$ and $b_l=0,\,2,\,3,...$ for $l > m+n$.  Let us distinguish further between those $l>m+n$ where $b_l=0$ and those for which $b_l\geq 2$. 

Choose first $p=0, 1,  \dots,  \lfloor \frac{2k-m}{2} \rfloor $ and choose then $p$ indices $l_1,\ldots, l_p$ between $m+n+1$ and $N$ which exactly correspond to $b_l\geq 2$. There are $\binom{N-m-n}{p}$ such possibilities. 

Once these $l_1,\ldots,l_p$ have been chosen, the set of possible multiplicities for $J_{2k}\in {\cal J}_{m,n}$ is given by
\[\begin{split}
&{\cal B}_{m,n,p,l_1,\ldots,l_p}=\Big\{(b_1,\ldots,b_N)\,|\;b_1,\ldots, b_m\geq 1,\ b_{l_1},\ldots,b_{l_p}\geq 2,\\
&\quad b_{l}=0\ \mbox{if}\ l>m+n\ \mbox{and}\ l\neq l_1,\ldots,l_p,\ \mbox{and}\ b_1+b_2+\cdots+b_N=2k\}.
\end{split}\]
After the multiplicities are known it is straightforward to obtain the number of $J_{2k}$ in ${\cal J}_{m,n}$, using Lemma \ref{multiplicity}. Decomposing all the possible $J_{2k}$ according to those possibilities, one hence finds
\[
|{\cal J}_{m,n}|=\sum_{p=0}^{\lfloor \frac{2k-m}{2} \rfloor}\sum_{l_1,\ldots ,l_p=m+n+1,\ldots, N}\sum_{(b_1,\ldots,b_N)\in {\cal B}_{m,n,p,l_1,\ldots,l_p}} \frac{(2k)!}{b_1!\cdots b_N!}.
\]
Note that since $b_{l_1},\ldots, b_{l_p}\geq 2$ and $b_1,\ldots,b_m\geq 1$ one has that $m+2p\leq b_1+\cdots+b_N=2k$, leading to the upper bound $p\leq k-m/2$.

Furthermore using the invariance by permutation, one may immediately reduce this expression by assuming that $l_1=m+n+1$, $l_2=m+n+2$... Denoting the partial sums $s_m=b_1+\cdots+b_m$ and $s_n=b_{m+n+1}+\cdot+b_{m+n+p}$, one has
\[\begin{split}
&|{\cal J}_{m,n}|=\sum_{p=0}^{k-m/2} \binom{N-m-n}{p} \sum_{s_m=m}^{2k-2p}\ \sum_{b_1\,\ldots b_m\geq 1,\ b_1+\cdots+b_m=s_m}\\
&\sum_{s_n=2p}^{2k-s_m}\ \sum_{b_{m+n+1},\ldots, b_{m+n+p}\geq 2,\ b_{m+n+1}+\cdots+ b_{m+n+p}= s_n}\\
&\sum_{b_{m+1},\ldots,b_{m+n}\geq 0,\ b_{m+1}+\cdots+b_{m+n}=2k-s_m-s_n} \frac{(2k)!}{b_1!\cdots b_{m+n+p}!}.
\end{split}\]
Using the standard multinomial summation \eqref{multinomialsum}, one can easily calculate the last sum to obtain
\[\begin{split}
&|{\cal J}_{m,n}|=\sum_{p=0}^{k-m/2} \binom{N-m-n}{p} \\
&\ \sum_{s_m=m}^{2k-2p}\ \sum_{b_1\,\ldots b_m\geq 1,\ b_1+\cdots+b_m=s_m}\,
\sum_{s_n=2p}^{2k-s_m}\,\frac{n^{2k-s_m-s_n}}{(2k-s_m-s_n)!}\\
& \sum_{b_{m+n+1},\ldots, b_{m+n+p}\geq 2,\ b_{m+n+1}+\cdots+ b_{m+n+p}= s_n}
 \frac{(2k)!}{b_1!\cdots b_{m}!\,b_{m+n+1 }!\cdots b_{m+n+p}!}.
\end{split}\]
Now bound the sum on $b_1\dots b_m$ by the sum starting at $b_1,\ldots,b_m=0$ and similarly for the sum on $b_{m+n+1}\dots b_{m+n+p}$ to obtain
\[\begin{split}
&|{\cal J}_{m,n}|\leq \sum_{p=0}^{k-m/2} \binom{N-m-n}{p} \\
&\qquad \sum_{s_m=m}^{2k-2p}\,
\sum_{s_n=2p}^{2k-s_m}\,\frac{(2k)!\,n^{2k-s_m-s_n}\,m^{s_m}\,p^{s_n}}{(2k-s_m-s_n)!\,s_m!\,s_n!}.\\
\end{split}\]
We recall the obvious bound $\binom{a}{b}\leq 2^a$ so that
\[\begin{split}
\frac{(2k)!}{(2k-s_m-s_n)!\,s_m!\,s_n!} & =\frac{(2k-s_m)!}{(2k-s_m-s_n)!\,s_n!}\,\frac{(2k)!}{(2k-s_m)!\,s_m!} \\ & =  \binom{2k }{s_m} \binom{ 2k- s_m }{ s_n }   \leq 2^{4k}.\\
\end{split}\]
Furthermore by Lemma \ref{BioCoeBouLem}, $\binom{N-m-n}{p}\leq e^p\,N^{p}\,p^{-p}$. Thus
\[\begin{split}
&|{\cal J}_{m,n}|\leq 2^{4k}\,\sum_{p=0}^{k-m/2 } e^p\,N^p \sum_{s_m=m}^{2k-2p}\,
\sum_{s_n=2p}^{2k-s_m}n^{2k-s_m-s_n}p^{s_n-p}\,m^{s_m}.\\
%&\quad= \sum_{p=0}^{\inf(2k-m-n,k-m/2)} \binom{2k-m-n}{p}\,(p+n)^{-2p}\,\frac{(2k)!}{(2k-2p)!}\\
%&\qquad\qquad\qquad \sum_{s_m=m}^{2k-2p}\,
%\frac{(2k)!\,(p+n)^{2k-2p-s_m}\,m^{s_m}}{(2k-s_m)!\,s_m!},
\end{split}\]
Note that $2k-s_m-s_n\geq 0$ and $s_n-p\geq 0$ and $m,\,n,\,p\leq 2k$ so
\[
n^{2k-s_m-s_n}\,p^{s_n-p}\,m^{s_m}\leq (2k)^{2k-p}.
\]
Therefore finally
\[\begin{split}
|{\cal J}_{m,n}|&\leq 2^{6k}\,e^k\,(2k)^2\,\sum_{p=0}^{k-m/2} N^p\,k^{2k-p}\\
&\leq 2^{6k}\,e^k\,(2k)^2\,  k \,N^{k-m/2}\,k^{k+m/2} \leq  (2^9 e)^k \,N^{k-m/2}\,k^{k+m/2}, 
\end{split}\]
since $N\geq k$, the maximum of $N^p\,k^{2k-p}$ is attained for the maximal value of $p$.
\end{proof}

%%%%%%%%%%%%%
\subsubsection{Conclusion of the proof of the Proposition \ref{PropkSmall}}
%%%%%%%%%%%%%%%
Observe that for a particular choice of $I_{2k}$ and $J_{2k}$
\begin{equation}
\label{boundforIJ}\begin{split}
& \int_{\Pi^{d\,N}} \bar{\rho}_N\,\phi(x_{i_1},x_{j_1})\cdots\phi(x_{i_{2k}},x_{j_{2k}}) \, \ud X^N\\
&\qquad\leq   \int_{\Pi^{d\,N}} \bar{\rho}_N\, \Pi_{\nu =1}^{2k} \sup_z |\phi(x_{i_\nu},z)|\,  \ud X^N\\   
&\qquad\leq  \int_{\Pi^{d\,N}} \bar{\rho}_N\,( \sup_z |\phi(x_{1},z)|)^{a_1}\ldots ( \sup_z |\phi(x_{N},z)|)^{a_N}\ud X^N. \end{split}
\end{equation}
As one readily sees this bound only depends on the multiplicity in $I_{2k}$.

We use the cancellations obtained in Lemma \ref{identifycancel} to deduce from \eqref{boundforIJ}, 
\[\begin{split}
&  \int_{\Pi^{d\,N}} \bar{\rho}_N\, \Big|\frac{1}{N} \sum_{i, j=1}^N \phi(x_i,x_j)\Big|^{2k}\, \ud X^N \\
&\qquad \leq \frac{1}{N^{2k}}\,
\sum_{ \substack{a_1+\cdots + a_N=2k, \\a_1 \geq 0, \cdots, a_N \geq 0. } }\,|{\cal J}_{m_a,n_a}|\,|\{I_{2k}\,|\;\Phi_{N,2k}(I_{2k})=(a_1,\ldots,a_N)\}|\\
&\qquad\quad \cdot \int_{\Pi^{d\,N}} \bar{\rho}_N\, ( \sup_z |\phi(x_{1},z)|)^{a_1}\ldots ( \sup_z |\phi(x_{N},z)|)^{a_N}  \,\ud X^N,
\end{split} 
\]
where we denote $m_a=m_{(a_1,\ldots,a_N)}=|\{l\,|\;a_l=1\}|$, $n_a=n_{(a_1,\ldots,a_N)}=|\{l\,|\;a_l>1\}|$ and we recall that $\Phi_{N,2k}(I_{2k})$ is the multiplicity function associating to each $I_{2k}$ the vector $(a_1,\ldots,a_N)$ of multiplicities.

Remark that
\[\begin{split}
& \int_{\Pi^{d\,N}} \bar{\rho}_N\,( \sup_z |\phi(x_{1},z)|)^{a_1}\ldots ( \sup_z |\phi(x_{N},z)|)^{a_N} \,\ud X^N\\
&\quad\leq  e^{2k} \left(\sup_{p\geq 1} \frac{\|\sup_z |\phi(.,z)|\|_{L^p(\bar\rho\,dx)}}p\right)^{2k}\,a_1!\cdots a_N!,
\end{split}
\]
since $a^a\leq e^a\,a!$.

On the other hand by Lemma \ref{multiplicity}
\[
|\{I_{2k}\,|\;\Phi_{N,2k}(I_{2k})=(a_1,\ldots,a_N)\}|\leq \frac{(2k)!}{a_1!\cdots a_N!},
\]
which implies that
\[\begin{split}
&  \frac{1}{(2k)!}\,\int_{\Pi^{d\,N}} \bar{\rho}_N\, \Big|\frac{1}{N} \sum_{i, j=1}^N \phi(x_i,x_j)\Big|^{2k}\, \ud X^N \\
& \quad \leq\frac{ e^{2k}}{N^{2k}}\,\Big(\sup_{p\geq 1} \frac{\|\sup_z |\phi(.,z)|\|_{L^p(\bar\rho\,dx)}}p \Big)^{2k}\, \sum_{\substack{a_1+\cdots + a_N=2k, \\a_1 \geq 0, \cdots, a_N \geq 0. }}\,|{\cal J}_{m_a,n_a}|.
\end{split}
\]
We apply Lemma \ref{cardJmn}
\[\begin{split}
&  \frac{1}{(2k)!}\,\int_{\Pi^{d\,N}} \bar{\rho}_N\, \Big|\frac{1}{N} \sum_{i, j=1}^N \phi(x_i,x_j)\Big|^{2k}\, \ud X^N  \leq \frac{ e^{2k}}{N^{2k}}\\
&\qquad\quad \cdot \left(\sup_{p\geq 1} \frac{\|\sup_z |\phi(.,z)|\|_{L^p(\bar\rho\,dx)}}p\right)^{2k}
\,\sum_{\substack{a_1+\cdots + a_N=2k, \\a_1 \geq 0, \cdots, a_N \geq 0. }}\,C^k\,N^{k-m_a/2}\,k^{k+m_a/2}.
\end{split}
\]
%Note that since $N>2k$, we cannot simply use Lemma \ref{Diophantine} to conclude.
 Consider any $(a_1,\ldots,a_N)$ with exactly $p$ coefficients $a_l\geq 1$. Up to $\binom{N}{p}$ permutations, we can actually assume that $a_1 ,  \cdots,  a_{p}\geq 1$. All the other $a_l$ are $0$. Since we have $m_a+n_a=p$ and $m_a+2\,n_a\leq 2k$ then $m_a\geq 2\,(p-k)$. As $N\geq k$ then
\[
N^{k-m_a/2}\,k^{k+m_a/2}\leq N^{k-(p-k)_+}\,k^{k+(p-k)_+}.
\]
Hence
\[\begin{split}
&\sum_{a_1+\cdots a_N=2k} N^{k-m_a/2}\,k^{k+m_a/2}\\
&\qquad=\sum_{p=1}^{2k} \binom{N}{p}\, \sum_{a_{1},\ldots,a_{p}\geq 1,\;a_{1}+\cdots+a_{p}= 2k} N^{k-(p-k)_+}\,k^{k+(p-k)_+}\\
&\qquad\leq \sum_{p=1}^{2k} \binom{N}{p}\,\binom{2k-1}{p-1}\,N^{k-(p-k)_+}\,k^{k+(p-k)_+},
\end{split}
\]
by Lemma \ref{Diophantine}. 
Since  $\binom{2k-1}{p-1}\leq 2^{2k}$ and for $p\leq k$,  $\binom{N}{p}$ is maximum when  $p=k$, 
\[
\sum_{p=1}^{k} \binom{N}{p}\,\binom{2k-1}{p-1}\,N^{k-(p-k)_+}\,k^{k+(p-k)_+}\leq 2^{2k}\,k\,\binom{N}{k}\,N^k\,k^k\leq (8 e )^k\, N^{2k},
\]
by  Lemma \ref{BioCoeBouLem}. Still by Lemma \ref{BioCoeBouLem} for $p >  k$,
\[
\binom{N}{p}\,N^{k-(p-k)_+}\,k^{k+(p-k)_+}\leq e^p\,N^p\,p^{-p}\, N^{2k-p}\,k^{p} \leq e^p\,N^{2k}.
\]
Hence again
\[
\sum_{p=k+1 }^{2k} \binom{N}{p}\,\binom{2k-1}{p-1}\,N^{k-(p-k)_+}\,k^{k+(p-k)_+} \leq k 2^{2k } e^{2k} N^{2k} \leq   \frac{1}{2} (8e^2)^k \, N^{2k}.
\]
Finally,
\[\begin{split}
& \frac{1}{(2k)!}\,\int_{\Pi^{d\,N}} \bar{\rho}_N\, \Big|\frac{1}{N} \sum_{i, j=1}^N \phi(x_i,x_j)\Big|^{2k}\, \ud X^N \\
&\qquad\quad \leq { (8\,e^4 C )^{k}}\,\left(\sup_{p\geq 1} \frac{\|\sup_z |\phi(.,z)|\|_{L^p(\bar\rho\,dx)}}p\right)^{2k}\\
&\qquad\quad \leq { 800^{2k}}\,\left(\sup_{p\geq 1} \frac{\|\sup_z |\phi(.,z)|\|_{L^p(\bar\rho\,dx)}}p\right)^{2k},\\
\end{split}
\]
concluding the proof of  Proposition \ref{PropkSmall}. 

\smallskip 

%%%%%%%%%%%%%%%%%%%%%%%%%%%%%%
%%%%%%%%%%%%%%%%%%%%%%%%%%%
\section{Conclusion}
%%%%%%%%%%%%%%%%%%%%%%%%%%%%
%%%%%%%%%%%%%%%%%%%%%%%%%%%5%%
We have presented a new approach to the mean-field limit based on relative entropies at the level of the Liouville equations. While the role of the entropy had long been recognized (for example in \cite{FHM} and later in \cite{GQ,LY}), our method allows to quantitatively estimate the convergence of each marginal at the optimal rate $N^{-1/2}$. The key for the technical argument is a large deviation bound
\begin{equation}
\sup_N \int_{\Pi^{d\,N}} e^{\frac{1}{N}\sum_{i,j=1}^N \phi(x_i,x_j)}\,\bar\rho_N\, \ud X^N<0,
\label{LDB}
\end{equation}
for a modified potential $\phi$ that is not the potential of the dynamics and that is not continuous.

While this allows us to treat a large class of interaction kernels, there are many questions left open by the present work that we mention briefly below.

\begin{itemize}
\item {\em Going from the case with periodic boundary 
conditions  to non-compact settings or boundary condition}. Choosing to study the dynamics in the torus $\Pi^d$, as we did here, is convenient but somewhat artificial.
  The main difficulty to extending our theory to more realistic domains are the assumption $\inf \bar\rho>0$ in Theorem \ref{maintheorem} and $\log \bar\rho\in BMO$ (more precisely when $\sigma_N \equiv \sigma$) in Theorem \ref{thmvanishingviscosity}.

  $\inf \bar\rho>0$ is simply not compatible with any unbounded domain (and keeping finite mass), while  $\log \bar\rho\in BMO$ would limit the application to slowly decaying (polynomially) densities.
A possible solution would involve introducing appropriate weights in the relative entropy. 

The case of smooth bounded domains $\Omega\subset \R^d$ depends much on the precise boundary condition that is imposed; Reflective and incoming boundary conditions for example are generally compatible with our relative entropy method.  But we may still sometimes have difficulties with the previous assumption $\inf\bar\rho>0$, if for instance the incoming density vanishes.
\item {\em Extending the large deviation estimate}. It is not clear to us at this point what would be an optimal assumption on $\phi$ for \eqref{LDB} to hold. An important issue is how important a lower bound on $\phi$ is and whether we need less on $\phi_-$ (the negative part) than on $\phi_+$. For the classical large deviation result for example if $\phi$ is continuous, then the smallness of $\phi_+$ is enough. Clearly the negative part of $\phi$ only helps in \eqref{LDB} but our combinatorial analysis does not easily allow us to differentiate between $\phi_+$ and $\phi_-$.
\item {\em Gradient flow dynamics}. Theorems \ref{maintheorem} and \ref{thmvanishingviscosity} do not perform well for gradient flows: This is due to the assumptions $\udiv K$ in $W^{-1,\infty}$ or $L^\infty$. If $K=-\nabla V$ then this almost imposes $K\in L^\infty$ or $W^{1,\infty}$ in that case. If the dynamics is attractive then some difficulties are expected.  The repulsive case is however connected to the previous remark as $\phi$ in \eqref{LDB} includes a $\udiv K$ term: If we did not need to impose conditions on $\phi_-$ then we would not need to impose conditions either on $\udiv K$. We would then be able to derive the Keller-Segel equations, {\em i.e.} the Poisson case. 
\item  {\em Better use of the energy of the dynamics}. We are not employing in this article the energy or other dissipated or invariant quantities of the system, which could obviously be useful.

In particular, \cite{DuerSer,Serfaty,SerfatyCou} recently introduced a relative entropy method at the the level of the empirical measure based on the energy of the system. This allows to obtain quantitative estimates, in particular for deterministic settings, with quite singular interactions of  Riesz potential form.
  
We should also mention here the techniques developed by \cite{SandierSerfaty,Serfaty11} for non-convex setting.  For the case of hydrodynamics for Ginzburg-Landau spin systems, we also refer to \cite{Fathi} for a quantitative convergence results of relative entropy of particles systems towards  its counterpart in the scaling limit. 
\item {\em Is it possible to make fluctuations explicit?} This would for example mean making explicit  the $O(N)$ term in our relative entropy estimates.  
In the smooth case ($K \in W^{1, \infty}$), large deviations from the limiting PDE were notably established in \cite{DawsontG}. In the framework developed here, this would likely require being more precise than the bound \eqref{LDB} and so impose more regularity on $\phi$ and then $K$. There are several applications of entropy bounds and super-exponential estimates in scaling limits for instance \cite{GPV}. 

There are nevertheless results on fluctuations for singular kernels, in particular \cite{FontbonaLDP} where a large deviation estimate is obtained from the limiting law for systems introduced in \cite{CepaL}, i.e.  stochastic interacting particle systems in 1D with $K(x)= 1/x$. 
\item {\em Other settings: Collisional models, quantum systems...}  The notion of propagation of chaos is of course critical in many other frameworks. First come to mind the collisional regime, with Boltzmann or Landau equations. While Theorem \ref{maintheorem} allows for kernels $K$ leading to collision dynamics, this requires a fully non-degenerate diffusion and a different scaling.  

The stochastic particles approximation of the Boltzmann equation had been studied in \cite{GrahamM}. The propagation of chaos for the Landau equation was obtained in \cite{Carrapatoso,FournierHLandau}, with the so-called Nanbu particles investigated in \cite{FournierMischler}. 

We also refer to the review \cite{Sai} for a discussion of the role of exchangeability and entropy for systems of particles in a larger context; to \cite{GalSaiTex} more specifically for a thorough discussion derivation of the Boltzmann equation from deterministic (Newton) particles dynamics and to the recent \cite{BodGalSai} for the derivation of Brownian dynamics from hard spheres in realistic time scales.

The discussion of the mean-field limit for many particles quantum systems would of course deserve a review of its own. The Von Neumann entropy is the direct equivalent of the relative entropy that we are using, but it is not clear to us how the strategy of the present paper could be extended to that setting.

We do note that quantitative estimates and rates of convergence are well known for quantum systems with singular interactions such as in \cite{Gol,KnowlesPickl,PaulPS}.
\end{itemize}

%%%%%%%%%%%%%%%%%%%%%%%%%%%%%%%%%%%%%%%%%%%%%%%%%%%%%%%%%%%%%%%%%%%%%%%%%%%%%

%\begin{acknowledgements}
%If you'd like to thank anyone, place your comments here
%and remove the percent signs.
%\end{acknowledgements}

% BibTeX users please use one of
%\bibliographystyle{spbasic}      % basic style, author-year citations
%\bibliographystyle{spmpsci}      % mathematics and physical sciences
%\bibliographystyle{spphys}       % APS-like style for physics
%\bibliography{}   % name your BibTeX data base

% Non-BibTeX users please use

\end{document}